\pdfoutput=1

\documentclass[11pt,a4paper,dvipsnames]{amsart}
\usepackage[T1]{fontenc}
\usepackage[utf8]{inputenc}
\usepackage[british]{babel}
\usepackage[a4paper,margin=1in]{geometry}
\usepackage{microtype,lmodern,needspace}
\usepackage{mathpazo}
\usepackage{euler}
\AtBeginDocument{%
 \DeclareFontShape{T1}{ppl}{m}{up}{<->ssub*ppl/m/n}{}%
 \DeclareFontShape{T1}{ppl}{b}{up}{<->ssub*ppl/b/n}{}%
}
\usepackage{amsmath,amssymb,amsthm,mathtools,mathrsfs}
\usepackage[shortlabels]{enumitem}
\usepackage{xcolor}
\definecolor{RonReviewColor}{RGB}{160,62,25}

\definecolor{RevisionColor}{RGB}{160,62,25}

\usepackage{array,booktabs,tabularx}
\usepackage{tikz-cd}
\usetikzlibrary{arrows.meta,positioning,calc}
\usepackage{placeins}
\usepackage[colorlinks=true,linkcolor=blue!60!black,citecolor=blue!60!black,urlcolor=blue!60!black]{hyperref}
\usepackage[capitalize,nameinlink]{cleveref}

\numberwithin{equation}{section}
\theoremstyle{plain}
\newtheorem{theorem}{Theorem}[section]
\newtheorem{proposition}[theorem]{Proposition}
\newtheorem{lemma}[theorem]{Lemma}
\newtheorem{corollary}[theorem]{Corollary}
\theoremstyle{definition}
\newtheorem{definition}[theorem]{Definition}
\theoremstyle{remark}
\newtheorem{remark}[theorem]{Remark}

\newcommand{\C}{\mathbb C}
\newcommand{\cA}{\mathcal A}
\newcommand{\cE}{\mathcal E}
\newcommand{\cF}{\mathcal F}
\newcommand{\cG}{\mathcal G}
\newcommand{\cI}{\mathcal I}
\newcommand{\cJ}{\mathcal J}

\newcommand{\cL}{\mathcal L}
\newcommand{\cM}{\mathcal M}

\newcommand{\cO}{\mathcal O}
\newcommand{\cP}{\mathcal P}
\newcommand{\cQ}{\mathcal Q}

\newcommand{\cV}{\mathcal V}
\newcommand{\M}{\mathfrak M}
\newcommand{\Mbar}{\overline{\mathfrak M}}
\newcommand{\Dsep}{\mathfrak D_{h,1}^{+,\varepsilon}}
\newcommand{\Hom}{\operatorname{Hom}}
\newcommand{\gr}{\operatorname{gr}}
\newcommand{\Id}{\operatorname{Id}}

\title[Non-projectedness of supermoduli spaces]{Supermoduli spaces are non-projected in every genus $g\geq 3$}
\author{Mauricio Corr\^ea}
\address{Dipartimento di Matematica, Universit\`a degli Studi di Bari
Aldo Moro, Bari, Italy}
\email{mauricio.barros@uniba.it}

\author{Ron Donagi}
\address{Departments of Mathematics and Physics, University of Pennsylvania,
Philadelphia, PA 19104, USA}
\email{donagi@math.upenn.edu}

\author{Simone Noja}
\address{Dipartimento di Matematica, Universit\`a degli Studi di Bari
Aldo Moro, Bari, Italy}
\email{simone.noja@uniba.it}
\date{}


\hypersetup{pdftitle={Supermoduli spaces are non-projected in every genus at least three},
 pdfauthor={Mauricio Corr\^ea, Ron Donagi, Simone Noja}}
\newcommand{\eps}{\varepsilon}

\newcommand{\Ext}{\operatorname{Ext}}
\newcommand{\Spec}{\operatorname{Spec}}
\newcommand{\Spf}{\operatorname{Spf}}
\newcommand{\red}{\mathrm{red}}
\newcommand{\an}{\mathrm{an}}
\newcommand{\cB}{\mathcal B}

\begin{document}
\begin{abstract}
We prove that, for every genus $g\geq 3$, the primary obstruction
of each parity component of the moduli superstack of smooth
unpunctured super Riemann surfaces is nonzero, both algebraically
and holomorphically. Consequently, both components are
non-projected and, in particular, non-split: neither admits
an algebraic or holomorphic retraction onto its reduced
moduli stack of spin curves.
\end{abstract}

\maketitle
\tableofcontents

\section{Introduction}\label{sec:introduction}

Genus three is the threshold at which smooth unpunctured supermoduli
spaces cease to admit a holomorphic projection to the moduli of
spin curves.  We prove this in both spin parities and establish the
same obstruction in every higher genus by a uniform degeneration
argument.  The result determines a basic feature of these moduli
spaces: although their local deformation theory is unobstructed,
their global geometry contains nontrivial extension data over the
classical spin moduli stack.  These data obstruct a retraction
already at the first nontrivial infinitesimal order.  The theorem
thus gives a uniform description of the failure of the split
model throughout the stable range above genus two.

Let $\M_g^\eps$ be the moduli stack of smooth unpunctured super
Riemann surfaces of genus $g\geq2$, over $\C$, with spin parity
$\eps\in\{+,-\}$.  Its dimension is $(3g-3\mid2g-2)$, and
its reduction $\mathcal S_g^\eps$ parametrises smooth curves
$C$ with a theta characteristic $L$, where $L^2\simeq K_C$
and the parity of $h^0(C,L)$ is prescribed.  A projection is a
retraction of the canonical closed immersion
$\mathcal S_g^\eps\hookrightarrow\M_g^\eps$.  Splitness is
stronger: it identifies the structure sheaf with the exterior
algebra of the odd cotangent bundle on the reduction.  Either
property requires the vanishing of the primary obstruction
$$
 \omega_2(\M_g^\eps)\in
 H^1\!\left(\mathcal S_g^\eps,
   \Hom\!\left(\bigwedge^2T_-\M_g^\eps,T_+\M_g^\eps\right)
 \right).
$$
Here $T_+$ and $T_-$ denote the even and odd parts of the tangent bundle
restricted to the reduction.  Our main result is the following.

\begin{theorem}[Main theorem]\label{thm:intro-main}
For every $g\geq3$ and both $\eps=+,-$,
$$
 \omega_2^{\rm alg}(\M_g^\eps)\ne0,
 \qquad
 \omega_2^{\rm hol}((\M_g^\eps)^{\an})\ne0.
$$
Consequently both components of smooth unpunctured supermoduli
space are non-projected, and hence non-split, in the algebraic and
holomorphic categories.
\end{theorem}

This gives the first proof of non-projectedness in genus three.
In joint work with Witten, one of
us proved non-projectedness of both components for $g\geq5$
and suggested that the result should hold for $g\geq3$
\cite{DW}.  Two of us subsequently proved the even genus-four
case \cite{DonagiNoja}.  Theorem~\ref{thm:intro-main} settles the genus-three case and the odd genus-four case, and provides a single degeneration argument for both parities in every genus $g\geq3$.

The threshold is sharp.
Both smooth genus-two components are split, for a cohomological reason.
Their coarse spin moduli spaces are finite over the affine coarse
space $M_2$; see \cite[Section~5.1]{DW} for the affineness of
$M_2$.  In characteristic zero, pushforward to the coarse space
is exact for these Deligne--Mumford stacks \cite{AOV}.  Their
higher coherent cohomology therefore vanishes.  In particular,
the primary class is zero, and odd dimension two leaves no further
splitting obstruction.  On the even component, the super period
matrix gives the projection used in the genus-two calculations
of D'Hoker--Phong \cite{DHP2}; its relation to the primary
obstruction is developed in \cite{Bettadapura}.
The condition $H^0(C,L)=0$ required for this super-period
construction holds throughout the smooth even component: every
even theta characteristic in genus two is ineffective.

Viewed on the ordinary moduli stack of spin curves, the theorem
is a uniform nonvanishing statement for a canonical coherent
extension class. Its deformation-theoretic significance is
that local unobstructedness does not ensure a global separation
of the even and odd moduli. At a spin curve $(C,L)$, the
corresponding tangent spaces are $H^1(C,T_C)$ and
$H^1(C,L^{-1})$, and the local deformation space is smooth.
For the filtration by powers of the odd ideal, the associated
graded structure sheaf is canonically the exterior algebra
of the odd cotangent bundle on $\mathcal S_g^\eps$.
Nevertheless, for every $g\geq3$, the quadratic terms in the
odd variables in the even transition functions cannot be
removed consistently on overlaps. The failure of a global
retraction is thus already visible in the first nonlinear
infinitesimal neighbourhood of the reduced moduli stack.

There is an entirely classical infinitesimal formulation of this
statement.  Write $X=\mathcal S_g^\eps$ and let $\cJ$ be
the odd ideal of $\M_g^\eps$.  The even parts of the local
algebras $\cO_{\M_g^\eps}/\cJ^3$ define an ordinary
square-zero thickening of $X$, with structure sheaf $\cA_2$
fitting into an exact sequence of sheaves of $\C$-algebras
$$
 0\longrightarrow\bigwedge^2(T_-\M_g^\eps)^\vee
 \longrightarrow\cA_2\longrightarrow\cO_X\longrightarrow0.
$$
The differences between local splittings are derivations into
the square-zero ideal.  Their class is precisely $\omega_2$
under the identification
$$
 \Ext^1_X\!\left(\Omega_X^1,
                 \bigwedge^2(T_-\M_g^\eps)^\vee\right)
 \simeq
 H^1\!\left(X,
      \Hom(\bigwedge^2T_-\M_g^\eps,T_+\M_g^\eps)\right).
$$
Thus the theorem proves the nontriviality of a canonical
infinitesimal extension of each ordinary spin moduli stack in
the stated range.  Equivalently, it proves that the vector-bundle
extension representing the same class,
$$
 0\longrightarrow T_+\M_g^\eps
 \longrightarrow\mathcal E_g^\eps
 \longrightarrow\bigwedge^2T_-\M_g^\eps
 \longrightarrow0,
$$
has no global splitting.  Its interpretation as a component of
the super Atiyah class, and its description through line bundles
on the fibre square of the universal curve, were established in
\cite{DWAtiyah}.  The result is therefore a nonvanishing theorem
for a geometrically defined coherent extension class on classical
moduli.  Stable degenerations detect this class, and the geometry
of the boundary controls its persistence on the open moduli
stack of smooth curves.

The common construction starts with a curve of genus $h=g-1$
equipped with an ineffective even theta characteristic, and
attaches an elliptic component at a Neveu--Schwarz node.  Such a
core is available for every $h\geq2$.  The pointed test family
constructed in \cite{DW} supplies a nonzero central comparison class.  The
spin structure of the elliptic component determines the parity of
the smoothed surface.  In the even case one includes its odd
deformation direction in the test source; in the odd case a
different elliptic character makes this enlargement unnecessary.
These are the essential parity distinctions.  The normal geometry
and the method for detecting the obstruction are shared.

To pass from the {boundary test map} to the interior, we deform its source
together with the morphism into stable supermoduli, preserving
the full normal line.  A transverse extension argument identifies
the successive obstruction groups.  Finite elliptic symmetries
eliminate the possible invariant obstructions at the first two
orders, and Serre duality eliminates the higher orders.  The
result is a transverse formal family.  Its comparison class
remains nonzero after the smoothing parameter is inverted: a
lowest Laurent coefficient argument rules out a primitive with
negative powers, and the constant coefficient is the original
nonzero boundary class.  {In the even component, the exceptional polar term of
order $-1$ is ruled out by the nonsplit extension defining the
odd cotangent bundle of the pointed test family.}
An algebraic splitting on smooth
supermoduli would restrict to just such a Laurent primitive,
since algebraic sections on a punctured neighbourhood have finite
pole order.  This proves nonvanishing on the smooth locus. For related non-projectedness results concerning stable
compactifications, see \cite{CorreaNojaCompactified}.

The passage from algebraic to holomorphic nonvanishing leads
to a separate question in ordinary algebraic geometry:
when can an algebraic extension of vector bundles become
split after analytification? Since the moduli spaces under
consideration are not proper, proper GAGA does not answer
this question directly. The following theorem gives a
sufficient geometric condition for preserving extension
classes on a nonproper variety.

\begin{theorem}[Comparison in codimension two]\label{thm:intro-gaga}
Let $U$ be a normal complex algebraic variety admitting a normal
projective compactification $X$ such that
$\operatorname{codim}_X(X\setminus U)\geq2$.
For every algebraic vector bundle $F$ on $U$, the natural map
$$
 H^1(U,F)\longrightarrow H^1(U^{\an},F^{\an})
$$
is injective.
\end{theorem}

Applying the theorem to $\Hom(E_2,E_1)$, for algebraic
vector bundles $E_1,E_2$ on $U$, gives an injection
$$
 \Ext^1_U(E_2,E_1)
 \longrightarrow
 \Ext^1_{U^{\an}}(E_2^{\an},E_1^{\an}).
$$
Thus an algebraic extension that admits a holomorphic splitting
already admits an algebraic splitting. This comparison applies
to arbitrary vector bundles under the stated compactification
hypothesis. Its proof extends the middle bundle of an
extension as a coherent reflexive sheaf on the compactification.
Analytic Hartogs extension and degree-zero proper GAGA then
algebraize any holomorphic section defining a splitting.

The required compactification geometry is supplied by the
Torelli morphism. For every $g\geq3$, both spin parities,
and full level $n\geq3$, we construct a normal projective
compactification of each connected component of the level
cover of the rigidified spin stack, with complement of
codimension at least two
(Proposition~\ref{prop:small-compactification}).
Here rigidification removes the central scalar automorphism
of the spin line. {We normalise the closure of the coarse
Jacobian locus in the Satake compactification, taking the full
function field of the chosen spin-level component.  This retains
the spin and level data while using only the boundary geometry
of the coarse Jacobian locus.}

Descent through the scalar spin gerbe and finite \'etale trace
then give a comparison theorem on the ordinary spin stack
itself: analytification is injective on $H^1$ for every
algebraic vector bundle
(Proposition~\ref{prop:spin-comparison}).  In particular, an algebraic
extension of vector bundles on the spin stack splits whenever
its analytification splits.  Applying the comparison to the
primary-obstruction bundle, and identifying the algebraic and
holomorphic primary classes under analytification, proves
holomorphic nonvanishing.

The theorem also has consequences for superstring perturbation
theory, where amplitudes are defined by integration over
supermoduli \cite{WittenPerturbation}. In the genus-two even sector, the construction of D'Hoker--Phong uses the super period matrix to integrate
out the odd moduli \cite{DHP2,DHP6}.
For $g\geq3$, our result rules out a global holomorphic
projection onto spin moduli as the basis for an analogous
construction. This obstruction does not, by itself, decide
the validity of the modular-form candidates for chiral
measures proposed by Cacciatori--Dalla Piazza--van Geemen
and Grushevsky \cite{CDG3,CDG4,Grushevsky};
relating these candidates to integration on supermoduli
remains a separate question.

\medskip

The paper is organised as follows.
Section~\ref{sec:transverse} develops primary-class functoriality
and proves the transverse extension theorem, including its
equivariant formal consequence.
Section~\ref{sec:boundary} constructs the common core and the
two {elliptic-tail constructions}, and identifies their central class.
Section~\ref{sec:vanishing} proves the obstruction vanishings
and constructs the formal smoothing.
Section~\ref{sec:Laurent} establishes survival under Laurent
localisation, and Section~\ref{sec:algebraic} deduces algebraic
non-projectedness on the smooth locus.
Section~\ref{sec:gaga} proves
Theorem~\ref{thm:intro-gaga}.
Finally, Section~\ref{sec:compactification} constructs the
spin-level compactifications, proves the required descent and
completes the proof of Theorem~\ref{thm:intro-main}.
Figure~\ref{fig:proof-flow} records the principal dependencies.
The upper chain produces the nonzero algebraic class; the lower
chain proves that it remains nonzero after analytification.
They meet in the final holomorphic nonvanishing theorem.

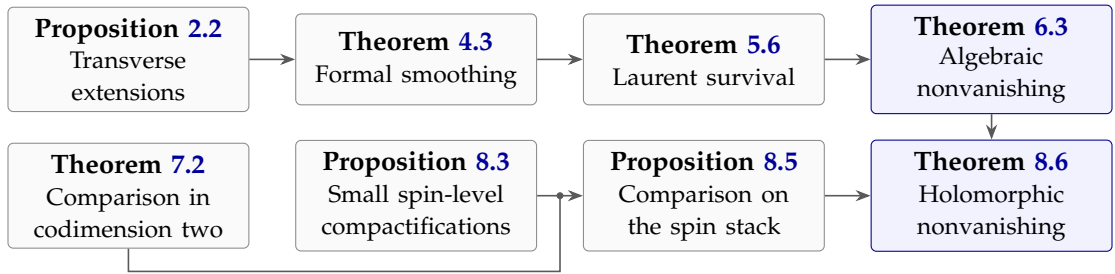
\begin{figure}[!htbp]
\centering
\begin{tikzpicture}[
 result/.style={draw=black!45,fill=black!2,rounded corners=2pt,
   text width=29mm,minimum height=12mm,inner sep=4pt,
   align=center,font=\small},
 conclusion/.style={result,draw=blue!55!black,fill=blue!5},
 dependency/.style={-{Stealth[length=2mm]},semithick,draw=black!65},
 connection/.style={semithick,draw=black!65}]
\node[result] (transverse) at (0,0)
 {\textbf{Proposition~\ref{prop:transverse-extension}}\\
  {\footnotesize Transverse extensions}};
\node[result] (formal) at (3.8,0)
 {\textbf{Theorem~\ref{thm:formal-map}}\\
  {\footnotesize Formal smoothing}};
\node[result] (laurent) at (7.6,0)
 {\textbf{Theorem~\ref{thm:Laurent}}\\
  {\footnotesize Laurent survival}};
\node[conclusion] (algebraic) at (11.4,0)
 {\textbf{Theorem~\ref{thm:algebraic}}\\
  {\footnotesize Algebraic\\nonvanishing}};
\draw[dependency] (transverse) -- (formal);
\draw[dependency] (formal) -- (laurent);
\draw[dependency] (laurent) -- (algebraic);

\node[result] (gaga) at (0,-1.8)
 {\textbf{Theorem~\ref{thm:GAGA}}\\
  {\footnotesize Comparison in codimension two}};
\node[result] (compactification) at (3.8,-1.8)
 {\textbf{Proposition~\ref{prop:small-compactification}}\\
  {\footnotesize Small spin-level compactifications}};
\node[result] (spin) at (7.6,-1.8)
 {\textbf{Proposition~\ref{prop:spin-comparison}}\\
  {\footnotesize Comparison on the spin stack}};
\node[conclusion] (holomorphic) at (11.4,-1.8)
 {\textbf{Theorem~\ref{thm:main}}\\
  {\footnotesize Holomorphic nonvanishing}};
\coordinate (joint) at (5.7,-1.8);
\draw[connection] (gaga.south) -- (0,-2.8) -- (5.7,-2.8) -- (joint);
\draw[connection] (compactification.east) -- (joint);
\fill[black!65] (joint) circle (1.3pt);
\draw[dependency] (joint) -- (spin.west);
\draw[dependency] (spin) -- (holomorphic);
\draw[dependency] (algebraic.south) -- (holomorphic.north);
\end{tikzpicture}
\caption{The two chains of the proof.  The upper chain also uses
the central construction of Section~\ref{sec:boundary} and the
cohomological vanishings in Sections~\ref{sec:vanishing}
and~\ref{sec:Laurent}; the lower chain includes gerbe and
finite-level descent.}
\label{fig:proof-flow}
\end{figure}
\FloatBarrier

\section*{Acknowledgements}

M.C. is partially supported by the Universit\`a degli Studi di Bari and by the PRIN 2022MWPMAB -- \emph{Interactions between Geometric Structures and Function Theories}.

The research of R.D. was partially supported by NSF grant DMS--2401422,
\emph{Geometry and Strings}, and by NSF grant DMS--2244978,
\emph{FRG: New Birational Invariants}.

S.N. acknowledges support by the Deutsche Forschungsgemeinschaft (DFG, German Research Foundation) under project number 517493862 (\emph{Homological Algebra of Supersymmetry: Locality, Unitary, Duality}).
S.N. thanks Sergio Cacciatori, Johannes Walcher,
and Heidelberg University, where he began working on
the problem of non-projectedness of supermoduli spaces
in 2024. 

An adversarial critique by ChatGPT helped eliminate some inaccuracies and allowed simplifications of some of the proofs.

\section{Primary classes and transverse extensions}
\label{sec:transverse}

We begin with the two general tools used in the construction:
compatibility of primary classes under a morphism, and extension of
an immersed test map transversely to a divisor.  The first identifies
the class to be detected; the second will {extend the test map away from the boundary}.

\subsection{Conventions and functoriality}\label{subsec:primary-conventions}

We first specify the coefficient bundle, namely the sheaf in
whose first cohomology the primary obstruction takes values.
We then establish the compatibility identity that compares a
known source class with the pullback of the target class.  This
identity will be used first on the boundary and then on its formal
neighbourhood.

All superstacks and morphisms are algebraic over $\C$, unless the
analytic category is explicitly mentioned.  For algebraic
Deligne--Mumford superstacks we use the conventions of
\cite[Section~2.1]{FKP}; the corresponding complex-analytic
foundations are developed in \cite[Section~3]{CodogniViviani}.
We use $\mathbb Z/2\mathbb Z$-gradings, with $\bar0$
denoting even parity and $\bar1$ odd parity.  For a graded
sheaf $\cE$, the subsheaf $\cE_{\bar0}$ consists of
its even sections.  For a homogeneous function or section $a$,
the notation $|a|$ denotes its parity (not an absolute value).
The notation $|Z|$, when applied to a superscheme, instead
denotes its underlying topological space.

Let $\mathfrak X$ be a smooth Deligne--Mumford superstack,
let $X=\mathfrak X_{\red}$, and denote the canonical closed
immersion by $\iota_X:X\hookrightarrow\mathfrak X$.
We say that $\mathfrak X$ is \emph{projected} if it admits
a retraction to its reduction: a morphism
$$
 \pi:\mathfrak X\longrightarrow X,
 \qquad \pi\circ\iota_X\simeq\Id_X.
$$
For a stack the displayed identification is a specified
$2$-isomorphism.  This definition is taken in the algebraic
or analytic category under consideration.

Let $\cJ_{\mathfrak X}$ be the ideal generated by odd
sections.  We use the tangent-bundle conventions
$$
 T_+\mathfrak X:=T_X,
 \qquad
 T_-\mathfrak X:=
 \left(\cJ_{\mathfrak X}/\cJ_{\mathfrak X}^2\right)^\vee.
$$
The primary obstruction to projectedness is represented on an \'etale
algebraic atlas by the terms quadratic in the odd variables in the
even transition functions.  Thus
$$
 \omega_2^{\rm alg}(\mathfrak X)
 \in
 H^1\!\left(
  X,
  \Hom\!\left(\bigwedge^2T_-\mathfrak X,T_+\mathfrak X\right)
 \right).
$$
Here and below cohomology on a stack is taken on its
lisse--\'{e}tale site.  The primary class vanishes whenever
$\mathfrak X$ is projected; see \cite{Green,DWAtiyah}.
For a recent treatment of splitting obstructions through formal
moduli problems, see also \cite{CorreaNojaFormal}.

For $g\geq2$ and $\eps\in\{+,-\}$, we write
$\M_g^\eps$ for the smooth Deligne--Mumford superstack
parametrising families of smooth \emph{unpunctured} super Riemann
surfaces of genus $g$ with even or odd spin parity, respectively.
Its reduction is the spin stack $\mathcal S_g^\eps$, whose
objects are smooth curves with a theta characteristic
$(C,L,L^{\otimes2}\simeq K_C)$ of the chosen parity.
For super Riemann surfaces and their deformation theory, see
\cite[Section~2]{WittenNotes}; for the algebraic moduli superstack
and its stable compactification, see \cite[Theorem~A]{FKP}.

We shall repeatedly use the arbitrary-morphism compatibility of the
primary class.  If $f:\mathfrak Y\to\mathfrak X$ is a morphism of
smooth superstacks, {write $Y=\mathfrak Y_{\red}$ and
$X=\mathfrak X_{\red}$, and put}
\begin{equation} \label{mixed}
 \cG_f:=
 \Hom\!\left(
  \bigwedge^2T_-\mathfrak Y,
  f_{\rm red}^*T_+\mathfrak X
 \right).
\end{equation}
The differentials of $f$ induce
\begin{align*}
 j_f:
 \Hom\!\left(\bigwedge^2T_-\mathfrak Y,T_+\mathfrak Y\right)
 &\longrightarrow\cG_f,
 &\alpha&\longmapsto df_+\circ\alpha,\\
 \iota_f:
 f_{\rm red}^*\Hom\!\left(
  \bigwedge^2T_-\mathfrak X,T_+\mathfrak X
 \right)
&\longrightarrow\cG_f,
&\beta&\longmapsto\beta\circ\bigwedge^2df_-.
\end{align*}
The compatibility identity is
\begin{equation}\label{eq:primary-naturality}
 H^1(j_f)\bigl(\omega_2(\mathfrak Y)\bigr)
 =
 H^1(\iota_f)\bigl(f_{\rm red}^*\omega_2(\mathfrak X)\bigr)
 \quad\text{in }H^1(Y,\cG_f).
\end{equation}
This identity follows by applying the chain rule to the terms quadratic in the odd variables; 
cf.\ \cite[Corollary~2.12]{DW}.

We emphasise that identity \eqref{eq:primary-naturality} holds for arbitrary morphisms $f$.  Choose compatible
\'etale charts and local split coordinates for source and target.
Let $c_{\mathfrak Y}$ and $c_{\mathfrak X}$ be the
\v{C}ech $1$-cocycles of quadratic-odd terms in the even
transition functions representing $\omega_2(\mathfrak Y)$
and $\omega_2(\mathfrak X)$, respectively.  In each source
chart, write the even part of $f$, modulo terms of odd degree
four, as $f_{\red}+b_{f,i}$, where $b_{f,i}$ is
quadratic in the source odd coordinates.  These local terms form
a \v{C}ech $0$-cochain $b_f=(b_{f,i})$ with values in
$\cG_f$.  Comparing the two expressions for $f$ on an
overlap in odd degree two gives
$$
 df_+\,c_{\mathfrak Y}
 -f_{\red}^*c_{\mathfrak X}\circ\bigwedge^2df_-
 =\delta b_f,
$$
with a simultaneous choice of sign for the cocycles.  Here
$\delta$ is the \v{C}ech coboundary, and the target cocycle
is pulled back to the chosen source cover.  Passing to cohomology
gives \eqref{eq:primary-naturality}.  The calculation respects
the comparison arrows on an \'etale groupoid and therefore applies
to smooth Deligne--Mumford superstacks.

\subsection{The extension problem}\label{subsec:extension-problem}

The boundary of stable supermoduli provides the starting point
for our nonvanishing argument.  In Section~\ref{sec:boundary}
we will construct a morphism $f:Z\to D$, where $D$ is
a smooth part of the separating Neveu--Schwarz boundary and
$Z$ is a smooth superscheme supported on a proper curve.
The compatibility identity \eqref{eq:primary-naturality} will
identify a nonzero comparison class on this source.  At that
stage the test map $f$ takes values entirely in the boundary;
it does not yet detect the class on the moduli of smooth
supercurves.  We call $Z$ the \emph{test superscheme} or
\emph{test source}, and reserve \emph{test curve} for its ordinary
reduction $C$.  The odd directions of $Z$ are essential
for the comparison of primary classes.

To reach the smooth locus, we extend both $Z$ and its map
in one even direction transverse to $D$.  Successive
extensions modulo increasing powers of a local boundary
parameter will produce a formal neighbourhood of $Z$ mapping
to the ambient superstack.  Inverting that parameter will then
give a map to the complement of $D$.  We shall still have
to prove that the comparison class survives this inversion;
that is the purpose of Sections~\ref{sec:Laurent}
and~\ref{sec:algebraic}.

This subsection isolates the extension problem needed for that
construction.  We fix the initial morphism $f:Z\to D$ but
allow the source thickening to be constructed together with the
extended map.  
Consequently, changes of local parametrisation are
identified when they give isomorphic source--map pairs.  The
coefficient sheaf for extending such a pair is therefore obtained
by quotienting tangential map corrections by source tangent
directions, and then tensoring with the appropriate power of the
conormal line.  These extension obstructions are distinct from
the primary class of Subsection~\ref{subsec:primary-conventions}:
they govern the construction of a transverse family on which
that class can be tested.
We first prove the one-order extension statement, identify its first-order
class, and then obtain the equivariant formal existence criterion
applied in Section~\ref{sec:vanishing}.

Let $D\subset\mathfrak X$ be a smooth effective even Cartier
divisor in a smooth separated Deligne--Mumford superstack of finite
type over $\C$.  Let $f:Z\to D$ be a morphism from a smooth
superscheme with proper reduced curve $C$, and assume that
$$
 df:T_Z\longrightarrow f^*T_D
 \quad\text{is a locally split injection}.
$$
Set
$$
 Q_f=\operatorname{coker}(df),\qquad
 A_f=f^*(\cI_D/\cI_D^2),\qquad
 \cM_r=(Q_f\otimes A_f^r)_{\bar0}\quad(r\geq1).
$$
Here $Q_f$ and $A_f$ are, respectively, a super vector
bundle and an even line bundle on the \emph{full} superscheme
$Z$.  The sheaf $\cM_r$ is the \emph{coefficient sheaf}
for the order-$r$ extension obstruction: local differences of
extensions are its sections, and their gluing obstruction will
lie in $H^1(Z,\cM_r)$.  Its even part is taken with respect
to total parity.  In particular, a product of an odd coefficient
function and an odd local section of $Q_f$ is retained;
taking the even part does not set odd functions equal to zero.
Cohomology of $\cM_r$ is the even part of the graded
coherent cohomology of $Q_f\otimes A_f^r$; equivalently,
it is computed on the ordinary scheme
$(|Z|,\cO_{Z,\bar0})$.  Thus this coefficient is not
replaced by its restriction to $C$.

Local descriptions of $f$ in a target atlas will always mean
the following.  Choose an \'etale atlas $V\to\mathfrak X$
and an \'etale superscheme neighbourhood $U$ of a chosen
point of $Z\times_{\mathfrak X}V$.  Projection gives
$\widetilde f:U\to V$, together with its comparison
$2$-isomorphism to the original map into $\mathfrak X$.
We call this a \emph{chosen local lift}.  All subsequent local
coordinate constructions are made near this specified lift.

\begin{definition}[Transverse extensions]\label{def:transverse}
Write $\iota_D:D\hookrightarrow\mathfrak X$, put
$Z_{[0]}=Z$, and regard $f_{[0]}=\iota_D\circ f$
as a morphism to $\mathfrak X$.  For $r\geq1$, an
\emph{order-$r$ transverse extension} of $f$ consists
of a superscheme $Z_{[r]}$, a closed immersion
$i_r:Z\hookrightarrow Z_{[r]}$, and a morphism
$f_{[r]}:Z_{[r]}\to\mathfrak X$, with a specified
$2$-commutative square
$$
\begin{tikzcd}[column sep=large,row sep=large]
 Z\arrow[r,hook,"i_r"]\arrow[d,"f"']&
 Z_{[r]}\arrow[d,"f_{[r]}"]\\
 D\arrow[r,hook,"\iota_D"']&\mathfrak X.
\end{tikzcd}
$$
We require this square to be $2$-Cartesian.  Explicitly, the
specified comparison
$f_{[r]}\circ i_r\simeq\iota_D\circ f$ must induce
an isomorphism
$$
 Z\xrightarrow{\sim}Z_{[r]}\times_{\mathfrak X}D.
$$
The adjective $2$-Cartesian records that this is a fibre
product of stacks.  For a superscheme $T$, an object of
$Z_{[r]}\times_{\mathfrak X}D$ is a triple
$$
 (a:T\to Z_{[r]},\ b:T\to D,\
   \varphi:f_{[r]}\circ a\xRightarrow{\sim}\iota_D\circ b).
$$
The isomorphism $\varphi$ is part of the datum; see
\cite[Tag~0040]{StacksProject}.  For superschemes viewed as
stacks this recovers the ordinary Cartesian fibre product.
Moreover, $\iota_D$ is a representable closed immersion,
so its base change here is the closed subsuperscheme cut out by
the pulled-back ideal of $D$.  Thus the displayed condition
identifies the entire inverse-image subsuperscheme with $Z$,
including its odd functions and nilpotent structure.

Let $\cI_{[r]}=f_{[r]}^{-1}\cI_D\cdot\cO_{Z_{[r]}}$.
\'Etale-locally on $Z$, we further require an isomorphism of
augmented superalgebras with ideals
$$
 (\cO_{Z_{[r]}},\cI_{[r]})
 \simeq (\cO_Z[t]/(t^{r+1}),(t)),
 \qquad |t|=\bar0,
$$
inducing the given identification with $\cO_Z$ modulo
$\cI_{[r]}$.  
Here $|t|=\bar0$ says that the new
parameter $t$ is even.  On a chosen \'etale-local lift to
a target atlas, as described above, the morphism is an immersion
and a local equation of $D$ pulls back to the generator
$t$.  Transversality means this identification of the new
even direction with the boundary-normal direction.

To extend a \emph{given} order-$(r-1)$ object, we also specify
an identification
$$
 (|Z|,\cO_{Z_{[r]}}/\cI_{[r]}^r)
 \xrightarrow{\sim}Z_{[r-1]},
$$
together with a compatible $2$-isomorphism between the
restricted map and $f_{[r-1]}$.  All isomorphisms of
extensions must respect these identifications.  In particular,
we keep the preceding-order extension fixed, but choose the
next source thickening as part of the extension.
\end{definition}

The parameter $t$ and the product presentations in
\cref{def:transverse} are local choices.  They need not glue
to a global generator of the boundary ideal or to a global
retraction $Z_{[r]}\to Z$.  The source thickening is
constructed together with its map to $\mathfrak X$.

For $r\geq1$, transversality identifies the conormal line
$\cI_{[r]}/\cI_{[r]}^2$ with $A_f$, although
$\cI_{[r]}$ itself is not a Cartier ideal.  Indeed, on a
local product chart with coefficient superalgebra $B$,
$$
 R_r=B[t]/(t^{r+1}),\qquad I_r=(t),\qquad
 t\cdot t^r=0\quad\hbox{in }R_r.
$$
Thus $t$ is a zero divisor at finite order.  In the compatible
formal limit the corresponding local ring and ideal are
$$
 \widehat R=B[[t]],\qquad I=(t),\qquad
 \widehat R\xrightarrow{\;\cdot t\;}I\text{ is an isomorphism}.
$$
Multiplication by $t$ is injective in $B[[t]]$, even if
$B$ has nilpotents, as follows by comparing coefficients of
power series.  Hence $I$ is an effective even Cartier ideal.
The proof of \cref{cor:equivariant-formal} constructs these
compatible product presentations on fixed affine open subsets of
the central superscheme.

\Needspace{11\baselineskip}
\begin{proposition}[Transverse extension with variable source]
\label{prop:transverse-extension}
Fix $r\geq1$ and an order-$(r-1)$ transverse extension.
There is a canonical obstruction
\begin{equation}\label{eq:transverse-obstruction}
 \operatorname{ob}_r(f_{[r-1]})\in
 H^1(Z,\cM_r)
 =H^1\bigl(Z,(Q_f\otimes A_f^r)_{\bar0}\bigr).
\end{equation}
An order-$r$ extension exists if and only if this class vanishes.
When it exists, the set of its isomorphism classes is a torsor
under $H^0(Z,\cM_r)$.  An automorphism inducing the identity
on the specified order-$(r-1)$ object is the identity.
\end{proposition}

We separate the local geometry from the obstruction calculation.
First, \cref{lem:etale-local-normal-form} constructs adapted
homogeneous coordinates and proves that transverse extensions
are graphs in those coordinates.  Then \cref{lem:graph-torsor}
identifies the difference between two lifts as an even section
of $Q_f\otimes A_f^r$.  We use those differences to
construct the obstruction cocycle and descend the local maps
and their sources, including the comparison $2$-isomorphisms
for a stack target.

We introduce the notation for the local description.  Since
$Z$ is smooth and $Z_{\red}=C$ is a curve, its dimension
near a chosen point is $1\mid q$.  Write
$\operatorname{rk}Q_f=a\mid b$.  For the adapted coordinates
to be constructed, we use the notation
$$
 z=(u\mid\alpha^1,\ldots,\alpha^q),
 \qquad
 w=(v^1,\ldots,v^a\mid\beta^1,\ldots,\beta^b),
 \qquad t.
$$
The tuple $z$ consists of coordinates on the source;
$w$ records the complementary directions within the divisor;
and $t$ is the even transverse parameter.  Thus
$u,v^1,\ldots,v^a,t$ are even, while the
$\alpha^i,\beta^j$ are odd.

All local descriptions below refer to the chosen lift to an
\'etale target atlas.  After localising, we continue to denote
the source and its lifted map by $Z$ and $f$.
The notation $\Phi(z,t)$ denotes a tuple of truncated
polynomials in $t$ with coefficients in $\cO_Z$.

\begin{lemma}[Adapted \'etale neighbourhoods and graphs]
\label{lem:etale-local-normal-form}
In the setting of \cref{prop:transverse-extension},
\'etale-locally near any point of $Z$, there is a target
neighbourhood $V\to\mathfrak X$ and an \'etale morphism
$$
 (\pi_Z,w,t):
 V\longrightarrow Z\times\mathbb A^{a+1\mid b}
$$
such that, writing $D_V=D\times_{\mathfrak X}V$,
$$
 D_V=(t=0),\qquad
 f(Z)=(w=t=0),\qquad
 \pi_Z\circ f=\Id_Z.
$$

For every $r\geq1$, projection to $(\pi_Z,t)$
identifies each local order-$r$ transverse extension with
$Z\times\Spec(\C[t]/(t^{r+1}))$.  Its map to $V$
is uniquely determined by a graph
$$
 w=\Phi(z,t),\qquad
 t^{r+1}=0,\qquad
 \Phi(z,0)=0,
$$
where each component of $\Phi$ has the parity of the
corresponding component of $w$.  Conversely, every such
graph determines a local transverse extension.
\end{lemma}

\begin{proof}
Choose an \'etale atlas $V\to\mathfrak X$ and pull back
the source.  We may then work with a morphism from a smooth
superscheme $Z$ to the smooth divisor $D_V\subset V$.
This is the super-coordinate counterpart of the ordinary
\'etale-local embedding statement recalled in
\cite[Section~1]{RydhEtale}. 

Dualising the split differential gives a locally split
surjection $f^*\Omega_{D_V}\twoheadrightarrow\Omega_Z$.
Choose homogeneous functions $z$ on $D_V$ whose
pulled-back differentials form a homogeneous basis at the
point under consideration.  After shrinking, their restrictions
give an \'etale morphism
{$h:Z\to\mathbb A^{1\mid q}$}.  Complete these functions
by homogeneous functions $\xi$ so that $(z,\xi)$
are \'etale coordinates on $D_V$.  Lift them to $V$
and choose a local equation $t$ of $D_V$.
The differentials of $(z,\xi,t)$ form a homogeneous
basis at the chosen point.  The super Jacobian criterion
therefore makes these \'etale coordinates near that point
in $V$.

Let $\pi:V\to\mathbb A^{1\mid q}$ be defined by
the lifted functions $z$, and form the Cartesian square
$$
\begin{tikzcd}[column sep=large,row sep=large]
 V'=V\times_{\mathbb A^{1\mid q}}Z
   \arrow[r]\arrow[d,"\pi_Z"']&V\arrow[d,"\pi"]\\
 Z\arrow[r,"h"']&\mathbb A^{1\mid q}.
\end{tikzcd}
$$
The upper horizontal map is \'etale, and $f$ lifts to
$f'=(f,\Id_Z):Z\to V'$.  On $V'$, subtract the
pullbacks of $f^*\xi$ from the functions $\xi$,
and denote the resulting tuple by $w$.  Translation in
these relative homogeneous coordinates shows that
$$
 (\pi_Z,w,t):V'\longrightarrow
 Z\times\mathbb A^{a+1\mid b}
$$
is \'etale and carries $f'$ to the zero section.
The inverse image of that section is \'etale over $Z$
and has the section $f'$, which is an open immersion.
The complement of its image is closed in this inverse image,
and the inverse image is closed in $V'$.  Removing that
complement from $V'$ leaves precisely the image of the
chosen section $f'$ over the zero section.  Thus this
local image is cut out by $w,t$, including the odd
components of $w$.

Now consider a transverse extension of order $r$.
Its central lift to $V'$ extends uniquely because
$V'\to V$ is formally \'etale.  Projection to
$(\pi_Z,t)$ gives a morphism
$$
 Z_{[r]}\longrightarrow Z\times\Spec(\C[t]/(t^{r+1}))
$$
which is the identity on the central superscheme.
Its map on the first conormal line is an isomorphism because
the pullback of $t$ generates the transverse ideal.
The local product condition identifies the degree-$k$
associated-graded piece with the $k$-th tensor power of
that line, for $1\leq k\leq r$.  The projection is
therefore an isomorphism on every piece of the finite ideal
filtration, hence an isomorphism on the whole structure sheaf.
The remaining functions $w$ have unique values
$\Phi(z,t)$, with $\Phi(z,0)=0$, proving the
graph description.

Conversely, the graph in the coordinate model has a unique
lift to the chosen \'etale neighbourhood extending the central
map, by formal \'etaleness across the nilpotent ideal
$(t)$.  Its map is an immersion: in the coordinate model
it is a closed immersion, and its lift is the section of the
pulled-back \'etale neighbourhood singled out by the prescribed
central lift.  The graph has the required local product structure and
inverse-image boundary.  It therefore defines a local transverse
extension.
\end{proof}

\begin{lemma}[Local graphs and their differences]\label{lem:graph-torsor}
In the setting of \cref{prop:transverse-extension}, extensions
exist locally in the \'etale topology.  Their local isomorphism
classes form a torsor under $\cM_r$.  This identification
is independent of the adapted coordinates and is natural under
isomorphisms of the given data.
\end{lemma}

\begin{proof}
Apply \cref{lem:etale-local-normal-form}.  The coordinates
$z$ parametrise the source, $w$ are the complementary
directions within the divisor, and $t$ is normal to the
divisor.  In our curve setting the even dimension of $z$
is one.  The given order-$(r-1)$ extension is represented
uniquely by
\begin{equation}\label{eq:graphs}
 w=\Phi(z,t),\qquad t^r=0,\qquad \Phi(z,0)=0.
\end{equation}
For $r=1$, this means the central graph $w=0$.
Lifting its coefficients modulo $t^{r+1}$ gives a local
order-$r$ extension.  With one such lift as origin, every
other graph is obtained uniquely by adding
$$
 t^r h(z),\qquad |h_a|=|w_a|.
$$
Its parametrisation is immaterial: the graph determines the
source together with its map.  

We next check how the difference $t^r h(z)$ transforms on an
overlap.  This will identify its intrinsic coefficient bundle
and show that the construction does not depend on the adapted
coordinates.  Write the change of coordinates as
$$
 z_j=F_{ij}(z_i,w_i,t_i),\qquad
 w_j=G_{ij}(z_i,w_i,t_i),\qquad
 t_j=a_{ij}(z_i,w_i,t_i)t_i,
$$
where $a_{ij}$ is an even unit and
$G_{ij}(z_i,0,0)=0$.  Put
$$
 a_{ij}^0=a_{ij}|_Z,\qquad
 B_{ij}=\left.\frac{\partial G_{ij}}{\partial w_i}\right|_Z.
$$
The supermatrix $B_{ij}$ represents the induced differential
on $Q_f$.  Compare two graphs differing by $t_i^r h_i$.
Their $t_j$-coordinates differ by
$t_i\,\delta a_{ij}\in(t_i^{r+1})$, so they agree to
the order under consideration.  Their $z_j$-coordinates
may differ in order $r$.  However, the reference graph
vanishes at $t_j=0$, so its derivative in $z_j$ is
divisible by $t_j$.  Re-expressing both graphs at the same
$(z_j,t_j)$ therefore changes their difference only in
order $r+1$.  The remaining difference is
$t_i^r B_{ij}h_i$ modulo $t_i^{r+1}$.  Consequently
\begin{equation}\label{eq:graph-difference-transition}
 h_j=(a_{ij}^0)^{-r}B_{ij}h_i,
\end{equation}
where the central change of the $z$-coordinates is understood.

Equation~\eqref{eq:graph-difference-transition} is exactly the
transition rule for an even section of $Q_f\otimes A_f^r$:
the local generator $t_j^r$ of $A_f^r$ is
$(a_{ij}^0)^r t_i^r$.  The calculation holds on the full
superscheme $Z$, including its odd functions, and also holds
for $r=1$.  Subtraction of graphs is additive at order $r$, since
products of two differences vanish modulo $t^{r+1}$.
This proves the torsor assertion and its naturality.
\end{proof}

\begin{proof}[Proof of Proposition~\ref{prop:transverse-extension}]
We first construct the obstruction class from the local
extensions supplied by \cref{lem:graph-torsor}.
Choose such lifts of the given object.  On an overlap, transport
the first graph to the chart of the second and take their
difference.  By \cref{lem:graph-torsor} this is a section
$c_{ij}$ of $\cM_r$.  On triple overlaps these
differences satisfy $\delta c=0$, because transport respects
subtraction and the preceding order already satisfies descent.
Changing the local lifts changes $c$ by a coboundary.  The
class $[c]$ is $\operatorname{ob}_r(f_{[r-1]})$.
Equivalently, it is the class of the sheaf of local graph lifts,
viewed as a $\cM_r$-torsor.  This description also proves
independence of the chosen cover.  Coherent \'etale and Zariski
cohomology agree here, so its group is precisely
\eqref{eq:transverse-obstruction}.

We now show that vanishing of this torsor class produces a
global extension.  For a stack target, this requires descending
the source and checking the comparison arrows as well as
matching the graphs.
Choose an \'etale atlas $V\to\mathfrak X$, and write
$R=V\times_{\mathfrak X}V$, with source and target
projections $s,t_R:R\rightrightarrows V$.  Both are
\'etale.  First lift the source overlap charts across the
nilpotent extension; \'etale morphisms admit unique such lifts.
Given the first local graph and the comparison arrow
$\gamma_{ij}$ at the preceding order, there is then a unique
dotted arrow in
$$
\begin{tikzcd}[column sep=large,row sep=large]
 Z_{ij,[r-1]} \arrow[r,"\gamma_{ij}"] \arrow[d,hook]
   & R \arrow[d,"s"]\\
 Z_{ij,[r]} \arrow[r,"f_{i,[r]}"']
   \arrow[ur,dashed,"\widetilde\gamma_{ij}" description]
   & V .
\end{tikzcd}
$$
Its second endpoint $t_R\widetilde\gamma_{ij}$ is the
\emph{transported} first graph, whose difference from the second graph
is $c_{ij}$.  If $[c]=0$, correcting the local
graphs by a cochain makes these differences zero.

Agreeing graphs induce unique isomorphisms of their source
quotients.  On a triple overlap, the two composite comparison
arrows lift the same preceding arrow with the same first endpoint.
Formal \'etaleness of $s$ makes them equal.  The source
isomorphisms then agree as well, since each chosen lift is an
immersion.  The same argument shows that a relative automorphism
is the identity: its target arrow is the unique lift of the
identity arrow, and its source map is fixed by the immersion.
Thus the corrected graphs and comparison arrows satisfy effective descent. 

The descended source is a superscheme.  Indeed, its ordinary even
part is a nilpotent algebraic-space thickening of
$(|Z|,\cO_{Z,\bar0})$, hence is a scheme
\cite[Tag~05ZR]{StacksProject}.  Nilpotent thickenings of affine
schemes are affine \cite[Tag~06AD]{StacksProject}.
The odd structure sheaf descends with its multiplication by
quasi-coherent descent.

A global extension makes $[c]$ zero.  Conversely,
the preceding construction produces one whenever $[c]=0$.
Finally, the difference between two global extensions is a global
section of $\cM_r$, and addition of any such section to
the local graphs gives another extension.  This action is free
and transitive on isomorphism classes, as asserted.
\end{proof}

\Needspace{9\baselineskip}
\begin{lemma}[The first-order normal sequence]\label{lem:first-normal}
Put $N_f=\operatorname{coker}(T_Z\to f^*T_{\mathfrak X})$.
There is an exact sequence of super vector bundles
\begin{equation}\label{eq:first-normal-sequence}
 0\longrightarrow Q_f\longrightarrow N_f
 \longrightarrow A_f^{-1}\longrightarrow0.
\end{equation}
First-order transverse extensions correspond to even splittings
of this sequence, and its extension class is
$\operatorname{ob}_1(f)$.
\end{lemma}

\begin{proof}
The normal sequence of the smooth divisor and the differential
of $f$ give the commutative diagram with exact rows
$$
\begin{tikzcd}[column sep=small,row sep=large]
0 \arrow[r] & T_Z \arrow[r,"df"] \arrow[d,equal]
 & f^*T_D \arrow[r] \arrow[d,hook]
 & Q_f \arrow[r] \arrow[d,hook] & 0\\
0 \arrow[r] & T_Z \arrow[r] & f^*T_{\mathfrak X}
 \arrow[r] \arrow[d,two heads]
 & N_f \arrow[r] \arrow[d,two heads] & 0\\
 & & A_f^{-1} \arrow[r,equal] & A_f^{-1}. 
\end{tikzcd}
$$
Taking cokernels yields \eqref{eq:first-normal-sequence}.
In adapted coordinates, the first-order graph
$w=t h(z)$ has transverse tangent
$\partial_t+h$ modulo $T_Z$.  This is a right inverse
of $N_f\to A_f^{-1}$, and every local right inverse
arises this way.  Coordinate changes preserve this correspondence
by \eqref{eq:graph-difference-transition}.  Hence the local
right inverses glue exactly when the first-order graphs glue.
Their differences lie in
$\Hom(A_f^{-1},Q_f)_{\bar0}=\cM_1$ and give the same
\v{C}ech cocycle, identifying the two obstruction classes.
\end{proof}

\begin{corollary}[Equivariant lifting and formal extension]
\label{cor:equivariant-formal}
Suppose that a finite group $H$ acts on $Z$ and on the
pair $(\mathfrak X,D)$, and that $f:Z\to D$ is
$H$-equivariant.  These actions are understood coherently:
the action on $\mathfrak X$ preserves $D$, and the
equivariance of $f$ includes specified comparison
$2$-isomorphisms satisfying the group law.  They induce
actions on $Q_f$, $A_f$, and the obstruction groups.

Let a given order-$(r-1)$ extension be $H$-equivariant,
with its action and comparison isomorphisms restricting to
those specified for $f$.  Then
$\operatorname{ob}_r(f_{[r-1]})$ is $H$-invariant,
and its vanishing suffices for a coherent equivariant extension.

In particular, starting from the $H$-equivariant morphism
$f:Z\to D$ alone, suppose that
\begin{equation}\label{eq:all-order-criterion}
 H^1\bigl(Z,(Q_f\otimes A_f^r)_{\bar0}\bigr)^H=0
 \qquad(r\geq1).
\end{equation}
Then there is a compatible equivariant formal map
$$
 \widehat f:\widehat Z\longrightarrow\widehat{\mathfrak X}_D,
 \qquad \widehat Z\text{ locally }\Spf\cO_Z[[t]].
$$
Its pulled-back boundary ideal $\cI$ is Cartier.  The
conormal line bundle of $Z\hookrightarrow\widehat Z$
is $\cI/\cI^2\simeq A_f$, and more generally
$\cI^r/\cI^{r+1}\simeq A_f^r$ for every $r\geq1$.
\end{corollary}

\begin{proof}
The actions on $Z$, on $(\mathfrak X,D)$, and on the
preceding extension act on the local graph lifts.  Naturality
in \cref{lem:graph-torsor} therefore makes the obstruction
class $H$-invariant.
If it vanishes, \cref{prop:transverse-extension} gives a
nonempty torsor $\mathscr T_r$ of global extension classes
under the complex vector space $V_r=H^0(Z,\cM_r)$.
The prescribed action on the preceding order induces an affine
$H$-action on this torsor, with linear part its natural
action on $V_r$.  Choose $x\in\mathscr T_r$.
The affine average
$$
 \overline x=x+\frac1{|H|}\sum_{h\in H}(h\cdot x-x)
$$
is $H$-fixed, because the group permutes the finite orbit.
Choose an extension representing $\overline x$.  Each
$h$ gives an isomorphism to its transform respecting the
specified preceding order.  Such an isomorphism is unique by
\cref{prop:transverse-extension}; therefore the composites for
$h,k$ equal the isomorphism for $hk$.  This proves
coherence of the source action and of the target comparison
arrows. 

Under \eqref{eq:all-order-criterion}, the invariant obstruction
vanishes at each step.  Inductively choose equivariant extensions
with their preceding-order identifications.  We now justify the
local form of their inverse limit on one fixed affine open,
so that the product presentations are compatible at every order.

Choose an affine open superscheme $U=\Spec B\subset Z$
on which the full super conormal line $A_f$ is trivial,
and let $U_{[r]}$ be its thickening in $Z_{[r]}$.
Each $U_{[r]}$ is affine: its ordinary even scheme is
an affine nilpotent thickening, and its odd coherent algebra
is recovered from its sections.  Write
$$
 R_r=\Gamma(U_{[r]},\cO_{U_{[r]}}),\qquad
 I_r=\Gamma(U_{[r]},\cI_{[r]}),\qquad R_0=B.
$$
The maps $R_r\twoheadrightarrow R_{r-1}$ are square-zero
surjections.  The superalgebra $B$ 
(though generally not its even subalgebra)
is smooth over $\C$,
so a section $B\to R_{r-1}$ of its augmentation lifts
to a section $B\to R_r$.  Starting with $\Id_B$,
we obtain compatible sections of all augmentations $R_r\to B$.

Choose a generator of $A_f|_U$.  It determines an even
generator $t_1\in I_1$, and choose compatible even lifts
$t_r\in I_r$.  These lifts exist because the transition
map of the ideal sheaves is surjective with quasi-coherent kernel
on an affine thickening.  The class of $t_r$ generates
$I_r/I_r^2$, so nilpotent Nakayama gives $I_r=(t_r)$.
The resulting compatible homomorphisms
$$
 B[t]/(t^{r+1})\longrightarrow R_r,\qquad t\longmapsto t_r,
$$
are isomorphisms on every associated-graded piece for the finite
boundary-ideal filtrations, and hence are isomorphisms.
Consequently $\varprojlim_r R_r\simeq B[[t]]$.

Multiplication by the even variable $t$ is injective in
$B[[t]]$, even though $B$ may have odd nilpotents.  The resulting ideal is
therefore an effective Cartier ideal.  Pullback of the boundary
equation identifies the conormal line bundle
$\cI/\cI^2$ of $Z$ in $\widehat Z$ with $A_f$.
Multiplication of local generators gives the higher identifications
$\cI^r/\cI^{r+1}\simeq(\cI/\cI^2)^{\otimes r}
\simeq A_f^r$.
\end{proof}

\begin{remark}[Extensions with variable source]
\label{rem:variable-source}
The initial source $Z$, the pair $(\mathfrak X,D)$,
and the morphism $f:Z\to D$ are fixed.  At order $r$,
the previously constructed extension is also fixed, but the next
source thickening and its map are chosen together.  In
particular, their local parametrisations may change.  This is
the freedom used in \cref{lem:graph-torsor}: differences in
source tangent directions are absorbed by reparametrisation,
leaving the quotient $Q_f=f^*T_D/df(T_Z)$.  There is no
additional independent boundary-normal summand, since the local
parameter $t$ is already the pullback of the boundary
equation.  Its changes of trivialisation are recorded by the
factor $A_f^r$.

The criterion of \cref{cor:equivariant-formal} concerns
extensions with variable source.  Prescribing the higher-order
source thickenings in advance gives a different lifting problem,
whose obstructions are not determined by the preceding argument.
In Section~\ref{sec:vanishing}, we construct the source together
with the map, as required here.
For the passage from this formal extension to localised
algebraic maps, see in particular \cref{rem:formal-localisation}.
\end{remark}

\section{A common core and the two elliptic tails}
\label{sec:boundary}

Section~\ref{sec:transverse} gave a criterion for extending a
map to a smooth divisor in a transverse formal direction.  We now
construct the initial map to which that criterion will apply.
The divisor is a separating Neveu--Schwarz boundary component:
its curves have a component of genus $h=g-1$, which we call
the \emph{core}, joined at one node to an elliptic component,
which we call the \emph{tail}.  We choose an even spin structure
on the core and use the two spin parities of the tail to obtain
the two components of genus-$g$ supermoduli.

The construction has three requirements, each serving a different
part of the proof.  Its differential must be a locally split
injection, so that the tangent quotient in
\cref{prop:transverse-extension} is a super vector bundle.
It must carry a coherent finite group action: the characters of
this action will eliminate the invariant parts of the low-order
obstruction groups in
Section~\ref{sec:vanishing}.  Finally, primary-class functoriality
must give a nonzero class on the source after mapping its
coefficient into the ambient tangent bundle.  This is the
\emph{central comparison class} that will later be transported
away from the boundary.

We first choose the core and compute the deformation
representations of the tails.  We then determine the tangent
and normal bundles along the family obtained by moving the
attaching point on the core.  These calculations allow us to
construct a source with odd parameters, verify the required
differential, and identify its nonzero comparison class.
In the even case we also move the attaching point on the tail;
the reason for this choice is made explicit in
\cref{rem:wag-tail}.  The formal extension itself is constructed
in the next section.

\subsection{The core}

The core is the genus-$h$ component of the normalisation of our
nodal genus-$g$ curve, together with its spin line bundle $L$.
We initially fix this pair and let its attaching point $p$
range over the complete curve $C$.  We will subsequently
deform the supercurve in one odd direction
$\eta\in H^1(C,L^{-1})$.  The resulting pointed family
will supply the primary class used in the comparison.

We choose $L$ to be ineffective.  This condition has a
specific cohomological purpose: $H^0(C,L)=H^1(C,L)=0$
will make the map carrying the pointed-core class into the
boundary tangent bundle an isomorphism in
\cref{lem:Jh-injection}.  The positivity of
$\deg L=h-1$ will also give the higher-order vanishings
needed for formal extension.  The finite symmetry, by contrast,
will come entirely from the elliptic tail.  The following
construction provides a suitable core in every genus under
consideration.

Fix $g=h+1$ with $h\geq2$.  We choose a smooth connected
projective curve $C$ of genus $h$ and a theta characteristic
$L$ satisfying
\begin{equation}\label{eq:core}
 L^2\simeq K_C,\qquad \deg L=h-1>0,\qquad
 H^0(C,L)=H^1(C,L)=0.
\end{equation}
The vanishing of $H^0(C,L)$ is the ineffectivity condition;
that of $H^1(C,L)$ follows also from Serre duality.  In
particular the core has even spin parity.

\begin{lemma}[An ineffective even theta characteristic in every genus]
\label{lem:ineffective-theta}
For every $h\geq2$ there exists a hyperelliptic pair $(C,L)$
satisfying \eqref{eq:core}.
\end{lemma}

\begin{proof}
Choose distinct $a_1,\ldots,a_{2h+2}\in\C$, and let
$C$ be the smooth projective completion of
$$
 y^2=\prod_{i=1}^{2h+2}(x-a_i).
$$
The function $x$ defines a finite flat double cover
$\pi:C\to\mathbb P^1$, branched at the $a_i$.
Write $p_i$ for the ramification point over $a_i$, and put
$$
 D_0=p_1+\cdots+p_{h+1},\qquad
 L=\cO_C(D_0)\otimes\pi^*\cO_{\mathbb P^1}(-1).
$$
Every $2p_i$ is a fibre of $\pi$, up to linear equivalence,
and the canonical bundle formula for the double cover gives
$K_C\simeq\pi^*\cO_{\mathbb P^1}(h-1)$.  Hence
$$
 L^2\simeq\pi^*\cO_{\mathbb P^1}(h+1-2)\simeq K_C.
$$

To prove ineffectivity, we compute a direct image explicitly.
The functions
$$
 1,\qquad w=\frac{y}{\prod_{i=1}^{h+1}(x-a_i)}
$$
are global sections of $\cO_C(D_0)$.  They define a map
between rank-two bundles
$$
 \cO_{\mathbb P^1}^{\oplus2}
 \longrightarrow\pi_*\cO_C(D_0).
$$
We check that it is an isomorphism at every fibre.  At a selected
ramification point $p_i$, choose a uniformiser $t$, so
that $x-a_i$ has order two.  The local module
$\cO_C(D_0)$ is $t^{-1}\cO_C$, and $w$ has a
simple pole with nonzero leading coefficient.  Modulo the maximal
ideal of the base, the classes of $t^{-1}$ and $1$
are a basis; therefore so are those of $w$ and $1$.
At an unselected ramification point the module is $\cO_C$,
with fibre basis $1,t$, and $w$ has a simple zero.
Again $1,w$ form a basis.  Away from ramification their
values on the two sheets are opposite and nonzero.  At the two
points above infinity these values tend to $+1$ and $-1$,
since numerator and denominator have the same pole order.
Thus the map is an isomorphism everywhere, by Nakayama's lemma.

The projection formula now gives
$$
 \pi_*L\simeq\cO_{\mathbb P^1}(-1)^{\oplus2}.
$$
Since $\pi$ is finite, the cohomology of $L$ is that of
this direct image.  Its zeroth and first cohomology vanish,
proving \eqref{eq:core}.
\end{proof}

Hyperellipticity is used only to exhibit a pair satisfying
\eqref{eq:core}; the construction below applies to an arbitrary 
core satisfying this.  Set
$$
 U_C=H^1(C,L^{-1}),\qquad W_C=H^1(C,T_C).
$$
Both $L^{-1}$ and $T_C=L^{-2}$ have negative degree.
Their zeroth cohomology vanishes, so Riemann--Roch gives
\begin{equation}\label{eq:dimensions}
 \dim U_C=2h-2,\qquad \dim W_C=3h-3.
\end{equation}
Here $U_C$ and $W_C$ are, respectively, the odd and even
infinitesimal deformation spaces of the unpointed core
supercurve.  In particular, a nonzero $\eta\in U_C$ exists
for every $h\geq2$.  More precisely,
\cref{lem:pointed-core-test} will construct a pointed family
whose primary class is $\eta$, with the odd base factor
included, and \cref{lem:Jh-injection} will show that this
class remains nonzero after applying the moving-point
differential.  Their combination in \cref{prop:central-class}
is the promised nonzero central comparison class.

\subsection{Tail characters}

Having fixed an even core, we now choose a pointed elliptic
supercurve whose spin parity will determine that of the clutched
curve.  We need more than its parity: we must know which
deformations can be produced by moving the attaching point and
how the tail automorphisms act on them.  These data will determine
both the source of the boundary test and the obstruction
coefficients in the next section.

A Neveu--Schwarz (NS) marking is a section through the smooth
locus of a family of supercurves.  In relative coordinates
$(x\mid\theta)$, it specifies an even coordinate and an
odd coordinate on the base; see \cite[Section~3.3]{DW} and
\cite[Definition~2.4]{FKP}.  We write $\M_{k,1}^\eps$
for smooth supermoduli with one such marking.  By \emph{odd
point variation} we mean changing the odd coordinate of this
section to first order while keeping the underlying supercurve
fixed.  Its image in the deformation space of the
\emph{pointed} supercurve may nevertheless vanish, because
isomorphisms of the unpointed supercurve can identify different
sections.

Let $(E,q)$ be a smooth pointed elliptic curve, with $q$
as origin, and choose a theta characteristic $M$ by
$$
 \begin{array}{c|c|c}
 \eps&M&h^0(E,M)\bmod2\\\hline
 +&\text{a nontrivial two-torsion line}&0\\
 -&\cO_E&1.
 \end{array}
$$
Define
$$
 R_E=M_q^{-1},\qquad
 U_E=H^1(E,M^{-1}(-q)),\qquad
 W_E=H^1(E,T_E(-q)),\qquad W_0=W_C\oplus W_E.
$$
The line $R_E$ parametrises odd point variations before
quotienting by isomorphisms.  The space $U_E$ parametrises
odd deformations of the pointed supercurve up to isomorphism,
and $W_E$ parametrises its even deformations; these are the
pointed tangent-space formulas of \cite[Section~3.3]{DW}.
Thus $R_E$ and $U_E$ describe different deformation
problems.  We will see that their natural map is an isomorphism
for the even tail and zero for the odd tail, although both
spaces have dimension one.

To compute that map, write $i_q:\{q\}\hookrightarrow E$.
The evaluation sequence restricts an odd superconformal vector
field to its value at the marking:
\begin{equation}\label{eq:tail-evaluation}
 0\longrightarrow M^{-1}(-q)\longrightarrow M^{-1}
 \longrightarrow (i_q)_*R_E\longrightarrow0.
\end{equation}
The kernel consists of fields vanishing at $q$.  Both
pointed infinitesimal automorphism spaces vanish:
$$
 H^0(E,T_E(-q))=H^0(E,M^{-1}(-q))=0,
$$
because the two bundles have degree $-1$.  Taking cohomology
of \eqref{eq:tail-evaluation} therefore gives
\begin{equation}\label{eq:tail-evaluation-cohomology}
 0\longrightarrow H^0(E,M^{-1})\longrightarrow R_E
 \xrightarrow{\delta_q}U_E
 \longrightarrow H^1(E,M^{-1})\longrightarrow0.
\end{equation}
Here $H^0(E,M^{-1})$ is the odd infinitesimal automorphism
space of the \emph{unpointed} supercurve, and $\delta_q$
sends a point variation to its pointed deformation class.
Indeed, lift the variation locally to superconformal vector
fields; their differences vanish at $q$ and give the
connecting cocycle in $M^{-1}(-q)$.  The last map forgets
the marking.  This sequence will also compute the characters
of the finite action that we now introduce.

Let $\iota_E:E\to E$ be inversion.  It preserves $M$,
and an isomorphism $\iota_E^*M\to M$ can be rescaled to
respect the spin isomorphism $M^2\simeq K_E$.  Choose such
a spin lift $a$.  At the fixed point $q$, the scalar
by which $a$ acts on $M_q$ has square $-1$, since
its square is the action of inversion on
$K_{E,q}$.  The automorphism $a^2$ of $M$ covers
the identity of $E$, hence is a constant scalar; its value
at $q$ shows that it is $-\Id_M$.  Consequently
$$
 A=\langle a\rangle\simeq\boldsymbol\mu_4,
 \qquad \sigma=a^2,\qquad a^2=-\Id_M
$$
has $\sigma$ equal to scalar spin parity on the tail.
Write $\rho$ for the character of $a$ on $R_E$,
so $\rho^2=-1$.  We use the induced linearisations on all
bundles, fibres and cohomology groups; dual lines carry inverse
characters.  Throughout the rest of the paper, $A$ denotes
this order-four group acting on the tail, and acting trivially
on the core.  Its subgroup generated by $\sigma$ will
separate even boundary tangent directions from the smoothing direction.

\begin{lemma}[The two tail representations]\label{lem:tail-characters}
Both $U_E$ and $W_E$ have dimension one.
The connecting map $\delta_q:R_E\to U_E$ in
\eqref{eq:tail-evaluation-cohomology}, which sends a motion of
the marking to its pointed deformation class, has the following
behaviour and characters:
\begin{equation}\label{eq:tail-table}
\begin{array}{c|c|c|c}
 \eps&\chi(R_E)&\chi(U_E)&\delta_q\\\hline
 +&\rho&\rho&\text{isomorphism}\\
 -&\rho&\rho^3=\rho^{-1}&0.
\end{array}
\end{equation}
In both cases the character of $W_E$ is trivial.
\end{lemma}

\begin{proof}
Since $a$ fixes $q$, the evaluation sequence and all its
cohomology maps are equivariant.  We first apply
\eqref{eq:tail-evaluation-cohomology} to the two spin
characteristics, and then compute the even pointed tangent
space common to both cases.

\emph{Even tail.}
A nontrivial degree-zero line bundle has no section.  Serre
duality, using $K_E\simeq\cO_E$, also gives
$H^1(E,M^{-1})=0$.  Thus $\delta_q$ is an equivariant
isomorphism.  It follows that $\dim U_E=1$ and
$\chi(U_E)=\rho$.

\emph{Odd tail.}
For $M=\cO_E$, evaluation
$H^0(E,M^{-1})\to R_E$ is an equivariant isomorphism.
Hence $\delta_q=0$, and the last arrow in
\eqref{eq:tail-evaluation-cohomology} identifies
$$
 U_E\simeq H^1(E,M^{-1})
 \simeq H^1(E,\cO_E)\otimes H^0(E,M^{-1}).
$$
The second factor has character $\rho$.  Inversion acts by
$-1$ on the first, by Serre duality with $H^0(E,K_E)$.
Their product therefore has character $-\rho=\rho^3$.
Each factor is one-dimensional.  The zero connecting map also
expresses the geometric fact that a point variation can be absorbed
by a global odd infinitesimal automorphism of the unpointed curve.

It remains to compute $W_E$, the space of \emph{even}
deformations of the pointed elliptic supercurve.  Here the
adjective even refers to the parity of the tangent direction,
not to the spin parity of $M$; the calculation applies to
both tails.  The even evaluation sequence is
$$
 0\longrightarrow T_E(-q)\longrightarrow T_E
 \longrightarrow(i_q)_*T_{E,q}\longrightarrow0.
$$
Evaluation $H^0(E,T_E)\to T_{E,q}$ is an isomorphism:
a translation-invariant vector field on $E$ is determined
by its value at the origin.  Its cohomology sequence therefore
identifies
$$
 W_E\simeq H^1(E,T_E)
 \simeq H^1(E,\cO_E)\otimes H^0(E,T_E).
$$
Each factor is one-dimensional.  Inversion acts by $-1$
on $H^1(E,\cO_E)$ and on translation-invariant vector
fields, so their product has trivial character.  Scalar spin
parity acts trivially on these even spaces.  Thus
$\dim W_E=1$ and its $A$-character is trivial, as
asserted.
\end{proof}

\subsection{The boundary tangent and normal bundles}

The preceding subsection computed the deformation spaces of a
single pointed tail.  We now keep that tail fixed and let the
attaching point on the core range over $C$.  Clutching produces
a family in the separating boundary, and the tangent spaces of
the pointed components assemble into vector bundles along this
family.  We need their exact sequences as bundles, because our
comparison class lives in cohomology on the whole curve $C$.

There are two maps to understand.  Moving the core point maps
$T_C$ into the tangent bundle of the boundary; allowing the
node to smooth then includes the even boundary tangent in the
even ambient tangent bundle.  The first map will preserve the relevant
cohomology by ineffectivity of $L$.  For the second, scalar
spin parity on the tail acts trivially in even boundary tangent
directions and by $-1$ in the normal direction, giving a
canonical splitting.  Thus both passages preserve the class
we will construct.

Let $\Mbar_g^\eps$ denote the fixed-parity component of the
moduli superstack of stable supercurves constructed in
\cite[Theorem~A]{FKP}.  It contains $\M_g^\eps$ as
its smooth-curve locus.  Let
$\mathfrak X_g^\eps$ be its open substack consisting of
smooth curves and curves with exactly one separating NS node,
whose smooth normalisation has genera $h,1$ and parities
$+,\eps$. The
boundary in this open is denoted by $D=\Dsep$, with
inclusion $\iota_D:D\hookrightarrow\mathfrak X_g^\eps$.
In an adapted local
deformation chart, one even parameter smooths that NS node,
and setting it equal to zero leaves the two pointed
normalisation components; see \cite[Section~7.5]{FKP}.
Clutching $(C,L,p)$, for $p\in C$, to the fixed pointed
tail gives
\begin{equation} \label{inclusion}
 i:C\longrightarrow D_{\red}
 \longrightarrow(\mathfrak X_g^\eps)_{\red}.   
\end{equation}
 
Put
$$
 \cE_\pm=i^*T_\pm D,
 \qquad T=i^*T_+\mathfrak X_g^\eps.
$$
We also use $i$ for the composite map to the ambient reduction.
These are ordinary vector bundles on $C$, with the indicated
even or odd tangent interpretation.  Recall that
$A=\langle a\rangle\simeq\boldsymbol\mu_4$ is the spin
lift of inversion on the fixed tail, with $\sigma=a^2$
and $\chi(R_E)=\rho$.  It fixes the base $C$ and
acts on these bundles through the automorphisms of the
clutched family.

\begin{proposition}[Boundary bundles and their linearisations]
\label{prop:boundary-bundles}
With the tail action $A=\langle a\rangle$ just described,
the divisor $D$ is smooth and effective even Cartier in the
smooth Deligne--Mumford superstack $\mathfrak X_g^\eps$.
Along $i$, the tangent and normal sequences are
\begin{align}
 0&\longrightarrow L^{-1}\longrightarrow\cE_-
 \longrightarrow(U_C\oplus U_E)\otimes\cO_C
 \longrightarrow0,\label{eq:Eminus}\\
 0&\longrightarrow T_C=L^{-2}\longrightarrow\cE_+
 \longrightarrow W_0\otimes\cO_C
 \longrightarrow0,\label{eq:Eplus}\\
 0&\longrightarrow\cE_+\longrightarrow T\longrightarrow N
 \longrightarrow0,\label{eq:normal-seq}
\end{align}
where the NS normal and conormal lines are
\begin{equation}\label{eq:N-L}
 N=L^{-1}\otimes R_E,\qquad
 \cL=N^\vee=L\otimes R_E^{-1}.
\end{equation}
The group $A$ acts trivially on $\cE_+$, and its
character on $N$ is $\rho$.  In particular, the
normal sequence has the canonical $\sigma$-eigenspace splitting
\begin{equation}\label{eq:sigma-splitting}
 T=T^{\sigma=+1}\oplus T^{\sigma=-1}
   =\cE_+\oplus N.
\end{equation}
All these statements hold for both tail characteristics.
\end{proposition}

\begin{proof}
We first justify the geometry of the divisor and identify its
tangent directions.  Smoothness of the ambient superstack follows
from \cite[Theorem~A]{FKP}.  Near a curve with one NS node,
the local deformation description gives one even smoothing
parameter as part of a smooth coordinate system.  Its vanishing
therefore defines a smooth effective Cartier divisor on our open.
Clutching the two pointed smooth components has tangent image
equal to the kernel of the differential of this parameter.
See \cite[Proposition~7.8 and
Lemmas~7.9--7.10]{FKP}.  Since $h\geq2$, the components
have different genera and are distinguished by the normalisation.

We next compute how the pointed-core tangent spaces vary with
the attaching point.  This produces the bundle sequences
\eqref{eq:Eminus} and \eqref{eq:Eplus}.
Let $p_1,p_2:C\times C\to C$ be the projections, with
the first factor parametrising the moving point, and let
$\Delta$ be the diagonal.  For $B=L^{-1}$ or $T_C$,
define
$$
 \cP_B=R^1p_{1*}\bigl(p_2^*B(-\Delta)\bigr).
$$
Its fibre at $p$ is $H^1(C,B(-p))$, the odd or even
pointed-core tangent respectively.  The bundles $B(-p)$
have negative degree, so their zeroth cohomology vanishes and
their first cohomology has constant dimension.  Cohomology and
base change therefore makes $\cP_B$ locally free and
identifies these fibres.

Push forward the diagonal sequence
$$
 0\longrightarrow p_2^*B(-\Delta)\longrightarrow p_2^*B
 \longrightarrow p_2^*B|_\Delta\longrightarrow0
$$
along $p_1$.  Since $H^0(C,B)=0$ and
$p_1|_\Delta$ is an isomorphism, the result is
\begin{equation}\label{eq:pointed-core-tangent}
 0\longrightarrow B\longrightarrow\cP_B
 \longrightarrow H^1(C,B)\otimes\cO_C\longrightarrow0.
\end{equation}
The first map is the differential of point motion; the last is
the differential of forgetting the point.  Clutching adds the
independent pointed-tail tangent.  Thus, with
$\cE_{C,-}=\cP_{L^{-1}}$ and
$\cE_{C,+}=\cP_{T_C}$, there are natural identifications
$$
 \cE_-=\cE_{C,-}\oplus(U_E\otimes\cO_C),\qquad
 \cE_+=\cE_{C,+}\oplus(W_E\otimes\cO_C).
$$
Together with \eqref{eq:pointed-core-tangent}, these identifications give
\eqref{eq:Eminus} and \eqref{eq:Eplus}.

It remains to identify the smoothing direction and show that
the boundary tangent is a direct summand of the ambient tangent.
The normal sequence of the smooth divisor gives
\eqref{eq:normal-seq}.  The NS normal formula is
$N=\Psi_1^{-1}\otimes\Psi_2^{-1}$, where the two
$\Psi$-lines restrict to $L$ and $M_q$;
see \cite[Section~9.1, equation~(9.6)]{FKP}.  This gives
\eqref{eq:N-L}, including its linearisation.  In particular
the characters on $N$ and $\cL$ are $\rho$ and
$\rho^{-1}$.

The tail action fixes the core family and all its tangent bundles.
By \cref{lem:tail-characters} it also fixes $W_E$, so it
acts trivially on $\cE_+$.  On the other hand $\sigma$
acts by $-1$ on $N$.  Applying the projectors
$(1\pm\sigma)/2$ to \eqref{eq:normal-seq} identifies
$T^{\sigma=+1}$ with $\cE_+$ and maps
$T^{\sigma=-1}$ isomorphically onto $N$.  This proves
the canonical splitting \eqref{eq:sigma-splitting}.
\end{proof}

Let $\jmath_\partial:T_C\to\cE_+$ be the differential
of the family obtained by moving the core point inside $D$.
The differential of the boundary inclusion gives
$i^*d\iota_{D,+}:\cE_+\hookrightarrow T$.  Define
$\jmath$ by the composite
$$
 \jmath:
 T_C\xrightarrow{\ \jmath_\partial\ }\cE_+
 \xrightarrow{\ i^*d\iota_{D,+}\ }T,
 \qquad
 \jmath=(i^*d\iota_{D,+})\circ\jmath_\partial.
$$
Using $T_C\otimes L=L^{-1}$, define the induced maps
$$
 J_h^\partial=H^1(\jmath_\partial\otimes\Id_L),
 \qquad J_h=H^1(\jmath\otimes\Id_L).
$$
The following observation will ensure that the source class is
still nonzero after passage to the target tangent bundle.

\begin{lemma}[Injection of the moving-point class]\label{lem:Jh-injection}
The map
$$
 J_h^\partial:H^1(C,L^{-1})
 \longrightarrow H^1(C,\cE_+\otimes L)
$$
is an isomorphism.  The map
$J_h:H^1(C,L^{-1})\to H^1(C,T\otimes L)$
is injective.  More explicitly,
\eqref{eq:sigma-splitting} gives
$$
 H^1(C,T\otimes L)
 =H^1(C,\cE_+\otimes L)\oplus H^1(C,N\otimes L),
 \qquad
 J_h(\eta)=\bigl(J_h^\partial(\eta),0\bigr).
$$
\end{lemma}

\begin{proof}
Tensor \eqref{eq:Eplus} by $L$.  Its long exact sequence
contains
$$
 H^0(C,W_0\otimes L)\longrightarrow H^1(C,L^{-1})
 \xrightarrow{J_h^\partial}H^1(C,\cE_+\otimes L)
 \longrightarrow H^1(C,W_0\otimes L).
$$
The two outer groups are zero by \eqref{eq:core}, proving that
$J_h^\partial$ is an isomorphism.  By
\eqref{eq:sigma-splitting}, the inclusion
$\cE_+\otimes L\hookrightarrow T\otimes L$ has a
bundle retraction, and hence induces an injection on cohomology.
Its composite with $J_h^\partial$ is $J_h$.
\end{proof}

\subsection{A pointed test family and the central source}

The preceding bundles were computed along the ordinary curve
$C$, before any odd deformation parameter was introduced.
We now construct the superscheme that will map to the boundary.
The nonzero vector $\eta\in U_C$ deforms the core over
an odd one-dimensional base.  Allowing an NS marking to vary
on this family gives a superscheme $X_\eta$ with reduction
$C$; its primary class will be the same deformation vector,
under an explicit identification of coefficients.

We then attach the tail.  For the even tail, the source includes
one further odd parameter moving its marking, using the
isomorphism $\delta_q$ of \cref{lem:tail-characters}.
For the odd tail, we keep the pointed tail constant.  We will
verify that the resulting map has a locally split differential
on the entire source, including its nilpotent structure, and
use \cref{lem:Jh-injection} to prove that its central comparison
class is nonzero.

Choose a one-dimensional complex vector space $S=\C s$,
place it in odd degree, and set
$$
 \mathbb A_S^{0|1}=\Spec\bigwedge S
 =\Spec(\C\oplus S).
$$
The multiplication satisfies $S^2=0$.  Thus $S$ is
precisely the conormal space of the reduced point
$\Spec\C\hookrightarrow\mathbb A_S^{0|1}$, and the
odd tangent of this base is $S^\vee$.
The tensor $s\otimes\eta$ specifies the
Kodaira--Spencer map $S^\vee\to U_C$ sending
$s^\vee$ to $\eta$.  It defines a first-order
family $f_\eta:\mathbb A_S^{0|1}\to\M_h^+$ through
$(C,L)$.  Since the base ideal is square-zero, this is
an actual family over the odd base.

Pull back the universal once-NS-pointed family.  Write
$\pi_\eta$ and $\kappa_\eta$ for the maps in the
$2$-Cartesian square
\begin{equation}\label{eq:pointed-core-square}
\begin{tikzcd}[column sep=large,row sep=large]
 X_\eta\arrow[r,"\kappa_\eta"]\arrow[d,"\pi_\eta"']
   &\M_{h,1}^+\arrow[d,"\mathrm{forget}"]\\
 \mathbb A_S^{0|1}\arrow[r,"f_\eta"']&\M_h^+.
\end{tikzcd}
\end{equation}
The right-hand morphism forgets the marking.  A point of
$X_\eta$ therefore specifies a member of the odd
deformation family together with a choice of NS marking on it.
The map $\pi_\eta$ is a smooth family of dimension
$1|1$ over the $0|1$ base, so $X_\eta$ is a
smooth superscheme of dimension $1|2$, with reduction
$C$.  We use the pointed test calculation of
\cite[Section~3.5 and Lemma~3.5]{DW}, recalling the cocycle
construction to make both its algebraicity and its coefficient
identifications explicit.

\begin{lemma}[The pointed-core test]\label{lem:pointed-core-test}
The odd cotangent bundle $\cF_\eta=T_-^\vee X_\eta$
fits into
\begin{equation}\label{eq:Feta}
 \begin{gathered}
 0\longrightarrow S\otimes\cO_C\longrightarrow\cF_\eta
 \longrightarrow L\longrightarrow0,\\
 [\cF_\eta]=\pm s\otimes\eta,
 \qquad \bigwedge^2\cF_\eta=S\otimes L.
 \end{gathered}
\end{equation}
Under $T_C\otimes\bigwedge^2\cF_\eta\simeq
S\otimes L^{-1}$, and with the primary-cocycle convention
of \cite[Section~3.5]{DW},
\begin{equation}\label{eq:DW-class}
 \omega_2(X_\eta)=s\otimes\eta
 \quad\text{in }H^1(C,T_C\otimes\bigwedge^2\cF_\eta).
\end{equation}
In particular this is a nonzero algebraic class.
\end{lemma}

\begin{proof}
We construct the family first, and then read its odd cotangent
bundle and primary cocycle from the same transition functions.
The split supercurve associated with $(C,L)$ has structure
sheaf $\cO_C\oplus L$, with $L$ in odd degree.
Choose affine \'etale superconformal charts
$(x\mid\theta)$, where $x$ is even and $\theta$
is an odd local generator of $L$.  Odd superconformal
vector fields form the sheaf $L^{-1}$.  Represent
$\eta$ by an algebraic cocycle in this sheaf, and write
its local coefficients as $g_{ij}(x)$, with all overlap
expressions transported to a common chart.

With the sign convention of \cite[Section~3.5]{DW},
multiplication by the odd base coordinate gives the
parity-preserving derivations
$$
 D_{ij}=-s\,g_{ij}(x)(\partial_\theta-\theta\partial_x).
$$
These derivations preserve the relative superconformal
structure and kill $s$.  Composing the split transition
functions with $1+D_{ij}$ constructs the deformation.
Every correction is divisible by $s$; hence
$D_{ij}D_{jk}=0$ after transport, because $s^2=0$.
The cocycle condition on triple overlaps is consequently
exactly the additive condition on the cocycle representing
$\eta$.  The inverse of a correction is $1-D_{ij}$.
This verifies gluing and constructs the algebraic family in
\eqref{eq:pointed-core-square}.

To identify the odd cotangent bundle, work modulo the square of
the odd ideal of $X_\eta$.  The base coordinate $s$
is unchanged on overlaps, so it spans the subline
$S\otimes\cO_C$.  Setting $s=0$ recovers the
original split supercurve and its odd cotangent $L$.
This proves the exact sequence in \eqref{eq:Feta}.
The term $-s g_{ij}\partial_\theta$ is the off-diagonal
entry of its transition matrix.  Thus its extension class in
$$
 \Ext^1_{\cO_C}(L,S\otimes\cO_C)
 =S\otimes H^1(C,L^{-1})
$$
is $\pm s\otimes\eta$, according to the convention for
transition cocycles.  Taking determinants gives the canonical
identification $\bigwedge^2\cF_\eta=S\otimes L$.

Finally, the quadratic-odd correction to the even coordinate is
$$
 s\,g_{ij}(x)\theta\,\partial_x.
$$
Since $T_C=L^{-2}$, the coefficient identification is
$$
 T_C\otimes\bigwedge^2\cF_\eta
 =T_C\otimes S\otimes L=S\otimes L^{-1}.
$$
Under this identification, the quadratic correction represents $s\otimes\eta$. 
Its coefficients lie in $T_C\otimes\bigwedge^2\cF_\eta$, where $\cF_\eta$ 
is the odd cotangent bundle determined above.  It proves \eqref{eq:DW-class}.
Because $s\ne0$ and $\eta\ne0$, the resulting class
is nonzero.  
\end{proof}

\emph{The even-tail family.}
For $\eps=+$, put $Y=U_E^\vee$.  The isomorphism
$\delta_q:R_E\to U_E$ of \cref{lem:tail-characters}
allows us to realise every odd pointed deformation by moving the
NS section on the fixed split elliptic supercurve.  More explicitly,
its dual is an equivariant isomorphism
$$
 \delta_q^\vee:Y\xrightarrow{\sim}M_q.
$$
The point over $\mathbb A_Y^{0|1}$ with reduction $q$
is defined by evaluation of the split elliptic structure sheaf
at $q$, followed by the superalgebra map
$$
 \C\oplus M_q\longrightarrow\C\oplus Y,
 \qquad c+m\longmapsto c+(\delta_q^\vee)^{-1}(m),
$$
 where $M_q$ and $Y$ are placed in odd degree.  Together
with the identity on the base, this gives the moving NS section
of the constant elliptic family.  Denote the resulting pointed
family by
$$
 e_Y:\mathbb A_Y^{0|1}\longrightarrow\M_{1,1}^+.
$$
Its point-variation differential is $\delta_q^{-1}:U_E\to R_E$,
so its pointed Kodaira--Spencer map is
$de_{Y,-}=\delta_q\delta_q^{-1}=\Id_{U_E}$, with the
point-variation convention of \eqref{eq:tail-evaluation-cohomology}.

This construction is coherently $A$-equivariant: the fixed
elliptic supercurve carries its spin lift, the character on the
base conormal $Y$ is $\rho^{-1}$, and both evaluation
and $\delta_q^\vee$ are equivariant.  Hence the moving
section is preserved by the simultaneous action on curve and base.
The induced comparison arrows on its moduli map satisfy the group
law because they come from this action on the marked family.

\smallskip
\noindent\emph{The odd-tail family.}
For $\eps=-$, the connecting map $\delta_q$ is
zero.  Moving the marking would therefore contribute no
pointed Kodaira--Spencer vector.  We use the constant pointed
family
$$
 e:\Spec\C\longrightarrow\M_{1,1}^-,
 \qquad e(*)=(E,\cO_E,q),
$$
with the chosen spin lift on the tail as its coherent
comparison action.  Its base has no odd tangent direction.
Consequently the pointed-tail deformation space $U_E$
will remain in the tangential quotient, with the character
$\rho^3$ computed above.

\smallskip
\noindent\emph{The central source and clutching.}
The two central sources and their odd cotangent bundles are
\begin{equation}\label{eq:central-sources}
 \begin{array}{c|c|c}
 \eps&Z_\eta^\eps&\cF^\eps=T_-^\vee Z_\eta^\eps\\\hline
 +&X_\eta\times\mathbb A_Y^{0|1}
   &\cF_\eta\oplus(Y\otimes\cO_C)\\
 -&X_\eta&\cF_\eta.
 \end{array}
\end{equation}
Both have reduced scheme $C$.  We call them central
because they will be the inverse image of the boundary in the
formal extensions.  The \emph{full central superscheme} means
$Z_\eta^\eps$ itself, with all its odd functions and their
products retained, rather than only its reduction $C$.
Locally it has one even coordinate and three odd coordinates
for $\eps=+$, or two odd coordinates for $\eps=-$.
In the following diagram, the upper map is $(\kappa_\eta,e_Y)$
for $\eps=+$ and $(\kappa_\eta,e)$ for $\eps=-$:
\begin{equation}\label{eq:central-clutching-diagram}
\begin{tikzcd}[column sep=large,row sep=large]
 Z_\eta^\eps\arrow[r]\arrow[dr,"F_0"']
 &\M_{h,1}^+\times\M_{1,1}^\eps
   \arrow[d,"\mathrm{clutch}"]\\
 &D\arrow[r,hook]&\mathfrak X_g^\eps.
\end{tikzcd}
\end{equation}
The group $A$ acts trivially on $X_\eta$, and on
the extra odd factor as specified above.  Clutching the equivariant
marked families gives coherent comparison arrows for $F_0$.
Here and in the application of \cref{cor:equivariant-formal},
we equip the target pair $(\mathfrak X_g^\eps,D)$ with
the trivial $A$-action.  Equivariance means that the family
automorphisms give isomorphisms $F_0\circ a\simeq F_0$,
for $a\in A$, satisfying the group law.  Although the action
on the target pair is trivial, these comparison isomorphisms
induce the linearisations of the pulled-back tangent bundles and
boundary conormal line.  The reduction of $F_0$ is $i$.  The total spin parity
is $\eps$: separating clutching adds the parities of the
two components, and parity is locally constant under smoothing.

\begin{lemma}[The full tangential quotient]\label{lem:central-differential}
Let $Z=Z_\eta^\eps$ be the superscheme in
\eqref{eq:central-sources}, and let $F_0:Z\to D$
be its clutching map.  On $Z$ there is a locally split
exact sequence of super vector bundles
$$
 0\longrightarrow T_Z\xrightarrow{\,dF_0\,}F_0^*T_D
 \longrightarrow\cQ\longrightarrow0.
$$
With $V_\eta=U_C/\C\eta$,
the parity summands of its reduction
$\cQ|_C=\cQ\otimes_{\cO_{Z_\eta^\eps}}\cO_C$ are
\begin{equation}\label{eq:Q-reduced}
 (\cQ|_C)_{\bar0}=W_0\otimes\cO_C,
 \qquad
 (\cQ|_C)_{\bar1}=
 \begin{cases}
 V_\eta\otimes\cO_C,&\eps=+,\\
 (V_\eta\oplus U_E)\otimes\cO_C,&\eps=-.
 \end{cases}
\end{equation}
These identifications are equivariant.  The characters on
$W_0,V_\eta$ are trivial, and the retained $U_E$ in
the odd case has character $\rho^3$.
\end{lemma}

\begin{proof}
We first compute the differential after restriction to $C$,
where the tangent bundles separate into ordinary even and odd
vector bundles.  For the odd part, dualising \eqref{eq:Feta}
and differentiating the Cartesian square
\eqref{eq:pointed-core-square} gives the commutative diagram
with exact rows
\begin{equation}\label{eq:odd-differential-diagram}
\begin{tikzcd}[column sep=small,row sep=large]
0\arrow[r]&L^{-1}\arrow[r]\arrow[d,equal]
 &\cF_\eta^\vee\arrow[r]\arrow[d,"d\kappa_{\eta,-}"]
 &S^\vee\otimes\cO_C\arrow[r]\arrow[d,"u_\eta"]&0\\
0\arrow[r]&L^{-1}\arrow[r]&\cE_{C,-}\arrow[r]
 &U_C\otimes\cO_C\arrow[r]&0.
\end{tikzcd}
\end{equation}
The left map is the identity on the tangent of the moving core
point.  With the chosen generator of $S$, the right map
sends $s^\vee$ to the Kodaira--Spencer vector $\eta$.
Since $\eta$ is non-zero, it follows that 
$u_\eta$ is an injection with constant cokernel
$V_\eta\otimes\cO_C$.  The snake lemma shows that the
middle map is injective with the same cokernel.

Clutching appends the constant tail summand $U_E$ to
$\cE_{C,-}$.  For the even tail the extra source tangent
is $Y^\vee=U_E$, and $de_{Y,-}$ maps it isomorphically
onto this summand.  For the odd tail the constant map contributes
no source tangent, so $U_E$ remains in the quotient.  This
proves the odd part of \eqref{eq:Q-reduced}.  The even differential
on $C$ is $\jmath_\partial:T_C\to\cE_+$, whose
cokernel is $W_0\otimes\cO_C$ by \eqref{eq:Eplus}.
All maps just used are equivariant, giving the asserted characters.

These calculations establish a locally split injection on
the reduction.  We must now show that they imply the full
bundle assertion on $Z$, which is the hypothesis required
by \cref{prop:transverse-extension}.  Let $\cJ$ be the
ideal generated by the odd functions on $Z$, so that
$\cO_Z/\cJ=\cO_C$ and $\cJ$ is nilpotent.
On a local
split chart of $Z_\eta^\eps$, choose homogeneous bases of
source and target tangent bundles.  The reduced images of the
source basis are part of a homogeneous target basis, since the
reduced cokernel just computed is locally free.  Complete them
to a basis and lift the complementary vectors.  The full images
under $dF_0$, together with those lifts, form an invertible
matrix modulo the nilpotent odd ideal.  They therefore form an
invertible matrix over the full ring: after lifting an inverse,
the remaining error is nilpotent and is inverted by a finite
geometric series.  In this target basis the differential has form
$$
 dF_0=\begin{pmatrix}\Id\\0\end{pmatrix}.
$$
It is consequently a locally split injection, and its cokernel
is locally free.  Its rank is
$(3h-2\mid2h-3)$ for $\eps=+$ and
$(3h-2\mid2h-2)$ for $\eps=-$, in agreement with
\eqref{eq:Q-reduced}.
\end{proof}

We can now combine the source calculation with the tangent
maps computed above.  To state the comparison precisely,
regard $F_0$ as the composite
$Z_\eta^\eps\to D\hookrightarrow\mathfrak X_g^\eps$,
with reduction $i$.  Its odd differential induces the
coefficient map
$$
 \begin{aligned}
 \iota_{F_0}:
 i^*\Hom\!\left(\bigwedge^2T_-\mathfrak X_g^\eps,
                         T_+\mathfrak X_g^\eps\right)
 &\longrightarrow T\otimes\bigwedge^2\cF^\eps,\\
 \beta&\longmapsto\beta\circ\bigwedge^2dF_{0,-}.
 \end{aligned}
$$
Define the central comparison class by
$$
 c_\eta^\eps:=
 H^1(\iota_{F_0})
 \bigl(i^*\omega_2(\mathfrak X_g^\eps)\bigr)
 \in H^1(C,T\otimes\bigwedge^2\cF^\eps).
$$
\Needspace{14\baselineskip}
\begin{proposition}[Nonzero central comparison]\label{prop:central-class}
The central comparison class satisfies
$$
 c_\eta^\eps=s\otimes J_h(\eta)\ne0,
$$
where the right-hand side is included through the direct
summand
$\bigwedge^2\cF_\eta=S\otimes L
\subset\bigwedge^2\cF^\eps$.
More precisely, $c_\eta^\eps$ is the image of
$$
 s\otimes J_h^\partial(\eta)
 \in H^1(C,S\otimes\cE_+\otimes L)
$$
under the boundary-tangent inclusion and this exterior-square
inclusion.  In particular $c_\eta^\eps$ is $A$-invariant.
\end{proposition}

\begin{proof}
We first identify the primary class of the source, and then
apply functoriality to its map into the ambient superstack.
For $\eps=-$, the source is $X_\eta$, with primary
class $s\otimes\eta$ by \cref{lem:pointed-core-test}.
For $\eps=+$, its product with a constant split odd line
has transitions independent of the extra odd coordinate.  Its
primary class is therefore the image of the same class under
the direct-summand inclusion
$$
 \bigwedge^2\cF_\eta
 \hookrightarrow\bigwedge^2\cF^+
 =\bigwedge^2\cF_\eta\oplus(\cF_\eta\otimes Y).
$$
Thus in both cases the source primary class has the same
nonzero component $s\otimes\eta$, and in the even case
its component involving $Y$ is zero.
The even differential of the central map is the composite
$\jmath:T_C\to\cE_+\to T$ defined above.
The primary-class compatibility \eqref{eq:primary-naturality}
therefore reads
$$
 H^1(\iota_{F_0})\bigl(i^*\omega_2(\mathfrak X_g^\eps)\bigr)
 =H^1(\jmath\otimes\Id_{\wedge^2\cF^\eps})
      \bigl(\omega_2(Z_\eta^\eps)\bigr)
 =s\otimes J_h(\eta).
$$
To check nonvanishing, follow the coefficient maps separately.
First, \cref{lem:Jh-injection} gives an isomorphism
$J_h^\partial$, so
$s\otimes J_h^\partial(\eta)\ne0$.
Second, the boundary-tangent inclusion has the retraction
given by \eqref{eq:sigma-splitting}.  Finally, the
exterior-square inclusion is the identity for the odd tail
and a direct-summand inclusion for the even tail.  Both
inclusions therefore induce injections on cohomology.
Their composite cannot kill the class.

The action of $A$ is trivial on $S,\cE_+$, and $L$.
The boundary-tangent summand and the exterior-square inclusion
are equivariant, so their image $c_\eta^\eps$ is invariant.
\end{proof}

\begin{remark}[{Why we wag the even tail}]\label{rem:wag-tail}
The reason for moving the even tail is an obstruction
calculation, rather than a need to obtain a nonzero central
class.  Suppose we kept an even pointed tail constant and used
only $X_\eta$ as source.  The odd tangent quotient on
$C$ would then retain $U_E$.  In the order-one
extension coefficient, the degree-one piece of the odd-ideal
filtration would consequently contain
$$
 U_E\otimes\cF_\eta\otimes L\otimes R_E^{-1}.
$$
Its $A$-character would be
$\rho\cdot\rho^{-1}=1$.  Moreover, tensoring
\eqref{eq:Feta} by $L$ and using $H^1(C,L)=0$
gives
$$
 H^1(C,\cF_\eta\otimes L)\simeq H^1(C,L^2)
 =H^1(C,K_C)\simeq\C.
$$
Thus our calculation with graded cohomology would encounter
a nonzero invariant line.  This shows why
the vanishing criterion used here would not follow from the
tail characters and the core ineffectivity alone.

For an even tail, $\delta_q:R_E\to U_E$ is an
isomorphism.  The moving NS section therefore adds to the
source a tangent direction mapping onto the whole of $U_E$,
and removes this summand from $\cQ|_C$.
It also adds the conormal $Y=U_E^\vee$ to the source.
Both effects are included in the complete calculation of
Section~\ref{sec:vanishing}; they leave no invariant graded
first cohomology at the first two orders.

For an odd tail, a movement of the marking is absorbed by
an unpointed odd infinitesimal automorphism, so $\delta_q=0$.
We therefore keep the pointed tail constant.  Its retained
$U_E$ has character $\rho^3$, and the corresponding
order-one coefficient has character
$\rho^3\rho^{-1}=-1$, so the problematic invariant
line above does not occur.  Thus the two choices of source
serve the same obstruction-vanishing argument by different
representations.
\end{remark}

\Needspace{10\baselineskip}
\section{The obstruction calculation and formal smoothing}
\label{sec:vanishing}

Section~\ref{sec:boundary} constructed a map
$F_0:Z_\eta^\eps\to D$ from a smooth central superscheme
into the separating boundary of $\mathfrak X_g^\eps$, where
$g=h+1$.  Its reduction is the curve $C$, and its
comparison class is nonzero by \cref{prop:central-class}.
Moreover, \cref{lem:central-differential} showed that the
differential of $F_0$ is a locally split injection, with
locally free quotient $\cQ$.  These are the initial data
for the transverse extension problem of Section~\ref{subsec:extension-problem}.
We now construct a formal thickening of this source and extend
$F_0$ in a direction that smooths the attaching node.

The source thickening and its map will be chosen together, as in
\cref{def:transverse}.  Thus the remaining task is to show that
the obstruction groups of \cref{cor:equivariant-formal} have
no invariant classes.  We first identify their coefficients
and compute their associated graded separately for the two
tails.  A common vanishing argument then proves
\cref{thm:formal-map}, the main result of this section.
This produces the formal smoothing needed in
Section~\ref{sec:Laurent}; the survival of the comparison class
after leaving the boundary will be proved there.

Write $Z=Z_\eta^\eps$.  Let $\cI_D$ be the ideal
of the boundary divisor $D$ in $\mathfrak X_g^\eps$.
Its conormal line bundle $\cI_D/\cI_D^2$ is a bundle
on $D$, so the central map defines an even line bundle
on the full superscheme $Z$:
$$
 \widetilde\cL=F_0^*(\cI_D/\cI_D^2).
$$
This is the prescribed conormal line bundle of $Z$ in the
formal extension we seek.  Its restriction to the reduced curve
is the dual of the smoothing normal line computed in
\cref{prop:boundary-bundles}:
$$
 \widetilde\cL|_C=\cL=N^\vee=L\otimes R_E^{-1},
 \qquad N=L^{-1}\otimes R_E.
$$
Here $R_E=M_q^{-1}$ is the one-dimensional tail spin
fibre introduced in Section~\ref{sec:boundary}.  Distinguishing
$\widetilde\cL$ from $\cL$ matters: the lifting
problem is posed on $Z$, including its odd nilpotents,
whereas the cohomology calculation below will use bundles on $C$.

For $r\geq1$, put
$$
 \cB_r=\cQ\otimes\widetilde\cL^r.
$$
For a chosen order-$(r-1)$ map,
\cref{prop:transverse-extension} places the next obstruction
in $H^1(Z,(\cB_r)_{\bar0})$.  The factor $\cQ$
records changes of the map modulo changes of source coordinates;
$\widetilde\cL^r$ records their normal order.  In a local
boundary parameter $t$, an order-$r$ difference is a
coefficient multiplying $t^r$.  

We compute these groups using a second, finite filtration.  Let
$\cJ$ be the ideal generated by the odd functions of $Z$,
and let $\cF^\eps=\cJ/\cJ^2$ be its odd cotangent
bundle on $C$.  Smoothness of $Z$ gives
$\gr_\cJ\cO_Z=\bigwedge^\bullet\cF^\eps$, and
local freeness of $\cB_r$ gives
\begin{equation}\label{eq:odd-filtration}
 \gr^k_\cJ\cB_r
 \simeq\cQ|_C\otimes\bigwedge^k\cF^\eps\otimes\cL^r.
\end{equation}
Here $r$ is the normal order of the extension problem,
whereas $k$ counts odd factors on the central source.  We
hold $r$ fixed while passing through the finitely many
values of $k$.  This canonical associated-graded formula
requires no splitting of $Z$ or of $\cQ$.  Its total
even part uses $(\cQ|_C)_{\bar0}$ when $k$ is even
and $(\cQ|_C)_{\bar1}$ when $k$ is odd.

Recall also that $A=\langle a\rangle\simeq\boldsymbol\mu_4$
acts trivially on $C,L,S,\cF_\eta,W_0,V_\eta$, and that
$\rho=\chi(R_E)$ satisfies $\rho^2=-1$.
The character on $\cL=L\otimes R_E^{-1}$ is therefore
$\rho^{-1}$.  In \eqref{eq:odd-filtration}, the factor
$\cL^r$ contributes $\rho^{-r}$.

For clarity, we fix a notation convention for the two coefficient
tables \eqref{eq:even-obstructions} and \eqref{eq:odd-obstructions}
below.  A constant vector space denotes its tensor product with
$\cO_C$.  Each entry in the bundle column must also be
tensored with $R_E^{-r}$, which is suppressed there to
keep the table readable.  The character column includes this
factor, as well as the character of $Y$ or $U_E$
when present.  Rows with the same value of $k$ are the
direct summands of the corresponding total-even graded piece.

\subsection{The even tail}

For an even tail, the central source of
\eqref{eq:central-sources} is
$Z_\eta^+=X_\eta\times\mathbb A_Y^{0|1}$, where
$Y=U_E^\vee$ and $\chi(Y)=\rho^{-1}$.
The additional odd parameter moves the marked point on the
tail.  Its differential removes $U_E$ from the tangent
quotient, while its functions add $Y$ to the odd cotangent
bundle of the central source.  Both effects must be included
in \eqref{eq:odd-filtration}.

More explicitly, \cref{lem:central-differential} gives
$$
 \cF^+=\cF_\eta\oplus Y,\qquad
 (\cQ|_C)_{\bar0}=W_0\otimes\cO_C,\qquad
 (\cQ|_C)_{\bar1}=V_\eta\otimes\cO_C,
$$
where $V_\eta=U_C/\C\eta$ and $W_0=W_C\oplus W_E$.
Since $\bigwedge^2\cF_\eta=S\otimes L$, the exterior
powers of the rank-three odd cotangent bundle are
\begin{align*}
 \bigwedge^0\cF^+&=\cO_C,&
 \bigwedge^1\cF^+&=\cF_\eta\oplus Y,\\
 \bigwedge^2\cF^+&=(S\otimes L)\oplus(\cF_\eta\otimes Y),&
 \bigwedge^3\cF^+&=S\otimes L\otimes Y.
\end{align*}
Higher exterior powers vanish.  Substituting these decompositions
into \eqref{eq:odd-filtration} gives all the total-even graded
pieces, with the convention about $R_E^{-r}$ just stated:
\begin{equation}\label{eq:even-obstructions}
\begin{array}{c|l|c}
 k&\text{bundle factor on }C& A\text{-character}\\\hline
 0&W_0\otimes L^r&\rho^{-r}\\
 1&V_\eta\otimes\cF_\eta\otimes L^r&\rho^{-r}\\
 1&V_\eta\otimes Y\otimes L^r&\rho^{-(r+1)}\\
 2&W_0\otimes S\otimes L^{r+1}&\rho^{-r}\\
 2&W_0\otimes\cF_\eta\otimes Y\otimes L^r&\rho^{-(r+1)}\\
 3&V_\eta\otimes S\otimes Y\otimes L^{r+1}&\rho^{-(r+1)}.
\end{array}
\end{equation}
For example, the second summand at $k=1$ is actually
$V_\eta\otimes Y\otimes L^r\otimes R_E^{-r}$.
Its character is $\rho^{-1}\rho^{-r}=\rho^{-(r+1)}$.
The degree-three term uses the odd quotient $V_\eta$:
its parity is odd plus odd, hence even.  It is therefore part
of the obstruction coefficient, despite its odd filtration degree.

\subsection{The odd tail}

For an odd tail, the central source is instead
$Z_\eta^-=X_\eta$, and the pointed tail family is
constant.  There is no additional source parameter to remove
the tail deformation direction $U_E$.  Thus the {odd cotangent bundle} has rank two, while the odd tangent quotient
has the extra summand $U_E$.  

In formulas,
$$
 \cF^-=\cF_\eta,\qquad
 (\cQ|_C)_{\bar0}=W_0\otimes\cO_C,\qquad
 (\cQ|_C)_{\bar1}=(V_\eta\oplus U_E)\otimes\cO_C.
$$
The exterior powers of the odd cotangent bundle are
$$
 \bigwedge^0\cF^-=\cO_C,\qquad
 \bigwedge^1\cF^-=\cF_\eta,\qquad
 \bigwedge^2\cF^-=S\otimes L,
$$
and the higher ones vanish.  By \cref{lem:tail-characters}, the
retained $U_E$ has character $\rho^3$.
Consequently \eqref{eq:odd-filtration} gives the complete table
\begin{equation}\label{eq:odd-obstructions}
\begin{array}{c|l|c}
 k&\text{bundle factor on }C& A\text{-character}\\\hline
 0&W_0\otimes L^r&\rho^{-r}\\
 1&V_\eta\otimes\cF_\eta\otimes L^r&\rho^{-r}\\
 1&U_E\otimes\cF_\eta\otimes L^r&\rho^{3-r}\\
 2&W_0\otimes S\otimes L^{r+1}&\rho^{-r}.
\end{array}
\end{equation}
In the third row, the two character factors are
$\rho^3$ from $U_E$ and $\rho^{-r}$ from
the conormal power.  In particular, although the extra terms
in the two tables have different geometric origins, their
characters agree: $\rho^{3-r}=\rho^{-(r+1)}$.
This is why the same low-order argument will apply to both tails.

\subsection{Common vanishing and formal extension}

The tables describe associated-graded bundles on $C$.
The extension criterion, however, requires a statement about
$H^1(Z,(\cB_r)_{\bar0})$ on the full central superscheme.
We first prove vanishing for the graded pieces and then recover
the required statement by the finite odd-ideal filtration.
At normal orders one and two, the cohomology groups need not
vanish, but their characters exclude invariant classes.
From order three onward, every relevant line-bundle factor has
degree greater than $2h-2$, so its entire first cohomology
vanishes.

Once this has been established, \cref{cor:equivariant-formal}
constructs compatible transverse extensions at every order,
starting from the central map alone.  The result is the formal
smoothing of \cref{thm:formal-map}, including the local
algebraic maps obtained by inverting the boundary parameter
in the completed source rings.

\begin{lemma}[Vanishing at every order]\label{lem:all-orders}
For the central data of Section~\ref{sec:boundary}, with
$h\geq2$, $0\ne\eta\in U_C$, and either parity,
the full obstruction coefficients satisfy
\begin{align*}
 H^1\bigl(Z_\eta^\eps,(\cQ\otimes\widetilde\cL^r)_{\bar0}\bigr)^A
 &=0 &&(r=1,2),\\
 H^1\bigl(Z_\eta^\eps,(\cQ\otimes\widetilde\cL^r)_{\bar0}\bigr)
 &=0 &&(r\geq3).
\end{align*}
In particular, the invariant obstruction group vanishes for
every $r\geq1$.
\end{lemma}

\begin{proof}
We begin with normal orders one and two.  Because $\rho^4=1$,
the complete lists of characters in
\eqref{eq:even-obstructions} and \eqref{eq:odd-obstructions}
reduce to
$$
 \begin{array}{c|c|c}
 r&\eps=+&\eps=-\\\hline
 1&\rho^3,\rho^2&\rho^3,\rho^2\\
 2&\rho^2,\rho&\rho^2,\rho.
 \end{array}
$$
Since $\rho^2=-1$, none of these characters is trivial.
The group $A$ fixes $C$ and acts on each bundle
summand by the displayed scalar character.  It acts by the same
scalar on its cohomology: this follows, for example, by taking
the \v{C}ech complex of an affine cover of $C$, on which
the action is scalar in every degree.  Thus every total-even
graded piece has zero invariant first cohomology for $r=1,2$.

For $r\geq3$, we can prove the stronger vanishing without
using the action.  The only nonconstant bundles in the tables
are powers of $L$ and twists of $\cF_\eta$.
Tensoring \eqref{eq:Feta} with $L^r$ gives
$$
 0\longrightarrow S\otimes L^r
 \longrightarrow\cF_\eta\otimes L^r
 \longrightarrow L^{r+1}\longrightarrow0.
$$
Hence each table entry is either a constant vector space tensored
with $L^r$ or $L^{r+1}$, or an extension of two
such bundles.  Recall that $L^2\simeq K_C$ and
$\deg L=h-1>0$.  For every $m\geq3$, Serre
duality gives
\begin{equation}\label{eq:degree-vanishing}
 H^1(C,L^m)^\vee
 =H^0(C,K_C\otimes L^{-m})
 =H^0(C,L^{2-m})=0,
\end{equation}
because $L^{2-m}$ has negative degree.  The cohomology
sequence of the displayed extension therefore also gives
$H^1(C,\cF_\eta\otimes L^r)=0$.
Tensoring by any of the constant factors preserves these
vanishings.  Every total-even graded piece consequently has
zero first cohomology when $r\geq3$.
In particular, the odd-tail term containing $U_E$ has
trivial character at $r=3$, but its entire first cohomology
already vanishes by this degree calculation.

It remains to transfer these conclusions from the associated
graded to the actual coefficient on $Z$.  Let
$\nu_+=3$ and $\nu_-=2$ be the odd dimensions
of the two central sources, and set
$$
 \cB_r^k=(\cJ^k\cB_r)_{\bar0},\qquad
 G_r^k=(\gr^k_\cJ\cB_r)_{\bar0}.
$$
The filtration is finite, with $\cB_r^{\nu_\eps+1}=0$,
and has equivariant exact sequences
\begin{equation}\label{eq:full-coefficient-filtration}
 0\longrightarrow\cB_r^{k+1}\longrightarrow\cB_r^k
 \longrightarrow G_r^k\longrightarrow0.
\end{equation}
The sheaf $G_r^k$ is annihilated by $\cJ$, so its
cohomology is that of the bundle on $C$ listed in the
tables.  For $r=1,2$, the long exact sequence followed
by finite-group invariants contains the exact segment
$$
 H^1(Z,\cB_r^{k+1})^A
 \longrightarrow H^1(Z,\cB_r^k)^A
 \longrightarrow H^1(C,G_r^k)^A.
$$
Here taking invariants is exact because averaging over the finite
group is available in characteristic zero.  The right-hand group
vanishes by the character calculation.  Starting with
$\cB_r^{\nu_\eps+1}=0$ and descending in $k$,
we obtain the vanishing of the middle group at every stage.
For $r\geq3$, the same argument works without taking
invariants, using the degree vanishing for $G_r^k$.
Finally, $\cB_r^0=(\cQ\otimes\widetilde\cL^r)_{\bar0}$,
which proves both assertions.  
\end{proof}

\begin{remark}[The role of the variable source]\label{rem:no-cubic}
The coefficient in this lemma is $\cQ\otimes\widetilde\cL^r$,
where $\cQ=F_0^*T_D/dF_0(T_Z)$.  In particular, the
moving-core-point directions $T_C$ and $L^{-1}$
have already been removed, as has the moving-tail direction
in the even case.  {This is the quotient arising from
source reparametrisation in \cref{prop:transverse-extension}.}

The lemma also includes the first-order obstruction. The class of the normal
sequence in \cref{lem:first-normal} is invariant and is killed
by the calculation at $r=1$.  Both low orders use the
full group $A$, not just its involution $\sigma=a^2$:
the nontrivial character $\rho^2$ becomes trivial on
$\langle\sigma\rangle$.  From order three onward,
$L^2\simeq K_C$ and $\deg L>0$ suffice to kill
the entire obstruction group.  
\end{remark}

We now apply the vanishing to the central map.  The following
theorem is the main conclusion of the section: it constructs
the formal source together with its map, and gives the local
maps to smooth supermoduli that will be used in the Laurent
comparison.

\begin{theorem}[Uniform transverse formal smoothing]\label{thm:formal-map}
Let $h\geq2$, $g=h+1$, $0\ne\eta\in U_C$,
and $\eps\in\{+,-\}$, with the core, tail, and central
map $F_0:Z_\eta^\eps\to D\subset\mathfrak X_g^\eps$
constructed in Section~\ref{sec:boundary}.  Here $D$ is
the divisor of curves with the selected separating NS node,
and $\mathfrak X_g^\eps\setminus D=\M_g^\eps$.
There exists a coherently $A$-equivariant transverse formal
extension
$$
 \widehat F:\widehat{\mathfrak Y}_\eta^\eps
 \longrightarrow\widehat{\mathfrak X_g^\eps}_D,
$$
where the target is the formal completion along $D$,
and $\widehat{\mathfrak Y}_\eta^\eps$ is a formal
superscheme containing $Z_\eta^\eps$.  Its central
restriction is specified by the $2$-Cartesian square
$$
\begin{tikzcd}[column sep=large,row sep=large]
 Z_\eta^\eps \arrow[r,hook] \arrow[d,"F_0"']
 &\widehat{\mathfrak Y}_\eta^\eps \arrow[d,"\widehat F"]\\
 D \arrow[r,hook]&\widehat{\mathfrak X_g^\eps}_D.
\end{tikzcd}
$$
The source is locally $\Spf B[[t]]$, where
$\Spec B$ is an affine chart of $Z_\eta^\eps$
and $t$ is one new even formal parameter.  Its local
coordinate dimensions are $2|3$ for $\eps=+$ and
$2|2$ for $\eps=-$.  The pulled-back boundary
ideal $\cI$ is Cartier and satisfies
$$
 \cO_{\widehat{\mathfrak Y}_\eta^\eps}/\cI
 \simeq\cO_{Z_\eta^\eps},\qquad
 \cI^r/\cI^{r+1}\simeq\widetilde\cL^r
 \quad(r\geq1).
$$

More explicitly, on a compatible affine \'etale target chart
$V=\Spec R\to\mathfrak X_g^\eps$, write
$D_V=D\times_{\mathfrak X_g^\eps}V=(u=0)$.
On an affine formal source chart mapping to its completion,
the induced homomorphism satisfies
$$
 \varphi:R\longrightarrow B[[t]],\qquad
 \varphi(u)=e\,t,
 \qquad e\in\bigl(B[[t]]_{\bar0}\bigr)^\times.
$$
It therefore induces a morphism of superschemes to the
smooth-curve locus, followed by its atlas map:
$$
 \Spec\!\bigl(B[[t]][t^{-1}]\bigr)
 \longrightarrow V\setminus D_V
 \longrightarrow\M_g^\eps.
$$
\end{theorem}

\begin{proof}
We first check that the central map satisfies the hypotheses of
\cref{cor:equivariant-formal}.  In the notation of that
corollary, the data are
$$
 (f,Z,Q_f,A_f,H)
 =(F_0,Z_\eta^\eps,\cQ,\widetilde\cL,A).
$$
The smoothness of the target and its Cartier boundary divisor
were recorded in \cref{prop:boundary-bundles}.  The differential
on the full central superscheme is a locally split injection
by \cref{lem:central-differential}.  The source action and
the coherent comparison arrows for $F_0$ were constructed
in \eqref{eq:central-clutching-diagram}, with the trivial action
on the target pair as specified there.  Finally,
\cref{lem:all-orders} proves
$$
 H^1\bigl(Z_\eta^\eps,
       (\cQ\otimes\widetilde\cL^r)_{\bar0}\bigr)^A=0
 \qquad\text{for every }r\geq1,
$$
which is exactly \eqref{eq:all-order-criterion}.

To spell out the resulting induction, start with
$\mathfrak Y_{[0]}=Z_\eta^\eps$ and the map $F_0$.
At $r=1$, the obstruction is the extension class of the
normal sequence in \cref{lem:first-normal}.  It is invariant,
so the first-order vanishing just proved makes it zero.
If an equivariant order-$(r-1)$ extension has
been chosen, \cref{cor:equivariant-formal} places its next
obstruction in the displayed zero invariant group and supplies
an equivariant order-$r$ extension.  The source action and
comparison $2$-isomorphisms extend coherently by that
corollary.  We obtain maps
$$
 F_{[r]}:\mathfrak Y_{[r]}\longrightarrow\mathfrak X_g^\eps,
 \qquad F_{[r]}|_{\mathfrak Y_{[r-1]}}=F_{[r-1]},
$$
with the prescribed identifications of each preceding order.
These restrictions are understood with their compatible
comparison isomorphisms.

The fixed-affine construction in the proof of
\cref{cor:equivariant-formal} applies to these thickenings:
on $U=\Spec B\subset Z_\eta^\eps$ trivialising
the full conormal line, it supplies compatible sections and
parameters identifying their rings with $B[t]/(t^{r+1})$.
{Their inverse limit is $B[[t]]$; the compatible
transition maps and target arrows give
$\widehat{\mathfrak Y}_\eta^\eps$ and $\widehat F$.}

The boundary ideal is locally $(t)$, and multiplication
by $t$ is injective in $B[[t]]$, even when $B$
has odd nilpotents.  Hence $\cI$ is Cartier, and its
zero locus is exactly the full central superscheme.  This
proves the Cartesian assertion.  Pullback of the boundary
conormal identifies $\cI/\cI^2$ with
$\widetilde\cL$; multiplication then gives
$$
 \cI^r/\cI^{r+1}
 \simeq(\cI/\cI^2)^{\otimes r}
 \simeq\widetilde\cL^r.
$$
Finally, adjoining the single even formal variable to a source
of dimension $1|3$, respectively $1|2$, gives the
stated local coordinate dimensions.

For the localised maps, choose an affine target chart
$V=\Spec R$ and a compatible source chart as in the
statement.  The formal map gives a continuous homomorphism
$$
 \widehat R_{(u)}=\varprojlim_r R/(u^{r+1})
 \longrightarrow B[[t]].
$$
Composing with $R\to\widehat R_{(u)}$ gives
$\varphi:R\to B[[t]]$.  Since $u$ and $t$
generate the target and pulled-back boundary ideals, respectively,
$\varphi(u)=e\,t$ for an invertible even element $e$.
After inverting $t$, the image of $u$ is invertible,
so the universal property of localisation gives
$$
 R[u^{-1}]\longrightarrow B[[t]][t^{-1}].
$$
Taking spectra gives the asserted map to $V\setminus D_V$.
The selected NS node is the only singularity permitted in
$\mathfrak X_g^\eps$, so removing $D$ leaves
precisely $\M_g^\eps$.
\end{proof}

\begin{remark}[Localisation after completion]\label{rem:formal-localisation}
The final part of \cref{thm:formal-map} uses the algebraic
spectrum of a completed ring and then inverts the boundary
parameter.  It does not take an open complement inside
$\Spf B[[t]]$: the latter has its underlying points
on $t=0$.  Nor can one invert $t$ at a finite
order, where it is nilpotent.  We first form the inverse limit
and then localise:
$$
 \varprojlim_r B[t]/(t^{r+1})=B[[t]]
 \quad\longrightarrow\quad B[[t]][t^{-1}],
$$
where the arrow is the algebraic localisation map.
Each element of the last ring has a
finite lower bound on its powers of $t$.  Replacing a
local generator $t$ by an invertible multiple gives the
same localisation, so a global choice of parameter is unnecessary.

These local algebraic maps require neither algebraisation of
the formal source nor analytic convergence of its transition
functions.  They allow algebraic data on the smooth locus to
be pulled back to the localised completed rings.
Section~\ref{sec:Laurent} develops the corresponding complete
and localised coefficient sheaves and proves that the central
comparison class survives localisation.
\end{remark}

\section{The comparison class after Laurent localisation}
\label{sec:Laurent}

Section~\ref{sec:boundary} produced the nonzero central
comparison class $c_\eta^\eps$, and
\cref{thm:formal-map} extended its source and map transversely
to the boundary.  We must now show that this class can still
be detected after the boundary parameter is inverted.
A nonzero class on the central curve does not by itself give
this conclusion: a cocycle that has no regular primitive could
acquire a primitive with poles after localisation.

Our argument rules out precisely this possibility.  We represent
the comparison by a regular cocycle on the actual formal source.
If it had a Laurent primitive, we could average that primitive
under tail parity.  Its first possible polar coefficient would
then define an invariant global section of an associated-graded
bundle on $C$.  We prove that all such negative-order
sections vanish.  Removing the possible poles one order at a
time would therefore produce a regular primitive, contradicting
the nonzero central class.

Section~\ref{subsec:complete-coefficients} constructs the
complete and localised coefficient modules and identifies their
boundary filtration.  Section~\ref{subsec:negative-coefficients}
proves the required vanishing of negative invariant sections.
Section~\ref{subsec:completed-cocycle} constructs compatible
local splittings and their cocycle, and identifies its reduction
with $c_\eta^\eps$.  Finally,
Section~\ref{subsec:Laurent-survival} gives the filtered
pole-removal argument and proves Laurent survival.
Section~\ref{sec:algebraic} will use this result to exclude
an algebraic splitting on smooth supermoduli.

\subsection{Complete coefficient modules}
\label{subsec:complete-coefficients}

We first specify where the comparison cocycle and a possible
primitive will live.  At every finite normal order, the odd
differential of the extended map changes the coefficient of the
target primary class to a vector bundle on the source's ordinary
thickening.  Taking compatible inverse limits produces the
complete coefficient module; inverting the boundary ideal then
allows finite-order poles.  Its associated graded is computable
from the central bundles, even though the formal source need
not be the completion of a line bundle over $C$.

Let $\widehat{\mathfrak Y}_\eta^\eps$,
$\widehat F$, and $\cI$ be the formal source,
map, and pulled-back Cartier boundary ideal of
\cref{thm:formal-map}.  The zero locus of $\cI$ is
the full central superscheme $Z_\eta^\eps$, whose
reduction is $C$.  Let $\cJ$ be the ideal generated
by odd functions on the formal source, and put
$$
 \widehat\cA=\cO_{\widehat{\mathfrak Y}_\eta^\eps}/\cJ,
 \qquad \cA_{[a]}=\widehat\cA/\cI^{a+1}
 \quad(a\geq0).
$$
In the second expression, and throughout this section, $\cI$
also denotes its image in the bosonic structure sheaf
$\widehat\cA$.  We have
$\cA_{[0]}=\cO_C$ and
$\widehat\cA\simeq\varprojlim_a\cA_{[a]}$.
{On an adapted open of $C$, these rings are locally
$\cO_C[[t]]$ and $\cO_C[t]/(t^{a+1})$, respectively.}

Write $\mathfrak Y_{[a]}$ for the order-$a$
source thickening, defined by $\cI^{a+1}$, and
$F_{[a]}$ for the corresponding map constructed in
\cref{thm:formal-map}.  Its bosonic truncation is an ordinary
scheme $(C,\cA_{[a]})$, with reduced scheme $C$.
The induced map is
\begin{equation} \label{fa}
f_a:(C,\cA_{[a]})\longrightarrow
 X:=(\mathfrak X_g^\eps)_{\red}.
\end{equation}

Let $\cJ_{[a]}$ be the odd ideal of
$\mathfrak Y_{[a]}$.  Its odd cotangent bundle and its
dual are
$$
 \cF_{[a]}=\cJ_{[a]}/\cJ_{[a]}^2,
 \qquad
 \cV_{[a]}=\Hom_{\cA_{[a]}}(\cF_{[a]},\cA_{[a]}).
$$
These are locally free $\cA_{[a]}$-modules, of rank
three for $\eps=+$ and two for $\eps=-$.
At order zero they are $\cF^\eps$ and
$(\cF^\eps)^\vee$, respectively.  The morphism
$F_{[a]}$ induces a map of odd cotangent bundles; dualising
it gives the odd differential
\begin{equation} \label{dF}
 dF_{a,-}:\cV_{[a]}\longrightarrow
 f_a^*T_-\mathfrak X_g^\eps.
\end{equation}

The comparison of primary classes in
\eqref{eq:primary-naturality} takes values in homomorphisms
from an exterior square of source odd tangent directions to
target even tangent directions.  This motivates the coefficient
at each order,
which is the specialization to our situation of \eqref{mixed},
and its two successive enlargements:
\begin{equation}\label{eq:comparison-module}
\begin{split}
 \cG_{[a]}&=\Hom_{\cA_{[a]}}
 \left(\bigwedge^2\cV_{[a]},f_a^*T_+\mathfrak X_g^\eps\right),\\
 \widehat\cG&=\varprojlim_a\cG_{[a]},
 \qquad
 \cG^\times=\widehat\cG\otimes_{\widehat\cA}\cA^\times,
 \qquad
 \cA^\times=\widehat\cA[\cI^{-1}].
\end{split}
\end{equation}
All these objects are viewed as sheaves on the underlying space
$C$.  A section of $\widehat\cG$ is called
\emph{regular} here: it is defined through all formal orders
with no pole along the boundary.  A section of $\cG^\times$
is a \emph{Laurent section}.  Locally it is a section of
$\widehat\cG$ divided by a finite power of $t$.
{The order of these operations is explained in
\cref{rem:formal-localisation}.}

\smallskip

The following table distinguishes the bundles on the central
source from the coefficient modules on its thickenings.  The
quotient $\cQ$ controls transverse extensions; the modules
$\cG_{[a]}$ compare primary classes after such extensions
have been constructed.  All notation in the table has been
defined above or in Sections~\ref{sec:boundary}--\ref{sec:vanishing}.
\medskip
\begin{center}
\small
\renewcommand{\arraystretch}{1.16}
\begin{tabularx}{\linewidth}{@{}>{\raggedright\arraybackslash}p{.18\linewidth}
 >{\raggedright\arraybackslash}p{.21\linewidth}
 >{\raggedright\arraybackslash}X@{}}
\toprule
Object & Base or coefficient ring & Role \\
\midrule
$L$ & $C$ & Ineffective theta characteristic of the core. \\
$\cL$, $\widetilde\cL$ & $C$, $Z_\eta^\eps$ &
Boundary conormal line bundles, with
$\cL=\widetilde\cL|_C=L\otimes R_E^{-1}$. \\
$\cF_\eta$, $\cF^\eps$ & $C$ &
Odd cotangent bundles of $X_\eta$ and $Z_\eta^\eps$. \\
$\cF_{[a]}$ & $\cA_{[a]}$ &
Odd cotangent bundle of $\mathfrak Y_{[a]}$. \\
$\cQ$ & $Z_\eta^\eps$ &
Tangential quotient $F_0^*T_D/dF_0(T_{Z_\eta^\eps})$. \\
$\cG_{[a]}$ & $\cA_{[a]}$ &
Primary-class comparison coefficient at order $a$. \\
$\widehat\cG$, $\cG^\times$ &
$\widehat\cA$, $\cA^\times$ &
Complete and Laurent comparison coefficients. \\
\bottomrule
\end{tabularx}
\end{center}

\begin{lemma}[The boundary filtration]\label{lem:complete-filtration}
The module $\widehat\cG$ is locally free of finite rank
over $\widehat\cA$, and
$\widehat\cG/\cI^{a+1}\widehat\cG\simeq\cG_{[a]}$
for every $a\geq0$.  In particular, it is complete and
separated for the boundary filtration:
$$
 \widehat\cG\simeq
 \varprojlim_a\bigl(\widehat\cG/\cI^{a+1}\widehat\cG\bigr).
$$
For $m>0$, regard the inverse line bundle
$\cI^{-m}=\Hom_{\widehat\cA}(\cI^m,\widehat\cA)$
as a subsheaf of $\cA^\times$; where $\cI=(t)$,
it is $t^{-m}\widehat\cA$.  Define, for every integer
$n$,
$$
 \mathrm F^n\cG^\times=\cI^n\widehat\cG
 \subset\cG^\times.
$$
Then $\mathrm F^0\cG^\times=\widehat\cG$, and there
are canonical equivariant identifications
\begin{equation}\label{eq:full-Laurent-gr}
 \gr^n\cG^\times
 :=\frac{\mathrm F^n\cG^\times}{\mathrm F^{n+1}\cG^\times}
 \simeq T\otimes\bigwedge^2\cF^\eps\otimes\cL^n,
 \qquad n\in\mathbb Z.
\end{equation}
 {For $i$ the test morphism \eqref{inclusion},}
$T=i^*T_+\mathfrak X_g^\eps$, and
$\cL=L\otimes R_E^{-1}$ are the central tangent and
conormal bundles of \cref{prop:boundary-bundles}.
Every Laurent section on a quasi-compact open subset of $C$
belongs to $\mathrm F^{-m}\cG^\times$ for some finite
$m\geq0$.
\end{lemma}

\begin{proof}
We first check that the finite-order coefficient bundles form
a compatible inverse system.  On a local presentation of the
formal source, the odd cotangent bundle has a basis given by
the odd coordinates modulo their products.  Reducing the even
parameter from order $a+1$ to order $a$ preserves
this basis.  Thus
$$
 \cF_{[a+1]}\otimes_{\cA_{[a+1]}}\cA_{[a]}
 \simeq\cF_{[a]}.
$$
More explicitly, the odd ideal $\cJ_{[a]}$ is the image
of $\cJ$ in $\cO_{\mathfrak Y_{[a]}}$.
On a local product chart with odd generators
$\theta_1,\ldots,\theta_\nu$, we have
$$
 \cF_{[a]}=
 \bigoplus_{j=1}^{\nu}\cA_{[a]}\,\overline\theta_j,
 \qquad \nu=3\text{ or }2
$$
according to the tail parity.  These are the same generators
at consecutive normal orders.  The resulting base-change maps
are canonical, since they are induced by the quotient maps of
the odd ideals, and hence agree on overlaps.  Dualising gives
$$
 \cV_{[a+1]}\otimes_{\cA_{[a+1]}}\cA_{[a]}
 \xrightarrow{\sim}\cV_{[a]}.
$$
The map on odd cotangent bundles is induced by the homomorphism
of structure sheaves of $F_{[a]}$.  Compatibility of the
formal maps therefore gives, under these identifications and
the corresponding identification of the pulled-back target bundles,
\begin{equation}\label{eq:odd-differential-base-change}
 dF_{a+1,-}\otimes_{\cA_{[a+1]}}\cA_{[a]}
 =dF_{a,-}.
\end{equation}
Duals, exterior powers, and the Hom in
\eqref{eq:comparison-module} commute with this base change
because the modules involved are locally free.  Compatibility
 of the maps $f_a$ of \eqref{fa}
consequently gives
$$
 \cG_{[a+1]}\otimes_{\cA_{[a+1]}}\cA_{[a]}
 \simeq\cG_{[a]}.
$$

Choose a sufficiently small affine open subset of $C$ on
which $\cG_{[0]}$ has a basis.  Each of its ordinary
nilpotent thickenings is affine.  The transition
$\cG_{[a+1]}\to\cG_{[a]}$ is surjective with
quasi-coherent kernel on that thickening, so sections lift:
the first cohomology of the kernel vanishes.  Lift the chosen
basis successively at every order.  A lifted tuple is still
a basis, since its reduction is a basis and the intervening
ideal is nilpotent.  Equivalently, the corresponding matrix
is invertible modulo a nilpotent ideal and therefore invertible.
Taking inverse limits of these compatible bases identifies
$\widehat\cG$ locally with a finite free
$\widehat\cA$-module.  This proves its local freeness,
its reductions, and the asserted completeness and separatedness.

At the central order,
$$
 \cG_{[0]}=
 \Hom\!\left(\bigwedge^2(\cF^\eps)^\vee,T\right)
 \simeq T\otimes\bigwedge^2\cF^\eps.
$$
The ideal $\cI$ is an invertible ideal generated locally
by a non-zero-divisor.  Its positive and negative tensor powers
therefore embed in $\cA^\times$, and multiplication gives
$$
 \frac{\cI^n\widehat\cG}{\cI^{n+1}\widehat\cG}
 \simeq \cG_{[0]}\otimes_{\cO_C}(\cI/\cI^2)^n.
$$
For a negative $n$, the last factor means the corresponding
dual tensor power.  The bosonic conormal line bundle
$\cI/\cI^2$ is $\cL$, by
\cref{thm:formal-map}.  This proves
\eqref{eq:full-Laurent-gr} for all integers $n$.
Locally the quotient is
$t^n\widehat\cG/t^{n+1}\widehat\cG$; the intrinsic
multiplication construction makes the identification equivariant.

Finally, a section of the localisation is locally a fraction
with a finite power of a boundary generator in its denominator.
On a quasi-compact open subset, choose finitely many such
presentations and let $m$ be the largest denominator order.
Every restriction then lies in $\cI^{-m}\widehat\cG$.
Membership in this subsheaf is local, so the original section
lies there as well.  Injectivity into the localisation follows
from the fact that a local generator of $\cI$ acts
injectively on the locally free module $\widehat\cG$.
\end{proof}

We will use the lemma to control changes of formal coordinates,
not to split the complete module into its graded pieces.
For example, on an overlap choose local boundary generators
$t_i,t_j$ and local frames $e_i,e_j$ of
$\widehat\cG$, and write
$$
 t_j=v_{ij}t_i,\qquad e_j=P_{ij}e_i,
$$
with $v_{ij}$ a unit and $P_{ij}$ an invertible
matrix over the completed structure sheaf.  The transition
for $t_j^n e_j$, modulo order $n+1$, is then
$\overline v_{ij}^{\,n}\overline P_{ij}$, where a bar
means reduction modulo $\cI$.  These are exactly the
transitions of $\cG_{[0]}\otimes\cL^n$ in
\eqref{eq:full-Laurent-gr}.  Higher-order terms can change
higher coefficients but cannot change this symbol.
{This is the transformation law used to remove poles below.}

\subsection{The negative invariant coefficients}
\label{subsec:negative-coefficients}

We now identify the possible first polar coefficient of a
primitive and prove that it must vanish.  On an open cover
$\mathcal U=\{U_i\}$, a \emph{Laurent primitive}
of a regular \v{C}ech cocycle $\gamma$ is a collection
$b_i\in\cG^\times(U_i)$ satisfying
$$
 \gamma_{ij}=b_j-b_i\quad\text{on }U_i\cap U_j.
$$
Suppose the sections $b_i$ have poles of order at most
$m>0$.  Their symbols at order $-m$ are their
images in
$\mathrm F^{-m}\cG^\times/\mathrm F^{-m+1}\cG^\times$.
If $-m$ is the first order at which a coefficient is
nonzero, these are the \emph{leading negative symbols}.
Regularity of $\gamma$ makes the symbols agree on
overlaps, since $\gamma$ has zero symbol at every
negative order.  They therefore define a global section of
$\gr^{-m}\cG^\times$.

The cocycle to be constructed will be invariant under tail
parity, so any proposed primitive can first be averaged under
the involution $\sigma$.  It is consequently enough to
exclude invariant global sections at all negative orders.
We compute the invariant summands and prove their vanishing
now; the finite pole-removal argument is given in
\cref{lem:lowest-Laurent}.  The only term not settled by
negative degree occurs in the even component at order $-1$,
where the nonsplit extension $\cF_\eta$ is essential.

Recall that $\sigma=a^2$ is the tail-parity involution
of the group $A$ used in Section~\ref{sec:vanishing}.
It fixes the underlying curve and acts by $+1$ on
$\cF_\eta,\cE_+$, and by $-1$ on
$N,\cL$ and on $Y$ when present.
The central decompositions from Section~\ref{sec:boundary} are
$$
 T=\cE_+\oplus N,
 \qquad
 \bigwedge^2\cF^-=S\otimes L,
 \qquad
 \bigwedge^2\cF^+=S\otimes L\oplus\cF_\eta\otimes Y.
$$
The factor $\cL^n$ contributes sign $(-1)^n$.
Multiplying the signs and keeping the $+1$-eigenspace
in \eqref{eq:full-Laurent-gr} gives, for the odd component,
\begin{equation}\label{eq:odd-Laurent-gr}
 (\gr^n\cG^\times)^\sigma=
 \begin{cases}
 S\otimes\cE_+\otimes L\otimes\cL^n,&n\text{ even},\\
 S\otimes N\otimes L\otimes\cL^n,&n\text{ odd},
 \end{cases}
\end{equation}
and, for the even component,
\begin{equation}\label{eq:even-Laurent-gr}
 (\gr^n\cG^\times)^\sigma=
 \begin{cases}
 (S\otimes\cE_+\otimes L \ \oplus \ \ N\otimes\cF_\eta\otimes Y) \
       \otimes\cL^n,&n\text{ even},\\
 (S\otimes N\otimes L \ \ \oplus \ \cE_+\otimes\cF_\eta\otimes Y) \
       \otimes\cL^n,&n\text{ odd}.
 \end{cases}
\end{equation}
These are direct-sum decompositions of the associated
graded bundles: the complete formal module need not carry
corresponding direct-sum decompositions.

\begin{lemma}[No negative invariant sections]\label{lem:no-negative-H0}
For both parities and every integer $n<0$,
$$
 H^0\bigl(C,(\gr^n\cG^\times)^\sigma\bigr)=0.
$$
\end{lemma}

\begin{proof}
We use $\deg L>0$ and the two exact sequences
\eqref{eq:Eplus} and \eqref{eq:Feta}.  Constant vector-space
factors do not affect the vanishing of sections.  Thus, after
the invariant summands have been selected as above, we can
suppress such factors and replace
$N=L^{-1}\otimes R_E$ and
$\cL=L\otimes R_E^{-1}$ by their powers of $L$.

We first handle the summands that occur in both components.
For $n=-2m$, with $m\geq1$, the relevant bundle
is $\cE_+\otimes L^{1-2m}$.  Tensoring
\eqref{eq:Eplus} by $L^{1-2m}$ gives
$$
 0\longrightarrow L^{-2m-1}
 \longrightarrow\cE_+\otimes L^{1-2m}
 \longrightarrow W_0\otimes L^{1-2m}
 \longrightarrow0.
$$
Both outer bundles have no sections, since all their line-bundle
factors have negative degree.  The middle bundle therefore has
no sections either.  For $n=-(2m+1)$, with $m\geq0$,
the common summand is just $L^{-(2m+1)}$, up to a
constant factor, so the same degree argument applies.
This proves the whole assertion for the odd component,
including order $-1$.

For the even component we must also treat the second summand
in each line of \eqref{eq:even-Laurent-gr}.  If
$n=-2m$, with $m\geq1$, this is
$\cF_\eta\otimes L^{-(2m+1)}$.  Tensoring
\eqref{eq:Feta} shows that its two factors are
$S\otimes L^{-(2m+1)}$ and $L^{-2m}$, both
of negative degree.  If $n=-(2m+1)$, with $m\geq1$,
the extra bundle is
$\cE_+\otimes\cF_\eta\otimes L^{-(2m+1)}$.
The filtrations of $\cE_+$ and $\cF_\eta$
induce a filtration with successive factors, suppressing
constant multiplicities,
$$
 L^{-(2m+3)},\quad L^{-(2m+2)},\quad
 L^{-(2m+1)},\quad L^{-2m}.
$$
Again all degrees are negative.  Applying the left-exact
global-section functor to the successive short exact sequences
proves vanishing for the whole bundle. 

The remaining case is the extra even-component summand at
$n=-1$.  We must prove
$H^0(C,\cE_+\otimes\cF_\eta\otimes L^{-1})=0$.
The preceding filtration now has a factor $\cO_C$,
so negative degree alone does not suffice.  The issue is
whether its constant section lifts through the extension.
Tensor \eqref{eq:Feta} by $L^{-1}$ to obtain
\begin{equation}\label{eq:exceptional-extension}
 0\longrightarrow S\otimes L^{-1}
 \longrightarrow\cF_\eta\otimes L^{-1}
 \longrightarrow\cO_C\longrightarrow0.
\end{equation}
Since $C$ is connected and proper and
$H^0(C,L^{-1})=0$, the cohomology sequence begins
$$
 0\longrightarrow H^0(C,\cF_\eta\otimes L^{-1})
 \longrightarrow\C
 \xrightarrow{\ \partial_\eta\ }S\otimes H^1(C,L^{-1}).
$$
The connecting map sends $1$ to the extension class
$\pm s\otimes\eta\ne0$.  Indeed, choose local lifts
of $1$ in \eqref{eq:exceptional-extension}; their
differences are the cocycle defining that extension.
The map $\partial_\eta$ is therefore nonzero and,
because its domain is one-dimensional, injective.  Consequently
\begin{equation}\label{eq:key-H0}
 H^0(C,\cF_\eta\otimes L^{-1})=0.
\end{equation}
To include the factor $\cE_+$, tensor
\eqref{eq:Eplus} by $\cF_\eta\otimes L^{-1}$:
$$
 0\longrightarrow\cF_\eta\otimes L^{-3}
 \longrightarrow\cE_+\otimes\cF_\eta\otimes L^{-1}
 \longrightarrow W_0\otimes\cF_\eta\otimes L^{-1}
 \longrightarrow0.
$$
The first bundle has negative-degree factors $L^{-3}$
and $L^{-2}$, and the last has no sections by
\eqref{eq:key-H0}.  The middle bundle therefore has no
sections.  This completes the exceptional case and the proof.
\end{proof}

\subsection{A compatible completed cocycle}
\label{subsec:completed-cocycle}

The preceding vanishing will apply once we have a regular
cocycle through all formal orders, with a known nonzero
central reduction.  We now construct that cocycle from a
vector-bundle extension representing the primary class on
$\mathfrak X_g^\eps$.  At each order we pull the
extension back along $f_a$ and change its coefficient
using $dF_{a,-}$.  We then choose local splittings
compatibly as the order increases, and take their differences.
This constructs the completed cocycle itself, rather than
only a list of finite-order cohomology classes.

All finite-order extensions below are extensions of vector
bundles on ordinary schemes mapping to the reduced target
stack.  We use the primary class of the smooth ambient
superstack $\mathfrak X_g^\eps$, which includes the
boundary divisor $D$; we do not need to assign a primary obstruction
class to a nonreduced source thickening or to its localisation.

On $X=(\mathfrak X_g^\eps)_{\red}$, put
$$
 K={\mathcal Hom}(\bigwedge^2T_-\mathfrak X_g^\eps,
              T_+\mathfrak X_g^\eps).
$$
Choose a vector-bundle extension representing
$\omega_2^{\rm alg}(\mathfrak X_g^\eps)$:
\begin{equation}\label{eq:Green-extension}
 0\longrightarrow K\longrightarrow B
 \xrightarrow{p}\cO_X\longrightarrow0.
\end{equation}
Such an extension exists because
$\Ext^1(\cO_X,K)=H^1(X,K)$.  Its quotient is locally
free, so the extension is locally split and its middle term
is locally free.  It is an extension on the stack itself,
including its descent data.

At order $a$, precomposition with the odd differential
defines the coefficient map
$$
 \phi_a:f_a^*K\longrightarrow\cG_{[a]},
 \qquad b\longmapsto b\circ\bigwedge^2dF_{a,-},
$$
 where $f_a$ is the map in \eqref{fa},
$\cG_{[a]}$ is defined in \eqref{eq:comparison-module},
and $dF_{a,-}$ is the odd differential \eqref{dF} induced by the map
$\widehat F$ of Theorem \ref{thm:formal-map}.
Pull back \eqref{eq:Green-extension} and push out along
$\phi_a$.  If $\cP_{[a]}$ is the middle term
of the pushout, the resulting diagram has exact rows:
\begin{equation}\label{eq:comparison-pushout}
\begin{tikzcd}[column sep=small,row sep=large]
 0\arrow[r]&f_a^*K\arrow[r]\arrow[d,"\phi_a"]
 &f_a^*B\arrow[r]\arrow[d]
 &\cA_{[a]}\arrow[r]\arrow[d,equal]&0\\
 0\arrow[r]&\cG_{[a]}\arrow[r]
 &\cP_{[a]}\arrow[r,"q_a"]
 &\cA_{[a]}\arrow[r]&0.
\end{tikzcd}
\end{equation}
The pullback is exact without a flatness assumption on $f_a$,
since the original sequence is locally split.  The pushout is
an extension with locally free quotient, hence is locally
split as well.  Compatibility of the formal maps and
\eqref{eq:odd-differential-base-change} give identifications
$$
 \cP_{[a+1]}\otimes_{\cA_{[a+1]}}\cA_{[a]}
 \simeq\cP_{[a]},\qquad
 q_{a+1}\otimes\cA_{[a]}=q_a.
$$
Thus these are compatible extensions at all orders, with
specified reduction maps.  The next lemma shows directly
that their inverse limit remains an extension and admits
the local splittings needed for a single cocycle.

\begin{lemma}[Compatible invariant local splittings]
\label{lem:compatible-splittings}
Let $\mathcal U=\{U_i\}$ be a finite affine open cover
of $C$, and put
$\widehat\cP=\varprojlim_a\cP_{[a]}$, using the
middle terms and reduction maps of
\eqref{eq:comparison-pushout}.  There is an exact sequence
of locally free $\widehat\cA$-modules
\begin{equation}\label{eq:completed-extension}
 0\longrightarrow\widehat\cG
 \longrightarrow\widehat\cP
 \xrightarrow{\widehat q}\widehat\cA
 \longrightarrow0.
\end{equation}
It reduces modulo $\cI^{a+1}$ to the lower sequence
of \eqref{eq:comparison-pushout}.  In particular,
$\widehat\cP/\cI^{a+1}\widehat\cP\simeq\cP_{[a]}$,
and these modules are complete and separated for the
$\cI$-adic topology.

There exist local sections
$\lambda_i\in\widehat\cP(U_i)$ such that
$\widehat q(\lambda_i)=1$ and
$\sigma\lambda_i=\lambda_i$.  Equivalently, the maps
$a\mapsto a\lambda_i$ are equivariant local splittings
of \eqref{eq:completed-extension}.  Their differences define
the regular invariant cocycle
\begin{equation}\label{eq:Gamma}
 \Gamma_{ij}=\lambda_j-\lambda_i,
 \qquad
 \Gamma\in\check Z^1(\mathcal U,\widehat\cG)^\sigma.
\end{equation}
Changing the local splittings changes $\Gamma$ by a
coboundary $\delta d$, with
$d\in\check C^0(\mathcal U,\widehat\cG)$.
\end{lemma}

\begin{proof}
Each $(U_i,\cA_{[a]})$ is a nilpotent thickening of an
affine scheme, and hence is affine.  The lower extension in
\eqref{eq:comparison-pushout} splits there, since
$H^1(U_i,\cG_{[a]})=0$.  Equivalently, one can
choose $\lambda_{i,a}\in\cP_{[a]}(U_i)$ with
$q_a(\lambda_{i,a})=1$.  We must make these choices
compatible, not merely obtain a splitting separately at each
order.

Suppose $\lambda_{i,a}$ has already been chosen.
At the next order choose a section satisfying
$$
 \mu_{i,a+1}\in\cP_{[a+1]}(U_i),\qquad
 q_{a+1}(\mu_{i,a+1})=1.
$$
Its reduction minus
$\lambda_{i,a}$ lies in the kernel of $q_a$,
which is $\cG_{[a]}$.  That difference lifts to a
section $d_{i,a+1}\in\cG_{[a+1]}(U_i)$: the
transition map is surjective, with quasi-coherent kernel on
an affine thickening, and the kernel has zero first cohomology.
Then
$$
 \lambda_{i,a+1}=\mu_{i,a+1}-d_{i,a+1}
$$
still maps to $1$, and reduces to $\lambda_{i,a}$.
Starting at order zero and repeating this construction gives
a compatible family of local splittings.

At every order these splittings identify the middle term locally
with $\cG_{[a]}\oplus\cA_{[a]}$, and the
identifications commute with reduction.  Their inverse limit
therefore identifies
$$
 \widehat\cP|_{U_i}\simeq
 \widehat\cG|_{U_i}\oplus\widehat\cA|_{U_i},
$$
with $\widehat q$ the projection to the second factor.
This proves exactness of \eqref{eq:completed-extension},
local freeness, and the assertions about reduction and
completeness.  

We next arrange equivariance.  At each finite order, the
comparison arrow $f_a\circ\sigma\simeq f_a$ induced
by $\widehat F$ identifies
$\sigma^*f_a^*B$ with $f_a^*B$, and similarly
for $K$ and $\cO_X$.  The extension
\eqref{eq:Green-extension} is defined on the target stack,
so these identifications respect its maps.  They satisfy the
group law by coherence of the comparison arrows.
The odd differential and the pushouts inherit
the same compatibility.  Thus $\sigma$ acts on the
extensions at every order and commutes with reduction.
It fixes the underlying curve and hence each open $U_i$,
although it need not act trivially on its complete ring of
functions.

Replace the compatible local sections by
$$
 \lambda_{i,a}^{\rm inv}
 =\tfrac12\bigl(\lambda_{i,a}+\sigma\lambda_{i,a}\bigr).
$$
Their images under $q_a$ remain equal to $1$, since
$\sigma(1)=1$.  Averaging commutes with reduction, so
their limits are invariant sections $\lambda_i$ with
$\widehat q(\lambda_i)=1$.  Each gives the
$\widehat\cA$-linear splitting
$a\mapsto a\lambda_i$; this splitting is equivariant
because
$\sigma(a\lambda_i)=\sigma(a)\lambda_i$.

On an overlap, $\lambda_j-\lambda_i$ has image zero
under $\widehat q$, so it is a section of
$\widehat\cG$.  On a triple overlap its differences
telescope, proving the cocycle identity in \eqref{eq:Gamma}.
If $\lambda_i'$ is another choice, then
$d_i=\lambda_i'-\lambda_i$ lies in
$\widehat\cG(U_i)$, and
$\Gamma'-\Gamma=\delta d$.  This proves the final
assertion.
\end{proof}

The remaining task is to identify its reduction at order zero.
By \cref{lem:complete-filtration}, the coefficient at that
order is
$\widehat\cG/\cI\widehat\cG
=\cG_{[0]}\simeq T\otimes\bigwedge^2\cF^\eps$.
Thus reducing $\Gamma$ means applying this quotient
map to every overlap section $\Gamma_{ij}$.  The
next proposition connects this concrete cocycle to the nonzero
class $c_\eta^\eps$ of \cref{prop:central-class}.

\begin{proposition}[The central symbol]\label{prop:weight-zero}
Let $\Gamma=(\lambda_j-\lambda_i)$ be the cocycle
of \cref{lem:compatible-splittings}, and denote its
coefficientwise reduction by
$$
 \Gamma_0=\Gamma\bmod\cI
 \in\check Z^1\!\left(\mathcal U,
                  T\otimes\bigwedge^2\cF^\eps\right).
$$
Its class in $H^1(C,T\otimes\bigwedge^2\cF^\eps)$
is $c_\eta^\eps\ne0$.  More precisely, it is the
image of
$$
 0\ne s\otimes J_h^\partial(\eta)
 \in H^1(C,S\otimes\cE_+\otimes L)
$$
under the invariant direct-summand inclusion
$$
 S\otimes\cE_+\otimes L
 \hookrightarrow\gr^0\cG^\times
 =T\otimes\bigwedge^2\cF^\eps.
$$
This is the summand displayed at $n=0$ in
\eqref{eq:odd-Laurent-gr} and \eqref{eq:even-Laurent-gr}.
\end{proposition}

\begin{proof}
At order zero, $f_0=i$,
$\cV_{[0]}=(\cF^\eps)^\vee$, and $\phi_0$
is precomposition with $\bigwedge^2dF_{0,-}$.
The lower extension in \eqref{eq:comparison-pushout} at
this order therefore represents exactly the central comparison
class $c_\eta^\eps$ of \cref{prop:central-class}.
The reductions of $\lambda_i$ are local sections of
$\cP_{[0]}$ mapping to $1$, and hence are local
splittings of this extension.  Their differences are
$\Gamma_0$, so its cohomology class is
$c_\eta^\eps$.  That proposition identifies the class
as the image of $s\otimes J_h^\partial(\eta)$.

The inclusion in the statement is a direct-summand
inclusion: use $T=\cE_+\oplus N$ and the
decomposition of $\bigwedge^2\cF^\eps$.
It therefore induces an injection on cohomology.
By \cref{lem:Jh-injection},
$s\otimes J_h^\partial(\eta)\ne0$, so the class
of $\Gamma_0$ is nonzero as well.
\end{proof}

\subsection{Removal of Laurent poles}
\label{subsec:Laurent-survival}

We now combine the two ingredients: $\Gamma$ has a
nonzero central class by \cref{prop:weight-zero}, and every
negative invariant graded bundle has no global sections by
\cref{lem:no-negative-H0}.  The following lemma is the
filtered argument needed for \cref{thm:Laurent}.
Given a Laurent primitive of a regular invariant cocycle,
we average it, choose a common finite pole bound, and look
at its symbol at that order.  The symbol glues globally and
must vanish, improving the bound by one.  Repetition removes
every negative order.  The argument uses only the filtration
and its quotient maps, so it also applies when the complete
transition functions contain higher-order terms.

\begin{lemma}[Removal of poles from a primitive]\label{lem:lowest-Laurent}
Let $\widehat{\mathcal R}$ be a sheaf of complex algebras
on $C$, complete and separated for the powers of a
locally principal Cartier ideal $\cI$, and let
$\widehat{\cM}$ be a locally free
$\widehat{\mathcal R}$-module of finite rank.  Set
$$
 \cM^\times=\widehat{\cM}[\cI^{-1}],\qquad
 \mathrm F^n\cM^\times=\cI^n\widehat{\cM}
 \quad(n\in\mathbb Z).
$$
Suppose a finite group $H$ acts compatibly on the ring,
ideal, and module, fixes the underlying curve, and satisfies
$$
 H^0\bigl(C,(\gr^n\cM^\times)^H\bigr)=0
 \qquad(n<0).
$$
Let $\mathcal U=\{U_i\}$ be a finite cover of $C$
by quasi-compact open subsets, and let
$\gamma\in\check Z^1(\mathcal U,\widehat{\cM})^H$.
If
$$
 \gamma=\delta b
 \quad\text{for some }b\in\check C^0(\mathcal U,\cM^\times),
$$
then there is an invariant regular primitive
$b^{\rm reg}\in\check C^0(\mathcal U,\widehat{\cM})^H$
with $\delta b^{\rm reg}=\gamma$.

In particular, if the reduction of $\gamma$ in
$\widehat{\cM}/\cI\widehat{\cM}$ represents a nonzero
cohomology class on $C$, then $\gamma$ is not a
Laurent coboundary on this cover.
\end{lemma}

\begin{proof}
Suppose $\gamma=\delta b$ for a Laurent
$0$-cochain $b$.  Since $H$ fixes the underlying
curve, every open subset in the cover is stable under its
action.  We can therefore replace $b$ by
$$
 b^{\rm av}=\frac1{|H|}\sum_{h\in H}h\cdot b.
$$
The differential $\delta$ commutes with the action, and
$\gamma$ is invariant, so
$\delta b^{\rm av}=\gamma$.  Write $b$ for this
averaged cochain from now on.

Each component $b_i$ has a finite pole bound on $U_i$.
Indeed it is locally a section of $\widehat{\cM}$
divided by a power of a generator of $\cI$, and a finite
subcover of the quasi-compact $U_i$ gives a common bound
there.  Since there are finitely many $U_i$, we obtain
one integer $N\leq0$ such that
$$
 b\in\check C^0(\mathcal U,\mathrm F^N\cM^\times)^H.
$$
This need not be the sharp pole bound; any such $N$
will start the argument.

If $N<0$, apply the quotient map to the order-$N$
symbol.  It commutes with restriction and hence with the
\v{C}ech differential, giving
\begin{equation}\label{eq:Laurent-symbol-diagram}
\begin{tikzcd}[column sep=large,row sep=large]
 \check C^0(\mathcal U,\mathrm F^N\cM^\times)
 \arrow[r,"\delta"]\arrow[d,"\operatorname{symb}_N"]
 &\check C^1(\mathcal U,\mathrm F^N\cM^\times)
 \arrow[d,"\operatorname{symb}_N"]\\
 \check C^0(\mathcal U,\gr^N\cM^\times)
 \arrow[r,"\delta"]
 &\check C^1(\mathcal U,\gr^N\cM^\times).
\end{tikzcd}
\end{equation}
The cocycle $\gamma$ is regular, and
$\mathrm F^0\subset\mathrm F^{N+1}$ for $N<0$.
Its order-$N$ symbol is consequently zero.  Since
$\delta b=\gamma$, the symbols of the $b_i$
agree on overlaps and glue to a global section of
$(\gr^N\cM^\times)^H$.  By hypothesis that section
is zero.  The kernel of the symbol map is
$\mathrm F^{N+1}\cM^\times$, so in fact every
$b_i$ belongs to this smaller subsheaf.

We can now repeat the argument at orders $N+1,\ldots,-1$.
After finitely many repetitions,
$b\in\check C^0(\mathcal U,\mathrm F^0\cM^\times)^H$,
so the averaged primitive is regular.  It gives the required
$b^{\rm reg}$.  Reduction of
$\gamma=\delta b^{\rm reg}$ modulo $\cI$
makes the central cocycle a coboundary, proving the final
assertion.  
\end{proof}

\begin{theorem}[Laurent survival]\label{thm:Laurent}
For either parity, let
$\Gamma=(\lambda_j-\lambda_i)
\in\check Z^1(\mathcal U,\widehat\cG)^\sigma$
be the regular cocycle constructed in
\cref{lem:compatible-splittings}.  Its image is not a
coboundary in $\check C^\bullet(\mathcal U,\cG^\times)$.
Equivalently, the localised extension
$$
 0\longrightarrow\cG^\times
 \longrightarrow\widehat\cP\otimes_{\widehat\cA}\cA^\times
 \xrightarrow{\widehat q^\times}\cA^\times
 \longrightarrow0
$$
admits no global splitting as a sequence of
$\cA^\times$-modules on $C$.
\end{theorem}

\begin{proof}
Apply \cref{lem:lowest-Laurent} with
$\widehat{\mathcal R}=\widehat\cA$,
$\widehat{\cM}=\widehat\cG$, and
$H=\langle\sigma\rangle$.
The Cartier filtration, completeness, and finite pole bounds
are supplied by \cref{lem:complete-filtration}.
The negative invariant global sections vanish by
\cref{lem:no-negative-H0}, and the class of
$\Gamma_0$ is nonzero by \cref{prop:weight-zero}.
Thus a Laurent primitive would be replaced by a regular one,
whose reduction would contradict that nonzero class.

To verify the equivalent formulation, a global splitting of
the localised extension is determined by a section $\beta$
of its middle term with $\widehat q^\times(\beta)=1$.
It would give
$$
 b_i=\lambda_i-\beta|_{U_i}\in\cG^\times(U_i),
 \qquad b_j-b_i=\lambda_j-\lambda_i=\Gamma_{ij},
$$
contradicting the first assertion.  Conversely, a Laurent
primitive $b$ would make the local sections
$\lambda_i-b_i$ agree on overlaps.  They would glue
to such a $\beta$, and multiplication by $\beta$
would define a splitting.  
\end{proof}

\section{Nonvanishing on smooth algebraic supermoduli}
\label{sec:algebraic}

We have constructed a transverse formal map in
Section~\ref{sec:vanishing} and a compatible comparison
extension in Section~\ref{sec:Laurent}.
Theorem~\ref{thm:Laurent} shows that this extension remains
nonsplit after the boundary parameter is inverted.  We now
use that conclusion to prove nonvanishing of the primary
class on the whole smooth algebraic component
$\M_g^\eps$.

The argument is by contradiction.  If the primary class vanished
on the smooth locus, its representing vector-bundle extension
would split there.  We must transfer this splitting to the
Laurent extension just constructed.  It cannot be pulled back
directly along a finite-order map $f_a$, since every
finite-order source is supported on the boundary.  Instead,
algebraicity gives a finite bound on its poles: after twisting
by a multiple of the boundary divisor, the splitting section
extends regularly and can be pulled back at every order.
Completing and then localising recovers a Laurent splitting,
contradicting \cref{thm:Laurent}.

Section~\ref{subsec:finite-pole-pullback} proves this implication,
and Section~\ref{subsec:algebraic-conclusion} assembles the
test construction for every genus and both parities.
The finite pole bound is the specifically algebraic input
here.  The passage to holomorphic nonvanishing will use the
comparison theorem developed in the next two sections.

\subsection{Algebraic sections and finite pole order}
\label{subsec:finite-pole-pullback}

We first isolate the auxiliary extension fact needed in
\cref{prop:finite-poles}.  An algebraic section defined away
from an effective Cartier divisor has only a finite pole order
on each affine chart.  Quasi-compactness gives one bound for
finitely many charts, and uniqueness of extension with that
bound makes the resulting sections descend to the stack.
Lemma~\ref{lem:algebraic-finite-poles} records this fact for
an arbitrary vector bundle.  We then apply it to the section
of the middle term that specifies a splitting, keeping track
of its image in the quotient throughout pullback, completion,
and localisation.

\begin{lemma}[Finite poles along a Cartier divisor]
\label{lem:algebraic-finite-poles}
Let $X$ be a quasi-compact Deligne--Mumford stack over $\C$,
let $D_0\subset X$ be an effective Cartier divisor, and let
$E$ be a vector bundle.  Every algebraic section of
$E|_{X\setminus D_0}$ extends, for some $m\geq0$,
to a section of $E(mD_0)$ on $X$.  Here restriction
to the complement uses the canonical trivialisation of
$\cO_X(mD_0)$ there.  For fixed $m$, such an
extension is unique if it exists.
\end{lemma}

\begin{proof}
Choose an \'etale atlas consisting of finitely many affine
schemes $V_\alpha$ on which the divisor is principal.
To obtain such an atlas, first refine any \'etale atlas by
affine opens trivialising the invertible ideal of $D_0$,
and then use quasi-compactness of $X$ to select finitely
many whose images cover it.

On one chart $V=\Spec R$, let $u\in R$ be an
equation of the pulled-back divisor, and let $P$ be the
finite projective module corresponding to $E|_V$.
Since $u$ is a non-zero-divisor, it acts injectively
on $P$.  We may therefore regard $P$ and its
twists as submodules of their localisation.  Sections on the
principal open complement are described by
\begin{equation}\label{eq:finite-pole-module}
 P[1/u]=\bigcup_{m\geq0}u^{-m}P.
\end{equation}
The restriction of the given section thus belongs to
$u^{-m_\alpha}P$ for some finite $m_\alpha$.
Taking $m=\max_\alpha m_\alpha$ yields regular
sections of $E(mD_0)$ on all charts with the same twist.

We check their compatibility on the \'etale groupoid of the
atlas.  On an overlap $V_\alpha\times_X V_\beta$,
the two sections agree away from the pulled-back divisor,
because they both restrict to the original section.
Restriction of a locally free sheaf across the complement of
an effective Cartier divisor is injective: locally, a section
that becomes zero after inverting $u$ is annihilated by
a power of $u$, and hence is zero.  This argument also
applies on the overlaps, since the \'etale pullbacks preserve
the Cartier property.  The two sections therefore agree on
the whole overlap and descend to a section on $X$.

The same injectivity shows that two extensions with the same
twist must coincide.  This proves existence and uniqueness.
\end{proof}

We apply the lemma to the extension
$0\to K\to B\xrightarrow{p}\cO_X\to0$ chosen
in \eqref{eq:Green-extension}.  A splitting is determined
by a section of $B$ mapping to $1$.  The following
proposition explains how a splitting on the smooth locus
would give precisely the Laurent splitting excluded in
Section~\ref{sec:Laurent}.

\begin{proposition}[Finite-pole pullback of a splitting]
\label{prop:finite-poles}
Fix the central data and formal map of \cref{thm:formal-map},
and let $\widehat\cP$ be the middle term of the completed
comparison extension \eqref{eq:completed-extension}.
If $\omega_2^{\rm alg}(\M_g^\eps)=0$, there exists
a global section
$$
 \beta^\times\in
 H^0\!\left(C,\widehat\cP
              \otimes_{\widehat\cA}\cA^\times\right),
 \qquad \widehat q^\times(\beta^\times)=1.
$$
Consequently the localised comparison extension splits, and
the cocycle $\Gamma=(\lambda_j-\lambda_i)$ of
\cref{lem:compatible-splittings} is a Laurent coboundary on
the finite cover $\mathcal U$.
\end{proposition}

\begin{proof}
We first turn the assumed vanishing into a section with a
finite pole bound on the ambient reduced stack.  Set
$$
 X=(\mathfrak X_g^\eps)_{\red},\qquad
 D_0=D_{\red},\qquad
 U=X\setminus D_0=\mathcal S_g^\eps.
$$
The stack $X$ is quasi-compact: it is an open substack
of the finite-type reduced stable moduli stack used in
Section~\ref{sec:boundary}.  The divisor $D_0$ is
effective Cartier by the local boundary coordinates in
\cref{prop:boundary-bundles}.  The restriction of
\eqref{eq:Green-extension} to $U$ represents
$\omega_2^{\rm alg}(\M_g^\eps)$.
Its vanishing means that this extension splits.  If
$\beta_U\in H^0(U,B|_U)$ is the image of $1$
under a splitting, then $p(\beta_U)=1$.

By \cref{lem:algebraic-finite-poles}, $\beta_U$
extends to a regular section
$\beta_m\in H^0(X,B(mD_0))$ for some finite
$m\geq0$.  Denote by $1_{mD_0}$ the canonical
image of $1$ under
$\cO_X\hookrightarrow\cO_X(mD_0)$.
Then
$$
 p(mD_0)(\beta_m)=1_{mD_0}.
$$
Indeed, both sides agree on $U$, and restriction of
$\cO_X(mD_0)$ to $U$ is injective by the
lemma.  The notation on the right records a section (not
a choice of frame).  On a chart where $D_0=(u=0)$,
the frame of $\cO_X(mD_0)$ is $u^{-m}$,
whereas its canonical section is $u^m u^{-m}$.
Writing $\beta_m=u^{-m}v$ in this frame, with
$v$ a regular section of $B$, the equality becomes
$p(v)=u^m$.

We can now pull back this regular twisted section at every
finite order.  Recall the ordinary maps
$$
 f_a:(C,\cA_{[a]})\longrightarrow X,
 \qquad \cA_{[a]}=\widehat\cA/\cI^{a+1},
$$
from Section~\ref{subsec:complete-coefficients}.
The sections $f_a^*\beta_m$ of
$f_a^*B(mD_0)$, and their images under the quotient
maps, commute with reduction from order $a+1$ to
order $a$.  For a vector bundle $E$ on $X$,
write
$$
 \widehat{f^*E}=\varprojlim_a f_a^*E.
$$
These complete modules are locally free: local frames of
the target bundle pull back compatibly on the formal charts.
Equivalently, one can lift frames successively as in
\cref{lem:complete-filtration}.

To identify the completed twist, consider first the line
bundle $\cO_X(-D_0)$ and its map to $\cO_X$.
On compatible charts, \cref{thm:formal-map} gives
$\widehat{f^*u}=e t$, with $e$ invertible and
$\cI=(t)$ in the complete bosonic source ring.
Consequently the completed map
$$
 \widehat{f^*\cO_X(-D_0)}\longrightarrow\widehat\cA
$$
is locally multiplication by $e t$.  Since $t$
is a non-zero-divisor in $\widehat\cA$, this map
identifies the completed pullback line bundle with $\cI$.
At finite order the pullback line bundle is still invertible,
but multiplication by the nilpotent image of $e t$ is
not injective.  Thus the identification with an ideal is
being made in the complete ring, not in a finite-order
truncation.

Taking dual tensor powers and using local frames therefore gives
$$
 \varprojlim_a f_a^*B(mD_0)
 \simeq\widehat{f^*B}\otimes_{\widehat\cA}\cI^{-m}.
$$
The compatible sections $f_a^*\beta_m$ define a section
of this module.  Its inclusion into the localisation gives
$$
 \widetilde\beta^\times\in
 H^0\!\left(C,\widehat{f^*B}
              \otimes_{\widehat\cA}\cA^\times\right).
$$
Let $\widehat p:\widehat{f^*B}\to\widehat\cA$
be the limit of the maps $f_a^*p$.  The image of this
section under its localisation is $1$.
This can also be checked directly in the local expression above:
$$
 \widetilde\beta^\times=(e t)^{-m}\widehat{f^*v},
 \qquad
 \widehat p^{\,\times}(\widetilde\beta^\times)
 =(e t)^{-m}\widehat{f^*(u^m)}=1.
$$
In particular, substitution by the formal map preserves a
finite pole bound.

This construction explains the order of operations.  We
extended a section of $B|_U$ to a regular section of
$B(mD_0)$, pulled that section back to all finite
orders, took the compatible limit, and then inverted the
Cartier ideal.  Localising $\cA_{[a]}$ itself at
$t$ would give the zero ring, since $t$ is
nilpotent there.  The regular twisted pullbacks are what
allow us to pass through those finite orders.

It remains to transfer the section to the comparison extension.
The compatible middle arrows in
\eqref{eq:comparison-pushout} give a map
$\widehat\psi:\widehat{f^*B}\to\widehat\cP$.
After localisation it fits into the commutative diagram
$$
\begin{tikzcd}[column sep=large,row sep=large]
 \widehat{f^*B}\otimes_{\widehat\cA}\cA^\times
 \arrow[r,"\widehat\psi^\times"]
 \arrow[d,"\widehat p^{\,\times}"']
 &\widehat\cP\otimes_{\widehat\cA}\cA^\times
 \arrow[d,"\widehat q^\times"]\\
 \cA^\times\arrow[r,equal]&\cA^\times.
\end{tikzcd}
$$
Hence $\beta^\times=\widehat\psi^\times
(\widetilde\beta^\times)$ satisfies
$\widehat q^\times(\beta^\times)=1$.
This section is global: $\beta_m$ descended on $X$,
and pullback, the compatible inverse limits, and the maps
in the pushout preserve its gluing identities.
Multiplication by $\beta^\times$ gives the required
$\cA^\times$-linear splitting.

Finally, the local sections $\lambda_i$ of
\cref{lem:compatible-splittings} also map to $1$.
Thus
$$
 b_i=\lambda_i-\beta^\times|_{U_i}
 \in\cG^\times(U_i).
$$
Their differences on overlaps are
$b_j-b_i=\lambda_j-\lambda_i=\Gamma_{ij}$, so
$\delta b=\Gamma$.  This is the asserted Laurent
primitive.
\end{proof}

\subsection{Nonvanishing in every genus and both parities}
\label{subsec:algebraic-conclusion}

The constructions in Section~\ref{sec:boundary} are available
for every $g\geq3$ and each parity.  We now combine Laurent
survival with \cref{prop:finite-poles} to deduce nonvanishing
on the entire smooth moduli component.

\Needspace{9\baselineskip}
\begin{theorem}[Algebraic nonvanishing in both parities]
\label{thm:algebraic}
For every integer $g\geq3$ and $\eps\in\{+,-\}$,
$$
 \omega_2^{\rm alg}(\M_g^\eps)\ne0
 \quad\hbox{in}\quad
 H^1\!\left(\mathcal S_g^\eps,
 \Hom(\bigwedge^2T_-\M_g^\eps,T_+\M_g^\eps)\right).
$$
In particular each such algebraic component is non-projected and
non-split.
\end{theorem}

\begin{proof}
Fix $g\geq3$ and $\eps\in\{+,-\}$, and put $h=g-1$.
Choose the ineffective even core $(C,L)$ of
\cref{lem:ineffective-theta} and a nonzero
$\eta\in U_C=H^1(C,L^{-1})$, whose dimension is
$2h-2>0$.  Section~\ref{sec:boundary} supplies the elliptic
tail of parity $\eps$ and the central map
$F_0:Z_\eta^\eps\to D\subset\mathfrak X_g^\eps$.
The core has even parity, so clutching and subsequent smoothing
have total parity $\eps$.

By \cref{prop:central-class}, this map has nonzero central
comparison class $c_\eta^\eps$.  This uses the compatibility
\eqref{eq:primary-naturality}, the source primary class of
\cref{lem:pointed-core-test}, and the injection of
\cref{lem:Jh-injection}.  Theorem~\ref{thm:formal-map}
extends $F_0$ formally, and \cref{prop:weight-zero}
identifies the central class of the completed comparison
extension with $c_\eta^\eps$.
Theorem~\ref{thm:Laurent} then excludes a splitting of that
extension after Laurent localisation.

If $\omega_2^{\rm alg}(\M_g^\eps)$ vanished,
\cref{prop:finite-poles} would produce precisely such a
Laurent splitting, a contradiction.  The construction applies
for every $h\geq2$ and both parities, proving the asserted
nonvanishing.  Projectedness forces the primary class to vanish,
and splitness implies projectedness, so both are excluded.
\end{proof}

This completes the algebraic argument.  Section~\ref{sec:gaga}
proves an injectivity theorem for analytification of first
cohomology, and Section~\ref{sec:compactification} verifies
its geometric hypotheses for spin-level covers and transfers
the conclusion to the spin stack.  Together they will show
that the primary classes just detected remain nonzero
holomorphically.

\section{Algebraic-to-analytic comparison in codimension two}
\label{sec:gaga}

Theorem~\ref{thm:algebraic} establishes nonvanishing of the
algebraic primary class.  To obtain the holomorphic conclusion,
we must show that this class cannot become zero after
analytification.  Since the smooth moduli space is not proper,
the usual GAGA theorem does not apply to it directly.  We prove
instead that, for any vector bundle on an open subset of a
normal projective variety whose complement has codimension at
least two, analytification is injective on first cohomology.
Section~\ref{sec:compactification} will construct such
compactifications for spin-level covers and use this injection
to complete the nonvanishing argument on the spin stack.

For the supergeometric result, this section and the
compactification part of Section~\ref{sec:compactification}
provide the auxiliary comparison machinery.  Their statements
also concern ordinary algebraic and complex geometry
independently of that application: the comparison applies to
arbitrary vector bundles and their extension classes, and the
compactifications in the next section concern moduli of
ordinary curves with spin and level structures. 

Throughout this section, $U$ is a dense open subset of a
normal complex algebraic variety $X$, and
$j:U\hookrightarrow X$ denotes the inclusion. We set
$Z=X\setminus U$, assume
$\operatorname{codim}_X Z\geq2$, and write
$j^{\an}:U^{\an}\hookrightarrow X^{\an}$ for the
analytified inclusion. For the comparison theorem, $X$
will also be projective, and hence a compactification of $U$.

The argument represents a class in $H^1(U,F)$, for an
algebraic vector bundle $F$, by an extension of
$\cO_U$ by $F$. If the class becomes zero after
analytification, this extension admits a holomorphic splitting.
Such a splitting is a section of its middle bundle mapping
to $1$ under the quotient morphism. We first show that this
section extends across $Z$ as a section of a coherent
reflexive sheaf on $X^{\an}$. When $X$ is projective,
GAGA then algebraizes the extended section. Compatibility
with the quotient morphism ensures that its restriction to
$U$ is an algebraic splitting, proving that the original
class was already zero.

\subsection{Reflexive extension across the complement}

We first establish the extension property needed for this
argument. An algebraic vector bundle $F$ on $U$ need not
extend as a vector bundle on $X$. The appropriate extension
is instead the coherent reflexive sheaf $E=j_*F$, which
may fail to be locally free along $Z$.

The following lemma shows that this extension also captures
holomorphic sections: for every analytic open subset
$W\subset X^{\an}$, sections of $F^{\an}$ on
$W\cap U^{\an}$ extend uniquely to sections of $E^{\an}$
on $W$. Taking $W=X^{\an}$ will provide the global
section to which we apply GAGA in the splitting argument.

\begin{lemma}[Reflexive Hartogs extension]\label{lem:reflexive-Hartogs}
Suppose that $X$ is normal and
$\operatorname{codim}_X Z\geq2$.  For a vector bundle $F$
on $U$, the sheaf $E=j_*F$ is coherent and reflexive, and
the canonical restriction morphism is an isomorphism
\begin{equation}\label{eq:reflexive-analytic-extension}
 E^{\an}\xrightarrow{\sim}j^{\an}_*F^{\an}.
\end{equation}
Equivalently, for every analytic open subset $W\subset X^{\an}$,
restriction induces an isomorphism
$$
 H^0(W,E^{\an})\xrightarrow{\sim}
 H^0(W\cap U^{\an},F^{\an}).
$$
\end{lemma}

\begin{proof}
We first construct the algebraic extension and explain why it
is canonical.  Choose a coherent extension $G$ of $F$
to $X$, and form its double dual $G^{\vee\vee}$.
On a normal variety this hull is reflexive, and its restriction
to $U$ is $F^{\vee\vee}=F$.  A reflexive sheaf on
a normal variety is recovered from its restriction away from
a subset of codimension at least two
\cite[Proposition~1.6]{Hartshorne}; see also
\cite[Tag~0EBJ]{StacksProject} for the normal case used here.
Applying this fact to
the hull gives
$$
 G^{\vee\vee}\simeq j_*\bigl(G^{\vee\vee}|_U\bigr)
 \simeq j_*F.
$$
Thus $E=j_*F$ is coherent and reflexive, independently of
the initial choice of $G$.

We next verify that this same extension has the required
Hartogs property on analytic open subsets.  The analytic
local ring is flat over the algebraic local ring
\cite[Section~6, Proposition~3 and its corollary]{SerreGAGA}.
Duals of finitely presented modules commute with flat base
change \cite[Proposition~21]{SerreGAGA}, so the isomorphism
$E\simeq E^{\vee\vee}$ analytifies to
$E^{\an}\simeq(E^{\an})^{\vee\vee}$.  In particular,
$E^{\an}$ is reflexive.  The space $X^{\an}$ is
normal, and $Z^{\an}$ has codimension at least two.

To see explicitly how sections extend, work locally on
$X^{\an}$ and dualise a finite presentation of
$(E^{\an})^\vee$.  Reflexivity expresses $E^{\an}$
as the kernel of a map between finite free sheaves:
$$
 0\longrightarrow E^{\an}\longrightarrow
 \cO_{X^{\an}}^{\oplus a}
 \xrightarrow{\Phi}\cO_{X^{\an}}^{\oplus b}.
$$
A section of $E^{\an}$ away from $Z^{\an}$ is
therefore a tuple of holomorphic functions satisfying
$\Phi(s)=0$.  The Hartogs theorem on normal complex
spaces extends each function uniquely across $Z^{\an}$.
The relation $\Phi(s)=0$ continues to hold by density, so
the extended tuple is a section of $E^{\an}$.  Uniqueness
makes these local extensions agree on overlaps.  They hence
glue on every analytic open subset, proving
\eqref{eq:reflexive-analytic-extension} and its stated
interpretation in terms of sections.
\end{proof}

\subsection{Algebraisation of splittings}

Projectivity was not needed for \cref{lem:reflexive-Hartogs}.
We now use it to apply proper GAGA to the coherent reflexive
extension.  This compares global sections on $U$ with
holomorphic sections on $U^{\an}$, and thereby detects
whether a vector-bundle extension splits.

\begin{theorem}[Codimension-two GAGA injection]
\label{thm:GAGA}
Let $U$ be a dense open subset of a normal projective complex
variety $X$, with $\operatorname{codim}_X(X\setminus U)\geq2$.
For every algebraic vector bundle $F$ on $U$, the natural map
$$
 H^1(U,F)\longrightarrow H^1(U^{\an},F^{\an})
$$
is injective.
\end{theorem}

\begin{proof}
Let $\alpha\in H^1(U,F)$ have zero analytic image.
The identification
$H^1(U,F)=\Ext^1_{\cO_U}(\cO_U,F)$ represents it by an
algebraic extension
$$
 0\longrightarrow F\longrightarrow B
 \xrightarrow{p}\cO_U\longrightarrow0.
$$
Its quotient is locally free, so the sequence splits locally
and $B$ is a vector bundle.  Analytification preserves this exact sequence, and the
resulting extension represents the image
$\alpha^{\an}$ under the natural comparison map. 
The vanishing of
$\alpha^{\an}$ gives a holomorphic section
$s^{\an}\in H^0(U^{\an},B^{\an})$ satisfying
$p^{\an}(s^{\an})=1$.

Apply \cref{lem:reflexive-Hartogs} to $B$.  The sheaf
$\overline B=j_*B$ is coherent reflexive, and
$\overline B^{\an}\simeq j^{\an}_*B^{\an}$.
Proper GAGA \cite{SerreGAGA} gives the natural isomorphisms
$$
\begin{aligned}
 H^0(U,B)&=H^0(X,\overline B)\\
 &\xrightarrow{\sim}H^0(X^{\an},\overline B^{\an})
 =H^0(U^{\an},B^{\an}).
\end{aligned}
$$
Thus $s^{\an}$ is the analytification of an algebraic
section $s\in H^0(U,B)$.  The same comparison for
$\cO_U$, and naturality with respect to $p$, give
the commutative square
$$
\begin{tikzcd}[column sep=large]
 H^0(U,B)\arrow[r,"\sim"]\arrow[d,"p"']&
 H^0(U^{\an},B^{\an})\arrow[d,"p^{\an}"]\\
 H^0(U,\cO_U)\arrow[r,"\sim"]&
 H^0(U^{\an},\cO_{U^{\an}}).
\end{tikzcd}
$$
Since the bottom arrow sends $1$ to $1$ and is
injective, $p(s)=1$.  The map $a\mapsto as$ therefore
splits the algebraic extension, so $\alpha=0$.
\end{proof}

For the application to obstruction classes, it is useful to
express the theorem directly in terms of extensions.  The
following consequence applies to any pair of vector bundles
on $U$, irrespective of their geometric origin.

\begin{corollary}[Algebraic splittings detected analytically]
\label{cor:analytic-splitting}
Under the hypotheses of \cref{thm:GAGA}, an algebraic extension
of vector bundles
$$
 0\longrightarrow F_1\longrightarrow B\longrightarrow F_2
 \longrightarrow0
$$
on $U$ splits algebraically if it splits holomorphically.
\end{corollary}

\begin{proof}
The extension class lies in $H^1(U,\Hom(F_2,F_1))$.
A holomorphic splitting makes its analytic image zero, so
\cref{thm:GAGA}, applied to $\Hom(F_2,F_1)$, makes the
algebraic class zero as well.
\end{proof}

The comparison theorem is now available for any normal
projective compactification with a codimension-two
complement.  Section~\ref{sec:compactification} verifies
this geometric condition for spin-level covers and
transfers the resulting injectivity to the spin stack,
where it will apply to the primary obstruction.

\section{Spin-level compactifications and holomorphic nonvanishing}
\label{sec:compactification}

The algebraic primary class is nonzero by
Theorem~\ref{thm:algebraic}, and Theorem~\ref{thm:GAGA}
shows how to preserve a first cohomology class under
analytification when a suitable compactification is available.
We now verify that geometric hypothesis for spin moduli.
We first construct normal projective compactifications of
its level covers with complement of codimension at least
two.  We then descend the resulting comparison through
the finite level cover and the central spin gerbe.
Applying it to the primary obstruction will prove that
both components of smooth unpunctured supermoduli are
non-projected, algebraically and holomorphically, in every
genus at least three.

The compactification argument takes place in ordinary
algebraic geometry.  We begin with coarse moduli spaces
and return to stacks only when retaining the spin data
and descending cohomology.  Write $\mathcal M_g$ for
the stack of smooth unpunctured curves, $M_g$ for its
coarse moduli space, and $\overline M_g$ for the
Deligne--Mumford compactification by stable curves
\cite{DeligneMumford}.  We write $A_g$ for the coarse
moduli space of principally polarised abelian varieties
of dimension $g$.  Throughout this section, $g\geq3$.

\subsection{The Torelli morphism and its Satake boundary}

The usual stable-curve compactification adds divisors,
whereas Theorem~\ref{thm:GAGA} requires a complement of
codimension at least two.  We obtain the smaller boundary
by passing to the Satake image of the Torelli morphism:
there the attaching points of a stable curve are forgotten.
Lemma~\ref{lem:coarse-Torelli} identifies the open locus
corresponding to smooth curves and proves finiteness over
that locus.  Lemma~\ref{lem:satake-boundary} then estimates
the dimension of the omitted locus.  These two facts will
allow the next subsection to recover the spin-level spaces
by finite normalisation.

For a smooth curve $C$, let $J(C)$ denote its
Jacobian with its canonical principal polarisation.  The
Torelli morphism is
$$
 \tau:M_g\longrightarrow A_g,\qquad [C]\longmapsto[J(C)].
$$
A principally polarised abelian variety is
\emph{indecomposable} if it is not a product of two
positive-dimensional principally polarised abelian
varieties with the product polarisation.  Denote this
open locus by $A_g^{\rm ind}$.  Smooth Jacobians are
indecomposable, and we write $J_g\subset A_g^{\rm ind}$
for the reduced image of $\tau$.  When lower genera
occur below, $J_h$ has the same meaning for $h\geq2$,
and $J_1=A_1$ is the moduli space of elliptic curves.

The Satake--Baily--Borel compactification is a projective
variety with a stratification
$$
 A_g^{\rm Sat}=\coprod_{h=0}^g A_h,
$$
where $A_0$ is a point and $A_g$ is the open
stratum \cite{BailyBorel}.  Thus $A_g^{\rm ind}$ is
also open in $A_g^{\rm Sat}$.  We write
$\overline J_g^{\rm Sat}$ for the reduced closure
of $J_g$ in this compactification.

The Torelli morphism extends to a proper morphism
$$
 \overline\tau:\overline M_g\longrightarrow A_g^{\rm Sat}.
$$
It can be obtained by composing the extension to the
second Voronoi compactification with its morphism to
Satake; see \cite{AlexeevBrunyate} for this construction
and its context.  Its effect on a stable curve is
particularly simple.  Let $C_v$ be the connected
components of the normalisation of a stable curve $C$,
and let $\Gamma_C$ be its dual graph, with one vertex
for each component and one edge for each node.  The genus
formula is
$g=\sum_v g(C_v)+b_1(\Gamma_C)$, where
$b_1(\Gamma_C)$ is the first Betti number of the graph.
The abelian part of the generalised Jacobian is the
product of the Jacobians of the components, while its
torus has rank $b_1(\Gamma_C)$.  Satake retains only
the abelian part, so
\begin{equation}\label{eq:stable-Satake-image}
 \overline\tau(C)=\prod_{g(C_v)>0}J(C_v)
 \ \in A_{g-b_1(\Gamma_C)}\subset A_g^{\rm Sat},
\end{equation}
with the product principal polarisation.  In particular,
the attaching points do not enter this image.

\begin{lemma}[The smooth Jacobian locus]
\label{lem:coarse-Torelli}
One has $\overline\tau^{-1}(A_g^{\rm ind})=M_g$.
The Torelli morphism $\tau:M_g\to J_g$ is finite and
surjective.  Moreover, $J_g$ is closed in $A_g^{\rm ind}$
and open in its reduced Satake closure $\overline J_g^{\rm Sat}$.
\end{lemma}

\begin{proof}
The principally polarised Jacobian of a smooth curve is
indecomposable.  Conversely, suppose that a stable curve
has image in $A_g^{\rm ind}$.  Its abelian part must
have dimension $g$, so
\eqref{eq:stable-Satake-image} forces
$b_1(\Gamma_C)=0$.  Thus $\Gamma_C$ is a tree
and the curve is of compact type.  Indecomposability of
the product then forces exactly one component to have
positive genus.

There can be no rational components.  Indeed, a tree
with at least two vertices has at least two leaves.
If only one vertex has positive genus, at least one
leaf is rational.  That component would have only one
attaching point and no markings, contradicting stability.
The graph therefore has a single vertex and no edges,
so the curve is smooth.

This identifies the inverse image of the indecomposable
open locus and gives the Cartesian square
$$
\begin{tikzcd}[column sep=large]
 M_g\arrow[r,hook]\arrow[d,"\tau"']&
 \overline M_g\arrow[d,"\overline\tau"]\\
 A_g^{\rm ind}\arrow[r,hook]&A_g^{\rm Sat}.
\end{tikzcd}
$$
Consequently $\tau:M_g\to A_g^{\rm ind}$ is proper,
being a base change of $\overline\tau$.  The classical
Torelli theorem makes it injective on geometric points
and hence quasi-finite.  A proper quasi-finite morphism
is finite.  Its reduced image $J_g$ is therefore
closed in $A_g^{\rm ind}$, and the induced morphism
to that image is finite and surjective.  Finally,
closedness in this open subset gives
$$
 \overline J_g^{\rm Sat}\cap A_g^{\rm ind}=J_g,
$$
which proves that $J_g$ is open in its Satake closure.
\end{proof}

\begin{lemma}[Codimension of the Satake--Jacobian boundary]
\label{lem:satake-boundary}
The closed subset
$B_g=\overline J_g^{\rm Sat}\setminus J_g$ satisfies
$$
 \dim\overline J_g^{\rm Sat}=3g-3,
 \qquad \dim B_g\leq3g-5.
$$
\end{lemma}

\begin{proof}
Properness of $\overline\tau$ and density of $M_g$
in $\overline M_g$ give
$\overline\tau(\overline M_g)=\overline J_g^{\rm Sat}$.
Lemma~\ref{lem:coarse-Torelli} identifies $B_g$ with
the image of the stable boundary.  The same lemma gives
$\dim J_g=\dim M_g=3g-3$, proving the first assertion.

To estimate the boundary image, write $\Delta_0$ for
the non-separating divisor and $\Delta_i$,
$1\leq i\leq\lfloor g/2\rfloor$, for the separating
divisors.  A general curve in $\Delta_0$ is obtained
by identifying two points on a smooth curve of genus
$g-1$.  A general curve in $\Delta_i$ joins
smooth curves of genera $i$ and $g-i$ at one
point on each.  Formula~\eqref{eq:stable-Satake-image}
therefore gives the following images, up to finite
identification of equal factors:
$$
\begin{array}{c|c|c}
 \text{divisor}&\text{general Satake image}&\text{dimension}\\\hline
 \Delta_0&J_{g-1}&3g-6\\
 \Delta_1&J_1\times J_{g-1}&3g-5\\
 \Delta_i,\quad 2\leq i\leq\lfloor g/2\rfloor
     &J_i\times J_{g-i}&3g-6.
\end{array}
$$
The table follows from $\dim J_1=1$ and $\dim J_h=3h-3$ for $h\geq2$, 
since the Satake image depends only on the Jacobians of the normalisation components.
Interchanging equal factors is a finite quotient and
does not change the dimension.

Each boundary divisor is complete, so its image is
the closure of the image of its dense general stratum.
It consequently has the dimension listed in the table.
Every deeper boundary stratum lies in a boundary
divisor.  Taking their union gives
$\dim B_g\leq3g-5$, as required.
\end{proof}

This estimate explains the lower bound $g\geq3$
in the compactification argument.  In genus two the
separating product locus $J_1\times J_1$ has dimension
two inside the three-dimensional Jacobian closure;
its codimension is one, so it does not satisfy the
hypothesis of Theorem~\ref{thm:GAGA}.

\subsection{Spin-level covers and normalisation}

The preceding lemmas provide a projective variety
$\overline J_g^{\rm Sat}$ containing the smooth Jacobian
locus $J_g$, with complement of codimension at least two.
To apply the comparison theorem to spin moduli, we need
compactifications with the same boundary property for
schemes carrying spin and level structures.

We first remove the scalar automorphisms of the spin line
and then add a full level structure to eliminate the remaining
automorphisms. This produces the smooth schemes defined below.
Their finite morphisms to $J_g$ allow us to normalise
$\overline J_g^{\rm Sat}$ in the function field of each
connected component. The essential point of
Proposition~\ref{prop:small-compactification} is that this
normalisation recovers the entire component over $J_g$;
its complement therefore lies over the small boundary
estimated in \cref{lem:satake-boundary}.

Recall from Section~\ref{sec:transverse} that
$\mathcal S_g^\eps=(\M_g^\eps)_{\red}$ is the
spin stack of parity $\eps\in\{+,-\}$.  An object
$(C,L,L^{\otimes2}\simeq K_C)$ has a canonical
central automorphism
$$
 \kappa=(\Id_C,-\Id_L).
$$
Together with the identity, it forms the scalar subgroup
$\boldsymbol\mu_2$.  Rigidifying means removing this
central subgroup from the stabilisers.  Write
$$
 r:\mathcal S_g^\eps\longrightarrow\mathcal R_g^\eps
$$
for the resulting $\boldsymbol\mu_2$-gerbe; the
rigidification construction is recalled in \cite[Appendix~A]{AOV}.
The map $\mathcal R_g^\eps\to\mathcal M_g$
forgets the theta characteristic and is representable,
finite and \'etale.  Indeed, \'etale-locally on the base
of a smooth family, choose one theta characteristic.
All other theta characteristics form a torsor under
the two-torsion of the relative Picard scheme.
Multiplication by two is finite \'etale in characteristic
zero, and fixing parity selects an open and closed
part.  Rigidification removes the scalar automorphisms
of each square root, leaving precisely this finite
\'etale space of choices.

Fix an integer $n\geq3$ and a primitive $n$-th
root of unity.  A full symplectic level structure on
$C$ is a symplectic identification
$$
 \alpha:(\mathbb Z/n\mathbb Z)^{2g}
 \xrightarrow{\sim}J(C)[n],
$$
where the Weil pairing is identified with the standard
pairing using the chosen root of unity.  The moduli
space $M_g[n]$ of curves with this structure is a
smooth quasi-projective scheme.  Its morphism to
$\mathcal M_g$ is finite \'etale and surjective:
the level structures form a finite \'etale set of
choices, and $n\geq3$ eliminates automorphisms
preserving a level structure
\cite[Section~2.3.3]{PikaartDeJong}.
Define the scheme of rigidified spin curves with level by
\begin{equation}\label{eq:uniform-level}
 U_{g,n}^\eps=\mathcal R_g^\eps\times_{\mathcal M_g}M_g[n].
\end{equation}
The projection $q$ to $\mathcal R_g^\eps$ occurs
in the Cartesian square
$$
\begin{tikzcd}[column sep=large]
 U_{g,n}^\eps\arrow[r]\arrow[d,"q"']&M_g[n]\arrow[d]\\
 \mathcal R_g^\eps\arrow[r]&\mathcal M_g.
\end{tikzcd}
$$
All four morphisms are finite \'etale, and $q$ is
surjective.  In particular, $U_{g,n}^\eps$ is a
smooth quasi-projective scheme of dimension $3g-3$.

We will also need the coarse forgetful morphism
$M_g[n]\to M_g$.  Changing the symplectic marking
gives an action of
$\operatorname{Sp}_{2g}(\mathbb Z/n\mathbb Z)$;
the quotient stack is $\mathcal M_g$, and its coarse
quotient is $M_g$.  Hence the coarse forgetful map
is finite, although it need not be \'etale.

Combining the finite morphisms just constructed with
\cref{lem:coarse-Torelli}, we obtain a finite composite
$$
 U_{g,n}^\eps\longrightarrow M_g[n]\longrightarrow M_g
 \xrightarrow{\tau}J_g.
$$
We will compactify its source, one connected component at
a time. Write
$$
 U_{g,n}^\eps=\coprod_i U_{g,n,i}^\eps
$$
for its finite decomposition into connected components.
Since a smooth scheme is regular, its irreducible components
are disjoint and open. Each connected component
$U_{g,n,i}^\eps$ is therefore integral and normal.
The next proposition constructs its compactification by
normalising $\overline J_g^{\rm Sat}$ in
$\C(U_{g,n,i}^\eps)$, and identifies the component with
the full inverse image of $J_g$.

\begin{proposition}[Small spin-level compactification]
\label{prop:small-compactification}
Fix $g\geq3$, $n\geq3$, and
$\eps\in\{+,-\}$.  Every connected component
$U_{g,n,i}^\eps$ of the scheme
\eqref{eq:uniform-level} admits an open immersion
into a normal projective variety $X_{g,n,i}^\eps$
such that
$$
 \operatorname{codim}_{X_{g,n,i}^\eps}
 (X_{g,n,i}^\eps\setminus U_{g,n,i}^\eps)\geq2.
$$
\end{proposition}

\begin{proof}
Fix a connected component $U_i=U_{g,n,i}^\eps$.
The composite
$$
 U_i\longrightarrow M_g[n]\longrightarrow M_g
 \xrightarrow{\tau}J_g
$$
is finite: the first two maps are finite by the constructions
above, and the last is finite by \cref{lem:coarse-Torelli}.
It is also surjective.  Indeed, $U_i$ has dimension
$3g-3$, so its nonempty closed image under the finite map
to $M_g$ has that same dimension.  Irreducibility of
$M_g$ forces this image to be all of $M_g$, and
$\tau$ is surjective onto $J_g$.

The reduced varieties $J_g$ and $\overline J_g^{\rm Sat}$
are integral, being the image and closure of the irreducible
variety $M_g$.  The finite surjective map just obtained
therefore defines a finite function-field extension
$$
 \C(J_g)\subset K_i:=\C(U_i).
$$
Since $J_g$ is dense and open in
$\overline J_g^{\rm Sat}$, these two varieties have the
same function field.
Set
$$
 X_i=\operatorname{Norm}_{K_i}(\overline J_g^{\rm Sat}),
 \qquad \nu:X_i\longrightarrow\overline J_g^{\rm Sat}.
$$
Here normalisation means taking the integral closure of each
affine coordinate ring in $K_i$.  For varieties over $\C$,
normalisation in a finite function-field extension is finite;
see \cite[Tag~0BAK]{StacksProject}.  Thus $\nu$ is finite
and surjective, and $X_i$ is normal and projective.

We identify $\nu^{-1}(J_g)$ with $U_i$.
For a nonempty affine open $\Spec A\subset J_g$, its
inverse image in $U_i$ is affine, say $\Spec B$.
The ring $B$ is a normal domain, finite over $A$,
with fraction field $K_i$.  It is exactly the integral
closure of $A$ in $K_i$: every element of $B$
is integral over $A$, and an element of $K_i$ integral
over $A$ is also integral over $B$, hence belongs
to $B$.  Integral closure commutes with localisation,
so these identifications glue and give the Cartesian square
$$
\begin{tikzcd}[column sep=large]
 U_i\arrow[r,hook]\arrow[d]&X_i\arrow[d,"\nu"]\\
 J_g\arrow[r,hook]&\overline J_g^{\rm Sat}.
\end{tikzcd}
$$
In particular, the complement is exactly
$X_i\setminus U_i=\nu^{-1}(B_g)$, where
$B_g=\overline J_g^{\rm Sat}\setminus J_g$.

Since $\nu$ is finite and surjective, the dimensions of
$X_i$ and $\nu^{-1}(B_g)$ equal those of
$\overline J_g^{\rm Sat}$ and $B_g$, respectively.
Lemma~\ref{lem:satake-boundary} gives
$$
 \dim X_i=3g-3,\qquad \dim(X_i\setminus U_i)\leq3g-5.
$$
The required codimension bound follows, with
$X_{g,n,i}^\eps=X_i$.
\end{proof}

\subsection{Descent from the level cover to the spin stack}

Proposition~\ref{prop:small-compactification} allows us to apply
Theorem~\ref{thm:GAGA} to every connected component of
$U_{g,n}^\eps$.  We transfer this comparison to the spin
stack using the central gerbe and the finite \'etale level
cover.  The resulting statement applies to every algebraic
vector bundle on $\mathcal S_g^\eps$.

Let $F$ be an algebraic vector bundle on
$\mathcal S_g^\eps$. The stabiliser of a spin curve
$(C,L,L^{\otimes2}\simeq K_C)$ acts on the corresponding
fibre of $F$. In particular, the central scalar
automorphism
$\kappa=(\Id_C,-\Id_L)$ induces a global involution
$\kappa_F$ of $F$. Its two eigensubbundles give the
canonical decomposition
$$
 F=F_{\rm inv}\oplus F_{\rm anti},
$$
where
$$
 \begin{aligned}
 F_{\rm inv}
   &=\operatorname{im}\!\left(\frac{\Id_F+\kappa_F}{2}\right),
 &\qquad \kappa_F|_{F_{\rm inv}}&=\Id,\\
 F_{\rm anti}
   &=\operatorname{im}\!\left(\frac{\Id_F-\kappa_F}{2}\right),
 &\qquad \kappa_F|_{F_{\rm anti}}&=-\Id.
 \end{aligned}
$$
These summands describe the scalar action on $F$;
they are independent of the fixed spin parity $\eps$.

The next lemma shows that $F_{\rm inv}$ descends to the
rigidification, whereas $F_{\rm anti}$ has zero cohomology.
We can therefore transfer the comparison problem to the
rigidified spin stack and then use finite trace to detect
its cohomology classes on the level cover.

\begin{lemma}[The central gerbe and finite trace]
\label{lem:gerbe-trace}
Let
$$
 r:\mathcal S_g^\eps\longrightarrow\mathcal R_g^\eps,
 \qquad q:U_{g,n}^\eps\longrightarrow\mathcal R_g^\eps
$$
be the scalar $\boldsymbol\mu_2$-gerbe and the
finite \'etale surjective level cover defined above.
Then the following hold.
\begin{enumerate}[label=\textup{(\roman*)}]
\item For every vector bundle $F$ on $\mathcal S_g^\eps$,
the summand $F_{\rm inv}$ descends to a vector bundle
$F_{\rm rig}$ on $\mathcal R_g^\eps$, with
$r^*F_{\rm rig}\simeq F_{\rm inv}$. Pullback
followed by the summand inclusion induces an isomorphism
$$
 H^1(\mathcal R_g^\eps,F_{\rm rig})
 \xrightarrow{\sim}H^1(\mathcal S_g^\eps,F).
$$
The anti-invariant summand satisfies
$$
 H^k(\mathcal S_g^\eps,F_{\rm anti})=0
 \qquad(k\geq0).
$$
\item For every vector bundle $G$ on
$\mathcal R_g^\eps$, pullback is injective:
$$
 q^*:H^1(\mathcal R_g^\eps,G)
 \lhook\joinrel\longrightarrow H^1(U_{g,n}^\eps,q^*G).
$$
\end{enumerate}
{Both assertions hold analytically as well.  The
decomposition and the comparison maps commute with
analytification.}
\end{lemma}

\begin{proof}
The central involution gives the two global idempotents
displayed above.  Their images are direct summands of a locally
free sheaf, hence are vector bundles.  These constructions are
canonical and commute with pullback.

Work \'etale-locally on $\mathcal R_g^\eps$, where the
gerbe is $B\boldsymbol\mu_2$ over the base.  A vector
bundle there is a bundle with a $\boldsymbol\mu_2$-action;
its invariant summand is a bundle with trivial action and thus
comes from the base.  This description is compatible with
descent and produces $F_{\rm rig}$, with
$r^*F_{\rm rig}\simeq F_{\rm inv}$.

On an \'etale base chart $T$ over which the gerbe is
$T\times B\boldsymbol\mu_2$, let $E$ be the
$\boldsymbol\mu_2$-equivariant bundle on $T$
corresponding to the pullback of $F$ to this gerbe.
Then
$$
 (r_*F)|_T \simeq E^{\boldsymbol\mu_2},
$$
where $E^{\boldsymbol\mu_2}\subset E$ is the subsheaf
of sections fixed by the scalar involution.
Thus the direct image functor is locally identified
with the invariants functor
$E\mapsto E^{\boldsymbol\mu_2}$.
This functor is exact over $\C$, by the averaging operator
$\frac12(\Id+\kappa)$.  Its derived functors therefore
vanish in positive degrees; this is the linearly reductive
case of tame descent \cite{AOV}.  Since $F_{\rm anti}$
has no invariant sections, the resulting identities are
$$
\begin{aligned}
 r_*F_{\rm inv}
   &\simeq F_{\rm rig},\\
 R^k r_*F_{\rm inv}
   &=0 \qquad (k>0),\\
 R^k r_*F_{\rm anti}
   &=0 \qquad (k\geq0).
\end{aligned}
$$
These are local identities on the base and hence hold globally.
The Leray spectral sequence now gives both the acyclicity of
$F_{\rm anti}$ and the isomorphism
$$
 H^1(\mathcal R_g^\eps,F_{\rm rig})
 \simeq H^1(\mathcal S_g^\eps,F_{\rm inv})
 \simeq H^1(\mathcal S_g^\eps,F).
$$

The same proof works analytically on local charts neutralising
the gerbe.  There the relative cohomology is computed by the
cohomology of the finite group $\mu_2$; averaging makes
invariants exact and kills its positive-degree cohomology.
Thus the same direct-image identities and Leray argument apply.
The idempotents and the descent identifications commute with
analytification, giving the stated compatibility.  This proves
(i) in both categories.

For the level cover, let $d$ be the positive,
locally constant degree of $q$.  The unit and
the trace, together with the projection formula,
give morphisms
$$
 G\longrightarrow q_*q^*G
   \simeq G\otimes q_*\cO_{U_{g,n}^\eps}
 \xrightarrow{\Id\otimes\operatorname{Tr}_q}G
$$
whose composite is $d\,\Id_G$.  Since $q$
is finite, its higher direct images on
quasi-coherent sheaves vanish.  Passing to first
cohomology therefore gives
\begin{equation}\label{eq:finite-trace-injection}
 \operatorname{Tr}_q\circ q^*=d\,\Id.
\end{equation}
Division by $d$, on each connected component
of the base, supplies a left inverse to pullback.
This proves (ii).  Analytically, a finite \'etale
map is locally a finite disjoint union of copies
of the base, and trace is summation over these
copies.  The same identity proves injectivity.
The unit, trace and gerbe averaging commute with
analytification, and so do the induced cohomology
maps. 
\end{proof}

\begin{proposition}[Comparison on the spin stack]
\label{prop:spin-comparison}
For $g\geq3$, $\eps\in\{+,-\}$, and every algebraic
vector bundle $F$ on $\mathcal S_g^\eps$, the natural
map
$$
 H^1(\mathcal S_g^\eps,F)
 \longrightarrow H^1((\mathcal S_g^\eps)^{\an},F^{\an})
$$
is injective.  Consequently, an algebraic extension of vector
bundles on $\mathcal S_g^\eps$ splits algebraically if
and only if it splits holomorphically after analytification.
\end{proposition}

\begin{proof}
Decompose $F=F_{\rm inv}\oplus F_{\rm anti}$.
Lemma~\ref{lem:gerbe-trace}(i) eliminates the anti-invariant
summand from first cohomology in both categories, compatibly
with analytification.  Its gerbe isomorphisms therefore reduce
the assertion to the descended bundle $G=F_{\rm rig}$
on $\mathcal R_g^\eps$.
For this descended bundle $G$, put
$R=\mathcal R_g^\eps$ and $U=U_{g,n}^\eps$, and
consider the finite \'etale cover $q:U\to R$.
Naturality of analytification with respect to pullback
along $q$ gives the commutative diagram
$$
\begin{tikzcd}[column sep=large]
 H^1(R,G)\arrow[r,hook,"q^*"]\arrow[d]&
 H^1(U,q^*G)\arrow[d,hook]\\
 H^1(R^{\an},G^{\an})\arrow[r,hook]&
 H^1(U^{\an},(q^*G)^{\an}).
\end{tikzcd}
$$
The horizontal maps are injective by the trace
statement in Lemma~\ref{lem:gerbe-trace}(ii).
The right comparison is injective on each connected
component of $U$, by
Proposition~\ref{prop:small-compactification} and
Theorem~\ref{thm:GAGA}.  Cohomology on a finite
disjoint union is the product of the cohomology
groups of its components, so the right comparison
is injective on all of $U$.

If a class at the upper left has zero analytic
image, its pullback at the upper right also has
zero analytic image.  Right-hand injectivity makes
that pullback zero algebraically, and top-row
injectivity makes the original class zero.
The gerbe isomorphisms transfer the conclusion
back to $\mathcal S_g^\eps$.

An extension of vector bundles $E_2$ by $E_1$ has class in
$$
 H^1(\mathcal S_g^\eps,\Hom(E_2,E_1)).
$$
A holomorphic splitting makes its analytic image zero, so the
injectivity just proved makes the algebraic extension split.
The converse follows by analytifying an algebraic splitting.
\end{proof}

\subsection{The holomorphic primary obstruction}

Theorem~\ref{thm:algebraic} gives a nonzero algebraic primary
class, and \cref{prop:spin-comparison} now applies to its
coefficient bundle. 
It remains to identify the analytic image of that class with
the holomorphic primary obstruction.  This compatibility will
complete the non-projectedness theorem for the entire smooth
moduli component, in every $g\geq3$ and both spin parities.

\begin{theorem}[Non-projectedness in every genus at least three]
\label{thm:main}
For every $g\geq3$ and $\eps\in\{+,-\}$, the primary
obstruction of the smooth unpunctured supermoduli component
$\M_g^\eps$ is nonzero both algebraically and holomorphically:
$$
 \omega_2^{\rm alg}(\M_g^\eps)\ne0,
 \qquad
 \omega_2^{\rm hol}((\M_g^\eps)^{\an})\ne0.
$$
Consequently both components are non-projected and non-split in
both categories.
\end{theorem}

\begin{proof}
By Section~\ref{sec:transverse}, the algebraic primary class
belongs to $H^1(\mathcal S_g^\eps,K_{2,g}^\eps)$, where
$$
 K_{2,g}^\eps=
 \Hom(\bigwedge^2T_-\M_g^\eps,T_+\M_g^\eps).
$$
This is an algebraic vector bundle, so
\cref{prop:spin-comparison} gives an injection on its first
cohomology under analytification.

We next identify this analytic image with the holomorphic
primary obstruction. On smooth charts, analytification
identifies algebraic differentials with holomorphic
differentials and preserves the odd ideal and its powers.
The definitions in Section~\ref{sec:transverse} therefore
give canonical identifications
$$
 \bigl(T_\pm\M_g^\eps\bigr)^{\an}
 \simeq T_\pm\bigl((\M_g^\eps)^{\an}\bigr).
$$
Since analytification also commutes with duals, tensor
products and exterior powers of vector bundles,
$(K_{2,g}^\eps)^{\an}$ is precisely the coefficient
bundle of the holomorphic primary obstruction.

Choose an \'etale atlas by smooth superschemes and,
after refinement, local split coordinates on its charts.
As recalled in Section~\ref{sec:transverse}, the quadratic
terms in the odd variables in the even coordinate changes
define a cocycle representing
$\omega_2^{\rm alg}(\M_g^\eps)$.
Analytification turns these coordinate changes into
holomorphic ones and preserves both their linear
transition maps and their quadratic coefficients.
The resulting cocycle is therefore the one defining
the holomorphic primary obstruction.

The cocycle identities on the atlas groupoid, as well as
the coboundary identities arising from changes of local
splittings, are preserved by analytification. Thus this
identification is independent of the chosen charts and
descends to the spin stack. Under the natural comparison
map, it gives
$$
 \omega_2^{\rm alg}(\M_g^\eps)
 \longmapsto
 \omega_2^{\rm hol}\bigl((\M_g^\eps)^{\an}\bigr).
$$

Finally, Theorem~\ref{thm:algebraic} says that
the class on the left is nonzero.  Injectivity
of Proposition~\ref{prop:spin-comparison} makes
its image on the right nonzero as well.
Projectedness would force the relevant primary
class to vanish, and a splitting would supply
a projection.  Both are therefore excluded in
both categories, proving the theorem.
\end{proof}



\end{document}